\documentclass[reqno]{amsart}
\usepackage[T1]{fontenc} 
\usepackage[english]{babel} 
\usepackage{amsmath,amsfonts, amsthm,amssymb,mathtools} 
\usepackage{mathrsfs,mathabx,bbold}    
  \usepackage{amsrefs}
\usepackage{hyperref} 
\usepackage{graphicx} 
\usepackage{framed}   
\usepackage{ifthen}   
\usepackage{lastpage} 
\usepackage{mathdots} 
\usepackage{enumerate}
\usepackage{tikz}     
\usepackage{verbatim} 
\usepackage{stmaryrd,mathrsfs,calrsfs,dutchcal,mathtools}
\usepackage{parskip}
\usepackage{abstract}
\usepackage{parskip}
\usepackage{enumitem}

\newcommand{\CC}{\mathbb{C}} 
\newcommand{\RR}{\mathbb{R}} 
\newcommand{\NN}{\mathbb{N}} 
\newcommand{\MM}{\mathbb{M}} 

\newcommand{\pp}{\mathcal{P}}

\newcommand{\cc}{\mathcal{C}}

\newcommand{\ZZ}{\mathbb{Z}} 
\newcommand{\bq}{\begin{question}}
	\newcommand{\eq}{\end{question}}
\def\ba#1\ea{\begin{align*}#1\end{align*}}

\newcommand{\floor}[1]{\left\lfloor #1 \right\rfloor}

\newcommand{\powset}[2][]{\ifthenelse{\equal{#2}{}}{\mathcal{P}\left(#1\right)}{\mathcal{P}_{#1}\left(#2\right)}}

\newcommand{\crk}{\cc^{\rv+\kappa\ef_1}(W,\dv)}
\newcommand{\lps}{L^p_{\sv+\kappa\ef_1}(W,\dv)}

\let\sumnonlimits\sum
\let\prodnonlimits\prod
\let\cupnonlimits\bigcup
\let\capnonlimits\bigcap
\renewcommand{\sum}{\sumnonlimits\limits}
\renewcommand{\prod}{\prodnonlimits\limits}
\renewcommand{\bigcup}{\cupnonlimits\limits}
\renewcommand{\bigcap}{\capnonlimits\limits}

\newtheoremstyle{plainsl}
{8pt plus 2pt minus 4pt}
{8pt plus 2pt minus 4pt}
{\slshape}
{0pt}
{\bfseries}
{.}
{5pt plus 1pt minus 1pt}
{}

\theoremstyle{plainsl}
\newtheorem{theorem}{Theorem}[section]
\newtheorem{prop}[theorem]{Proposition}
\newtheorem{prop*}{Proposition}

\newtheorem{lemma}[theorem]{Lemma}
\newtheorem{corollary}[theorem]{Corollary}

\theoremstyle{definition}
\newtheorem{definition}[theorem]{Definition}
\newtheorem{example}[theorem]{Example}
\newtheorem{notation}[theorem]{Notation}

\newtheorem{question}[theorem]{Question}

\theoremstyle{remark}
\newtheorem{remark}[theorem]{Remark}

\newcommand{\db}{d\hspace*{-0.08em}\bar{}\hspace*{0.1em}}
\newcommand{\dv}{\vec{d\hspace*{-0.08em}\bar{}\hspace*{0.1em}}}
\newcommand{\sv}{\vec{s}}
\newcommand{\rv}{\vec{r}}
\newcommand{\ef}{\mathfrak{e}}
\newcommand{\Wf}{\mathfrak{W}}
\def\ba#1\ea{\begin{align*}#1\end{align*}}

\begin{document}

\title{A regularity theorem for fully nonlinear maximally subelliptic PDE}


\author{Gautam Neelakantan Memana}
\address{Van Vleck Hall, 213, 480 Lincoln Dr, Madison, WI 53706}
\curraddr{}
\email{neelakantanm@wisc.edu}
\thanks{}

\date{}

\dedicatory{}
	\maketitle
 \begin{abstract} 
 We prove a sharp interior regularity theorem for fully nonlinear maximally subelliptic PDE in non-isotropic Sobolev spaces adapted to maximally subelliptic operators. This result complements the regularity theorem proven by Street for adapted Zygmund-H\"older spaces. This also recovers the classical regularity theorem for fully nonlinear elliptic differential operators for classical Sobolev spaces. We also obtain a sharp interior regularity theorem for fully nonlinear maximally subelliptic PDE in isotropic Sobolev spaces.
\end{abstract}
 \tableofcontents
\section{Introduction}\label{Introduction}
\numberwithin{equation}{section}

Maximally subelliptic operators also known as maximally hypoelliptic operators are one of the most natural generalizations of elliptic differential operators. They often play the same role in sub-Riemannian geometry as elliptic differential operators play in Riemannian geometry; see \cites{MR1739222,MR3155912,MR1739222,MR0741039} for connections to CR geometry. One of the most important results for nonlinear elliptic differential operators is interior regularity for its solutions. In this article, we prove an interior regularity theorem for fully nonlinear maximally subelliptic partial differential equations(PDEs). At present, there are a wide variety of proofs for regularity for fully nonlinear elliptic PDEs using fixed point theorem, maximum principle and energy minimizing techniques; for instance see \cites{MR1005611,MR1351007,MR0931007}. It will be interesting to see how many of these techniques can be successfully adapted to subelliptic operators. Quasilinear maximally subelliptic PDEs have been well studied in the past; see \cites{MR1135924,MR1459590,MR1239930} for some of the earliest work and \cites{CG,MR2085543,MR2126699,MR2117205,MR2545517,MR2676172,MR2245893, MR2079757,MR2377405,MR2173373,MR2000279,MR2045844,MR1863937,MR4273846,MR4241808,MR3510691,MR2336058,MR3787355,MR3444525,MR2606775,MR2298970} for more recent work.  There has also been some work on weak solutions of second order fully nonlinear equations in \cites{MR3388873,MR4018316,MR4197073}. However, there has not been much work on general fully nonlinear maximally subelliptic partial differenial equations until \cite{BS}.

In \cite{BS}, Street developed a theory of pseudodifferential operators, singular integrals, and Besov and Triebel-Lizorkin spaces adapted to maximally subelliptic operators to prove a sharp interior regularity theorem for linear maximally subelliptic operators in the adapted Besov and Triebel-Lizorkin spaces and for fully nonlinear maximally subelliptic operators in the adapted Zygmund-H\"older spaces. Since Sobolev spaces are one of the natural function spaces to study PDEs, we prove a theorem similar to \cite[Chapter 14, Theorem 4.8]{MT}(elliptic regularity for Sobolev spaces) in the case of maximally subelliptic operators. 

This article builds upon the the work of Street \cite{BS}. We refer the reader to \cite{BS} for a detailed exposition of the background theory of maximally subelliptic operators. 

\textbf{Basic Setup}. 
Let $\mathcal{M}$ be a connected $C^{\infty}$ manifold of dimension $n$ with a smooth, strictly positive density $Vol$. 

Let $C^{\infty}(\mathcal{M}; T\mathcal{M})$ denote the smooth sections of the tangent bundle $T\mathcal{M}$. Let 
\begin{align*}
    (W, \db)= \{(W_1, \db_1),..., (W_r, \db_k)\} \subset C^{\infty}(\mathcal{M}; T\mathcal{M}) \times \NN_+
\end{align*}
be such that $W=\{W_1,..., W_r\}$ satisfies H\"ormander's bracket generating condition. We call $(W, \db)$ H\"ormander vector fields with single-parameter formal degrees on $\mathcal{M}$. We think of $W_j$ as a differential operator of order $\db_j$ even though it is only of degree $1$. We fix $\kappa\in \NN_+$ such that $\db_j$ divides $\kappa, \forall  1\leq j\leq k$. 

\begin{notation}
    Let $L\in \NN_+$. For a multi-index $\alpha=(\alpha_1, ...., \alpha_L)\in \{1,..., k\}^L$ we set $W^{\alpha}:=W_{\alpha_1}W_{\alpha_2}...W_{\alpha_L}$, $|\alpha|=L$, and  $\deg_{\db}(\alpha):= \db_{\alpha_1}+\db_{\alpha_2}+...+\db_{\alpha_L}$. We call such an $\alpha$ an \textit{ordered multi-index}. 
\end{notation}
Fix $D_1, D_2\in \NN_+$ and let 
\begin{align}
    N:= D_2\times \#\{\alpha\ \text{an ordered-multi-index}\ : \deg_{\db}(\alpha)\leq \kappa\}.
\end{align}
Given a sufficiently smooth function $u=(u_1, ..., u_{D_2}): \mathcal{M}\to \RR^{D_2}$,we define  
\begin{align*}
    \{W^{\alpha} u\}_{\deg_{\db}(\alpha)\leq \kappa} : \mathcal{M}\to \RR^{N}
\end{align*}
to be the vector whose components are given by $W^{\alpha}u_j, \deg_{\db}(\alpha)\leq \kappa$ and $1\leq j\leq D_2$. 

Let 
$F(x, \zeta): \mathcal{M}\times \RR^N\to \RR^{D_1}$. The main intent of this article is to study interior regularity for a PDE of the form 
\begin{align}\label{introduction: general partial differential equation}
    F(x, \{W^{\alpha}u(x)\}_{\deg_{\db}(\alpha)\leq \kappa})=g(x),
\end{align}
which is \textit{maximally subelliptic} of degree $\kappa$ with respect to $(W, \db)$ as in the following definition. 
\begin{definition}\label{introduction: definition of maximally subelliptic PDE}
Let $x_0\in \mathcal{M}$ and $u:\mathcal{M}\to \RR^{D_2}$. We say that a PDE given by \eqref{introduction: general partial differential equation} is \textit{maximally subelliptic} at $(x_0,u)$ of degree $\kappa$ with respect to $(W, \db)$ if there exists an open neighbourhood $U\subseteq \mathcal{M}$ with $x_0\in U$ such that the linearized operator $\mathcal{P}_{u,x_0}$ defined by  
 \begin{align}\label{Introduction: linearized operator}
     \mathcal{P}_{u, x_0}v:= d_{\zeta} F(x_0, \{W^{\alpha}u(x_0)\}_{\deg_{\db}(\alpha)\leq \kappa})\{W^{\alpha}v\}_{\deg_{\db}(\alpha)\leq \kappa}
 \end{align}
is a linear \textit{maximally subelliptic} partial differential operator of degree $\kappa$ with respect to $(W, \db)$ on $U$, i.e., for every relatively compact, open set $\Omega \Subset U$ there exists $C_{\Omega}\geq 0$ satisfying
 \begin{align}
     \sum_{j=1}^{k} \|W_{j}^{n_j}f\|_{L^2(\mathcal{M},Vol; \CC^{D_2})} \leq C_{\Omega} (\|\pp_{u,x_0} f\|_{L^2(\mathcal{M},Vol; \CC^{D_1})}+\|f\|_{L^2(\mathcal{M},Vol; \CC^{D_2})}), 
 \end{align}
for every $f\in C_{0}^{\infty}(\Omega; \CC^{D_2})$, where $n_j=\kappa/\db_j\in \NN_+$.
\end{definition}
  Also, see Appendix \ref{Appendix A} for equivalent characterizations of a linear maximally subelliptic partial differential operator. We also recommend the reader to see \cite{AMY22}, where the authors show equivalent representation theoretic characterization of linear maximally subelliptic operators by proving a conjecture of Helffer and Nourrigat \cite{MR0558795}. 
 \begin{remark}
    Since $\{W_1,...,W_k\}$ satisfies H\"ormander's condition, every PDE with smooth coefficients is of the form \eqref{introduction: general partial differential equation}. Hence, we are not making any assumptions on the form of the PDE.
 \end{remark}
 Under the above assumptions, we will prove a regularity theorem given appropriate apriori regularity for $F,g$ and a solution $u$ of \eqref{introduction: general partial differential equation}. We make these regularity assumptions in Zygmund-H\"older spaces and Sobolev spaces adapted to $(W, \db)$. We will provide an informal description of these spaces in special cases and then they will be rigorously defined in Section \ref{Function spaces} in the general case. 

 Let $m\in \NN$ is such that $\db_j$ divides $m$ for $1\leq j \leq k$. For a distribution $u$ with compact support, we say that $u\in L^{p}_{m}(W, \db)$ if 
 \begin{align*}
     W^{\alpha} u\in L^{p}(\mathcal{M}, Vol), \ \forall \deg_{\db}(\alpha)\leq m.
 \end{align*}
 We describe the adapted Zygmund-H\"older spaces only in the case when $\db_1=...=\db_k=1$. Let $\rho_{(W, 1)}$ denote the "Carnot-Carath\'eodory metric" given by the vector fields $\{W_1,..., W_k\}$. For $s\in (0,1)$, $\cc^{s}(W,1)$ is locally equivalent to the H\"older space $C^{0,s}(W,1)$ with respect to the metric $\rho_{(W, 1)}$ given by 
 \begin{align*}
     \|f\|_{C^{0,s}(W,1)}:= \underset{x\in \mathcal{M}}{\sup}\ |f(x)| + \underset{x,y\in \mathcal{M}, x\neq y}{\sup}\ \rho_{(W,1)}(x,y)^{-s} |f(x)-f(y)|.
 \end{align*}
One can also give a similar informal description for $C^{0,s}(W, \db)$ when $\db$ is not necessarily $1$, but we do not proceed it here; see \cite[Section 7.7.1]{BS} for further details. Moreover, when $s\in (0,\infty)\backslash \NN_+$, $\cc^{s}(W,1)$ is locally equivalent to the H\"older space $C^{\lfloor s\rfloor,s-\lfloor s \rfloor}(W,1)$ with the norm 
\begin{align*}
    \|f\|_{C^{\lfloor s\rfloor,s-\lfloor s \rfloor}(W,1)}:= \sum_{|\alpha|\leq \lfloor s\rfloor} \|W^{\alpha}\|_{C^{0, \lfloor s\rfloor}(W,1)}.
\end{align*}
Local equivalence of the above H\"older norms and the norms we define in Section \ref{Zygmund-H\"older Space} is non-trivial. We refer the reader to \cite[Theorem 7.7.23]{BS} for a proof. 

 With these definitions, our main regularity theorem for the fully nonlinear maximally subelliptic PDE is the following:
  \begin{theorem}\label{Introduction: Main single parameter theorem with infinitely smooth nonlinearity}
     Fix $x_0\in \mathcal{M}$. Let $1<p<\infty$, $s>r>0$. Let $u$ be a solution to the following PDE
     \begin{align*}
         F(x, \{W^{\alpha}u(x)\}_{\deg_{\db}(\alpha)\leq \kappa})=g(x),
     \end{align*}
     which is maximally subelliptic at $(x_0,u)$ of degree $\kappa$ with respect to $(W, \db)$ (see Definition \ref{introduction: definition of maximally subelliptic PDE}). Suppose that
     \begin{itemize}
         \item $F\in C^{\infty}(\mathcal{M}\times \RR^N)$ near $x_0$,
         \item $g\in L^p_{s}(W, \db)$ near $x_0$, 
         \item $u\in \cc^{r+\kappa}(W, \db)$ near $x_0$. 
     \end{itemize}
     Then $u\in L^p_{s+\kappa}(W, \db)$ near $x_0$. 
 \end{theorem}
 \begin{remark}
     Clearly, Theorem \ref{Introduction: Main single parameter theorem with infinitely smooth nonlinearity} also holds if $r\geq s$, but in this case the statement of the theorem is trivialized if $r$ is large enough. So, we think of $r$ to be a number much smaller than $s$.
 \end{remark}
 \begin{remark}
     In Theorem \ref{Introduction: Main single parameter theorem with infinitely smooth nonlinearity}, the regularity assumptions on $F$ can be made weaker; see Theorem \ref{Introduction: Main single parameter theorem}.
 \end{remark}
 \begin{remark}
    The elliptic regularity for classical Sobolev spaces \cite[Theorem 4.8]{MT} is a special case of Theorem \ref{Introduction: Main single parameter theorem with infinitely smooth nonlinearity}.  
 \end{remark}
 \begin{remark}\label{Introduction: remark about Besov and Triebel Lizorkin space theorem}
     Techniques similar to what we use in this article also produces a sharp regularity theorem for solutions in general adapted Besov and Triebel-Lizorkin spaces (see Theorem \ref{Preliminaroes: Function spaces: Theorem about Besov and Triebel Lizorkin space}), but we do not prove it here. However, using our technique, it is not possible to obtain a sharp theorem like Theorem \ref{Intrroduction: No loss of derivative: Change of Vector field theorem: smooth F and Euclidean vector fields} for general standard Besov and Triebel-Lizorkin spaces (see Theorem \ref{Preliminaries: loss of derivative: Change of Vector field theorem: smooth F and Euclidean vector fields for Triebel Lizorkin space}). 
 \end{remark}
 
  \begin{example}\label{Introduction: Example about Monge-Ampere}
      \textit{Maximally subelliptic Monge-Amp\'ere}.
       Let $(W, \db):= \{(W_1, 1),.., (W_k, 1)\}$ be H\"ormander vector fields with formal degrees $1$ on a smooth connected manifold $\mathcal{M}$. Consider the Hessian matrix $\mathcal{W}u:=(\mathcal{W}_{i,j}u(x))$, where $\mathcal{W}_{i,j}(x):=\frac{1}{2} (W_iW_j+W_jW_i)u(x)$.
      Then, the Monge-Amp\'ere operator defined as 
      \begin{align*}
        u \mapsto  \det \mathcal{W} u(x)
      \end{align*}
     is a maximally subelliptic operator of degree 2 with respect to  $(W, \db)$ at a fixed point $x_0\in \mathcal{M}$ and $u$ if $\mathcal{W} u(x_0)$ is a strictly positive definite matrix(see \cite[Section 9.4.2]{BS}). Solutions to PDEs given by these operators have also been studied previously; see \cite[Section 10]{MR2014879},\cite{MR2383892}.

      \textit{Higher order Monge-Amp\'ere}. Consider the H\"ormander vector fields with formal degrees $(W, \db):= \{(W_1, \db_1),.., (W_k, \db_k)\}$ on $\mathcal{M}$. Fix $\kappa\in \NN_+$ such that $\db_j$ divides $\kappa$ for every $j$ and let $n_j:= \kappa/\db_j$. Now, for a real valued function $u: \mathcal{M}\to \RR$ consider the $k\times k$ symmetric matrix whose $(i,j)$ entries are given by 
      \begin{align*}
          \mathcal{W_{i,j}}u(x):= \frac{1}{2} (W_j^{n_j}W_i^{n_i}u(x)+ W_i^{n_i}W_j^{n_j}u(x)).
      \end{align*}
      Then, the Higher order Monge-Amp\'ere is defined as 
      \begin{align*}
        u \mapsto  \det \mathcal{W} u(x).  
      \end{align*}
      This operator is maximally subelliptic of order $\kappa$ with respect to $(W, \db)$ at a fixed point $x_0\in \mathcal{M}$ and $u$ if the matrix $\mathcal{W}[\kappa]u(x_0)$ is a strictly positive definite matrix; see \cite[Section 9.4.2]{BS}. 
  \end{example}
  \begin{remark}
       The tame estimates we obtain in Section \ref{Tame estimate for Sobolev space} are robust and it seems likely that our technique can be carried out in special cases to study weak solutions of maximally subelliptic PDEs. For example, one could consider weak solutions of the Higher order Monge-Amp\'ere equations in Example \ref{Introduction: Example about Monge-Ampere}; see \cite{MR3032325} for such results for classical Monge-Amp\'ere. Similar ideas for Monge-Amp\'ere have also been developed in \cite{MR2351130} for Carnot groups. There has been some work along these directions for general maximally subelliptic equations(see \cites{MR3388873,MR1005611,MR2173373} for example), but very little is known at this point.
  \end{remark}
  One of the main reasons why we use $L^p$-Sobolev space is a sharp regularity theorem we get when we compare the regularity of the solutions with respect to the standard Sobolev spaces. 
     \begin{definition}\label{Preliminaries: The Geometry: definition of Gen}
    Given a $\mathcal{S}\subseteq C^{\infty}(\mathcal{M}; T\mathcal{M})\times \NN_+$ we let $Gen(\mathcal{S})$ be the smallest subset of $C^{\infty}(\mathcal{M};T\mathcal{M})\times \NN_+$ such that 
\begin{itemize}
    \item $\forall (X_1,d_1), (X_2,d_2)\in Gen(\mathcal{S})$, $([X_1, X_2], d_1+d_2)\in Gen(\mathcal{S})$.
    \item $\mathcal{S}\subseteq Gen(\mathcal{S})$.
\end{itemize}
\end{definition}
  \begin{definition}\label{Introduction: definition of standard cost factor}
      For $x\in \mathcal{M}$, set 
      \begin{align*}
          \lambda_{std}(x, (W, \db)):= \min &\{\max \{d_1,...,d_n\}: (X_1, d_1),..., (X_n, d_n)\in Gen ((W, \db)), \\
        & \text{satisfy}\ span\{X_1(x),.., X_n(x)\}=T_{x}\mathcal{M} \}.
      \end{align*}
  \end{definition}
  \begin{example}
        When $(W, \db)=((W_1,1),...,(W_k,1))$ and $\{W_1,..., W_k\}$ are H\"ormander vector fields of order $m\in \NN$ at $x$, $\lambda_{std}(x, (W, \db))=m$.
  \end{example}

     For $s,r>0$, let $L^{p}_{s;std}$ and $\cc^{r}_{std}$ denote the standard $L^p$- Sobolev space and Zygmund-H\"older space respectively (these agree with $L^p_{s}(\{(\partial_{x_1}, 1),..., (\partial_{x_n},1)\})$ and $\cc^{r}(\{(\partial_{x_1}, 1),..., (\partial_{x_n},1)\})$; see Remark \ref{Preliminaries: Function space: Zygmund Holder space: remark about classical space} and Remark \ref{Preliminaries: Function space: Sobolev space: remark about classical space}). 
   \begin{theorem}\label{Intrroduction: No loss of derivative: Change of Vector field theorem: smooth F and Euclidean vector fields}
      Fix $x_0\in \mathcal{M}$. Let $s > r>0$ and $\lambda_{std}:= \lambda_{std}(x_0, (W, \db))$. Let $u$ be a solution to the following PDE
     \begin{align*}
         F(x, \{W^{\alpha}u(x)\}_{\deg_{\db}(\alpha)\leq \kappa})=g(x),
     \end{align*}
     which is maximally subelliptic at $(x_0,u)$ of degree $\kappa$ with respect to $(W, \db)$ (see Definition \ref{introduction: definition of maximally subelliptic PDE}). Suppose that 
\begin{itemize}
    \item $F\in \cc^{\infty}(\mathcal{M}\times \RR^N)$ near $x_0$,
    \item $g\in L^p_{s;std}$ near $x_0$,
    \item $W^{\alpha} u\in \cc^{r}_{std}$ near $x_0$, $\forall \deg_{\db}(\alpha)\leq \kappa$.
\end{itemize}
Then, we have
\begin{enumerate}[label=(\roman*)]
    \item $u\in L^p_{s+\kappa/\lambda_{std}; std}$ near $x_0$, and
    \item $W^{\alpha} u\in L^p_{s;std}$ near $x_0$, $\forall \deg_{\db}(\alpha)\leq \kappa$. 
\end{enumerate}
  \end{theorem}

  \begin{remark}
      The proof of the Theorem \ref{Intrroduction: No loss of derivative: Change of Vector field theorem: smooth F and Euclidean vector fields} requires the multi-parameter result Theorem \ref{Introduction: Main Theorem}.   This is also one of the main reasons why the multi-parameter theory was developed. 
      \end{remark}
      \begin{remark}
   It is easy to see that conclusion (ii) of Theorem \ref{Intrroduction: No loss of derivative: Change of Vector field theorem: smooth F and Euclidean vector fields} is sharp. Conclusion (i) of Theorem \ref{Intrroduction: No loss of derivative: Change of Vector field theorem: smooth F and Euclidean vector fields} is also sharp as the regularity of the solutions for a linear maximally subelliptic PDE cannot be improved beyond the cost factor $\lambda_{std}$(see \cite[Proposition 8.2.1]{BS}), and this result about linear maximally subelliptic PDE is sharp. Also, the main regularity theorem in \cite[Corollary 9.1.11]{BS} for Zygmund-H\"older spaces does not produce a sharp regularity theorem like Theorem \ref{Intrroduction: No loss of derivative: Change of Vector field theorem: smooth F and Euclidean vector fields}; it only comes with an "$\epsilon$-loss" in regularity. 
  \end{remark}
  In Section \ref{No loss of derivative} we prove a stronger version of Theorem \ref{Intrroduction: No loss of derivative: Change of Vector field theorem: smooth F and Euclidean vector fields}(see Theorem \ref{Intrroduction: No loss of derivative: Change of Vector field theorem}), where we replace $\{(\partial_{x_1}, 1),..., (\partial_{x_n},1)\}$ by arbitrary vector fields with formal degrees under some mild conditions and much weaker assumptions on the regularity of the function $F$. The proof goes through multi-parameter function spaces and a quantitative regularity theorem for the same. For the rest of the article we will only talk about multi-parameter theory as the single parameter results can be read off from these multi-parameter results. 
  
  \textbf{Acknowledgements}. The author wishes to thank Brian Street for all the helpful discussions during the preparation of this manuscript. The author was partially supported by NSF DMS 2153069. 
\section{Main multi-parameter theorem}\label{Main multi-parameter theorem}
 In this section, we will state the main multi-parameter theorem. Even though all the results henceforth are stated in the terms of general multi-parameter setting, we recommend the reader to the think of the results as in single or bi-parameter setting. Let $\nu\in \NN_+$ and suppose for each $\mu\in \{1,.., \nu\}$ we have the following lists of vector fields with formal degrees:
\begin{align*}
    (W^\mu, \db^{\mu}):=\{(W_1^{\mu}, \db_1^{\mu}), ..., (W_{r_{\mu}}^{\mu}, \db^{\mu}_{r_{\mu}})\} \subset C^{\infty}(\mathcal{M}; T\mathcal{M})\times \NN_+.
\end{align*}
Let $\mathfrak{e}_1,..., \mathfrak{e}_{\nu}$ be the standard basis of $\RR^{\nu}$. We denote 
\begin{align*}
    (W, \dv)&:= \{(W_1, \dv_1),..(W_r, \dv_r)\}\\
    &:=\{(W_j^{\mu}, \db_{j}^{\mu}\mathfrak{e}_{\mu}): 1\leq \mu\leq \nu, 1\leq j\leq r_{\mu}\}\\
    &=: (W^1, \db^1)\boxtimes (W^2, \db^2) \boxtimes ...(W^\nu, \db^{\nu})\subset C^{\infty}(\mathcal{M};T\mathcal{M}) \times (\NN^{\nu}\backslash \{0\}),
\end{align*}
where $(W, \dv)$ satisfies the following conditions: 
\begin{enumerate}[label=(\Roman*)]
    \item $(W^1, \db^1)$ are H\"ormander vector fields with formal degrees.
    \item $Gen((W^{\mu}, \db^{\mu}))$(see Definition \ref{Preliminaries: The Geometry: definition of Gen}) is locally finitely generated for $\mu\in \{2,..., \nu\}$.
    \item $(W^1, \db^1),..., (W^{\nu}, \db^{\nu})$ "pairwise locally weakly approximately commute"(see Definition \ref{Preliminaries: The Geometry: locally weakly approximately commute}) on $\mathcal{M}$. 
\end{enumerate}

 Fix $D_1,D_2, \kappa\in \NN_+$ such that $\db_{j}^1$ divides $\kappa$, $\forall j;$ and define 
 \begin{align}
     N:= D_2 \times \#\{\alpha\ \text{an ordered multi-index}: \deg_{\db^1}(\alpha)\leq \kappa\}.
 \end{align}
 
 Fix $\rv,\sv\in (0,\infty)^{\nu}$, $1<p<\infty$ and $M\geq (|\sv|_{\infty}+1)\nu +2>0$. We use multi-parameter function spaces similar to the single-parameter functions spaces we informally defined in Section \ref{Introduction}. For a compact set $\mathcal{K}\subset M$, $L^p_{\sv}(\mathcal{K}, (W, \dv))$ and $\cc^{\sv}(\mathcal{K}, (W, \dv))$ are defined in such a way that $W_j:L^p_{\sv}(\mathcal{K}, (W, \dv))\to L^p_{\sv-\dv_j}(\mathcal{K}, (W, \dv))$ and $W_j: \cc^{\sv}(\mathcal{K} (W, \dv))\to \cc^{\sv-\dv_j}(\mathcal{K}, (W, \dv))$, i.e $W_j$ acts as a differential operator of order $\dv_j\in \NN^{\nu}\backslash\{0\}$. We also introduce a new function space $\cc^{\sv, t}((W, \dv)\boxtimes \nabla_{\RR^N})$ for $t>0$, whose elements should be thought of as being in $\cc^{\sv}((W, \dv))$ in the first variable and $\cc^{t}_{std}$ in the second variable. All these function spaces will be rigorously defined in Section \ref{Function spaces}.
 
 Now, suppose that:
 \begin{align*}
     F(x,\zeta) &\in \cc^{\sv, 2M} (\widebar{B^{n}(7/8)} \times \RR^N, (W,\dv) \boxtimes \nabla_{\RR^N};\RR^{D_1}), \\
     & u\in \cc^{\rv+\kappa \ef_1} (\widebar{B^{n}(7/8)}, (W,\dv); \RR^{D_2}),\\
     & g\in L^{p}_{\sv}(\widebar{B^{n}(7/8)}, (W,\dv); \RR^{D_1}).
 \end{align*}
 Suppose also that we have the following:
 \begin{align}\label{Introduction: Main theorem: The PDE}
     F(x, \{(W^1)^{\alpha}u(x)\}_{\deg_{\db^1}(\alpha)\leq \kappa})= g(x), \ \forall x\in B^{n}(3/4).
 \end{align}
 \begin{remark}\label{Introdcution: Mapping property of the differential operator}
    We remark that under the assumption on $u$, $g$ is not just in $L^{p}_{\sv}(W,\dv)$ but also in $g\in \cc^{\rv}(W,\dv)$. This is a consequence of Theorem \ref{Preliminaries: Besov and Triebel-Lizorkin space: Prposition 6.5.9 of BS}. 
\end{remark}
 Fix a constant $C_0\geq 0$ such that 
 \begin{align}\label{Introduction: Maximal subellipticity inequality}
     \|F\|_{\cc^{\sv, 2M}((W,\dv)\boxtimes \nabla_{\RR^N})}, \|u\|_{\cc^{\rv+\kappa \ef_1}(W,\dv)}, \|g\|_{L^{p}_{\sv}(W,\dv)}\leq C_0. 
 \end{align}
 Set $n_{j}:= \kappa/ \db^{1}_{j}\in \NN_{+}$ and define $\mathcal{P}$ by  
 \begin{align}\label{Introduction: Main theorem: Linearization}
     \mathcal{P}v:= d_{\zeta} F(0, \{(W^1)^{\alpha}u(0)\}_{\deg_{\db^1}(\alpha)\leq \kappa)})\{(W^1)^{\alpha}v\}_{\deg_{\db^1}(\alpha)\leq \kappa}. 
 \end{align}
 We assume that $\mathcal{P}$ is maximally subelliptic in the quantitative sense that there exists $A\geq 0$ with 
 \begin{align}\label{Introduction: Main theorem: subelliptic Schauder type estimate}
     \sum_{j=1}^{r_1} \|(W_{j}^1)^{n_j}f\|_{L^2(B^{n}(1),h\sigma_{Leb}; \CC^{D_2})} \leq A (\|\pp f\|_{L^2(B^{n}(1),h\sigma_{Leb}; \CC^{D_1})}+\|f\|_{L^2(B^{n}(1),h\sigma_{Leb}; \CC^{D_2})}) 
 \end{align}
 for all $f\in C_{0}^{\infty}(B^{n}(1);\CC^{D_2})$, where $\sigma_{Leb}$ is the Lebesgue measure. 
 
 \begin{theorem}\label{Introduction: Main theorem: qualitative version}
   We assume that 
   \begin{align*}
       F\left(x, \{(W^{1})^{\alpha}u(x)\}_{\deg_{\db^1(\alpha)}\leq \kappa}\right)
   \end{align*}
   is maximally subelliptic at $(x_0, u)$ of degree $\kappa$ with respect to $(W^1, \db^1)$. Also, let 
   \begin{align*}
      F\left(x, \{W^{\alpha}u(x)\}_{\deg_{\db^1(\alpha)}\leq \kappa}\right)= g(x), 
   \end{align*}
   where $g: \mathcal{M}\to \RR^{D_1}$. 
   Suppose that 
     \begin{itemize}
         \item $F\in \cc^{\sv, 2M}((W, \dv)\boxtimes \nabla_{\RR^N})$ near $x_0$,
         \item $g\in L^{p}_{\sv}((W, \dv))$ near $x_0$,
         \item $u\in \cc^{\rv+\kappa \ef_1} (W, \dv)$ near $x_0$.
     \end{itemize}
     Then, $u\in L^p_{\sv+\kappa \ef_1}(W, \dv)$ near $x_0$. 
 \end{theorem}
 \textbf{Proof sketch of Theorem \ref{Introduction: Main theorem: qualitative version}}. The proof proceeds through a fixed point argument that was developed in \cite{BS}, which in itself is motivated by techniques from \cite{LS}. The proof is divided into three simpler steps. 
 \begin{itemize}
     \item[--] Reduction I: We study a simpler version of Theorem \ref{Introduction: Main theorem: qualitative version} and then reduce it to a perturbation of a linear PDE. 
     \item[--] Reduction II: We prove the interior regularity theorem for the perturbed linear operator obtained from Reduction I using fixed point argument. We use the sharp regularity theorem for linear maximally subelliptic PDE from \cite[Section 8]{BS}. One of the most crucial tool in this step is the tame estimate for $L^p$-Sobolev spaces proven in Section \ref{Tame estimate for Sobolev space}. 
     \item[--] In the final step we reduce the partial differential operator in Theorem \ref{Introduction: Main theorem: qualitative version} to the simpler form we assumed in Reduction I. 
 \end{itemize}

\section{Preliminaries}
\numberwithin{equation}{section}
In this section, we will explain some background results from \cite{BS}, and prove some technical lemmas necessary for our proof. We will mostly present definitions and statements of theorems in this section and prove only those statements not being proven in \cite{BS}.

\subsection{The geometry}\label{The Geometry}Throughout the article we will work with a $C^{\infty}$ connected manifold $\mathcal{M}$ of dimension $n\in \NN_+$, endowed with a smooth, strictly positive density $Vol$ unless otherwise specified.

Given a list of vector fields with formal degrees $(W, \dv):= \{(W_1, \db_1),..., (W_r, \db_r)\}$, with $\db_j\in \NN_+$ and $\delta>0$, we define the Carnot-Caratheodory ball of radius $\delta>0$, centered at $x\in \mathcal{M}$ to be 
\begin{align*}
    B_{W, \db}(x, \delta):= \Bigg\{ y\in \mathcal{M}\Bigg|\ & \exists \gamma: [0,1]\to \mathcal{M}, \gamma(0)=x, \gamma(1)=y\\
    & \gamma\ \text{is absolutely continuous},\\
    & \gamma'(t)= \sum_{j=1}^{r} a_{j}(t)\delta^{\db_j}W_j(\gamma(t))\ \text{almost everywhere},\\
    & a_{j}\in L^{\infty}([0,1]), \left\|\sum_{j=1}^{r} |a_j|^2\right\|_{L^{\infty}([0,1])}<1
    \Bigg\}.
\end{align*}

\begin{definition}(\cite[Definition 3.8.5]{BS})\label{Preliminaries: The Geometry: locally weakly approximately commute}
   Let $\nu_1, \nu_2\in \NN_+$. Given $\mathcal{S}_1\subset C^{\infty}(\mathcal{M}; T\mathcal{M})\times (\NN^{\nu_1}\backslash\{0\})$ and $\mathcal{S}_2 \subset C^{\infty}(\mathcal{M}; T\mathcal{M})\times (\NN^{\nu_2}\backslash\{0\})$, we say that $\mathcal{S}_1$ and $\mathcal{S}_2$ \textit{locally weakly approximately commute} if $Gen(\mathcal{S}_1)$ and $Gen(\mathcal{S}_2)$ satisfies the following for all open, relatively compact set $\Omega\Subset \mathcal{M}$: $\forall (Z_1, \vec{\widetilde{d}}_1)\in Gen(\mathcal{S}_1)$  and $\forall (Z_2, \vec{\widetilde{d}}_2)\in Gen(\mathcal{S}_2)$ there are finite sets 
   \begin{align*}
       \mathcal{F}_1\subset \{(Y, \vec{\hat{d}})\in Gen(\mathcal{S}_1): \vec{\hat{d}}\leq \vec{\widetilde{d}}_1\}, \quad    \mathcal{F}_2\subset \{(Y, \vec{\hat{d}})\in Gen(\mathcal{S}_2): \vec{\hat{d}}\leq \vec{\widetilde{d}}_2\}
   \end{align*}
   such that 
   \begin{align*}
       [Z_1, Z_2]= \sum_{(Y, \vec{\hat{d}})\in \mathcal{F}_1} a_{(Y, \vec{\hat{d}})}Y + \sum_{(Y, \vec{\hat{d}})\in \mathcal{F}_2} b_{(Y, \vec{\hat{d}})}Y, \quad  a_{(Y, \vec{\hat{d}})}, b_{(Y, \vec{\hat{d}})} \in C^{\infty}(\Omega). 
   \end{align*}
\end{definition}
\begin{example}
Let $\mathcal{S}_1\subset C^{\infty}(\mathcal{M}; T\RR^n)\times (\NN^{\nu_1}\backslash\{0\})$ be any set of vector fields with formal degrees on $\RR^n$. Let $\mathcal{S}_2:=\{(\partial_{x_1,1}..., (\partial_{x_n},1))\} \subset C^{\infty}(\mathcal{M}; T\mathcal{M})\times \NN_+$. Then, $\mathcal{S}_1$ and $\mathcal{S}_2$ locally weakly approximately commute as the commutator of any $(Z_1, \vec{\widetilde{d}}_1)\in Gen(\mathcal{S}_1)$ with any $\partial_{x_k}$ is a linear combination of $\{\partial_{x_1},..., \partial_{x_n}\}$.
\end{example}
\begin{remark}
    The Definition \ref{Preliminaries: The Geometry: locally weakly approximately commute} might seem complicated at first glance. However, the conditions we put on Definition \ref{Preliminaries: The Geometry: locally weakly approximately commute} is one of the bare minimum conditions so that the Carnot-Caratheodory balls, singular integrals and function spaces that we will define later on are well defined (see \cite[Section 3.5, Section 5]{BS}).
\end{remark}
\subsubsection{Multi-parameter unit scale}\label{Multi-parameter unit scale}
In this section we will work on the unit ball $B^{n}(1)$ around the origin in $\RR^n$. The goal is to rephrase Theorem \ref{Introduction: Main theorem: qualitative version} in the unit ball in a quantitative sense; see Theorem \ref{Introduction: Main Theorem}. 

Let $h\in C^{\infty}(B^{n}(1);\RR),$ with $\inf_{x\in B^{n}(1)}\ h(x)>0$, and let $\sigma_{Leb}$ be the usual Lebesgue measure on $\RR^n$. Fix $\nu\geq 1$ and we assume that for each $\mu\in \{1,...,\nu\}$ we are given 
\begin{align*}
    (W^{\mu},\db^{\mu}):= \{(W_1^{\mu},\db_1^{\mu}),..., (W_{r_{\mu}}^{\mu},\db_{r_{\mu}}^{\mu})\} \subset C^{\infty} (B^{n}(1); TB^{n}(1)\times \NN_+),\\
    (X^{\mu},d^{\mu})= \{(X_1^{\mu},d_1^{\mu}),....,(X_{q_{\mu}}^{}\mu), d_{q_{\mu}}^{\mu}\}\subset Gen((W^{\mu},\db^{\mu}))
\end{align*}
with $(W^{\mu},\db^{\mu})\subseteq (X^{\mu},d^{\mu})$ and such that 
\begin{enumerate}[label=(\roman*)]
    \item $[X_{j},X_{k}]=\sum_{\vec{d}_{l}\leq \vec{d}_{j}+\vec{d}_{k}}c_{j,k}^{l}X_{l},$ $c_{j,k}^{l}\in C^{\infty}(B^{n}(1))$.
    \item $span\{X_1^{1}(x),...,X_{q_1}^{1}(x)\}=T_{x}(B^{n}(1))$, and moreover
    \begin{align*}
        \underset{j_1,...,j_n\in \{1,...,q_1\}}{\max} \underset{x\in B^{n}(1)}{\inf} |\det (X_{j_1}^1(x),...,X_{j_{n}}^{1}(x))|>0.
    \end{align*}
\end{enumerate}
\begin{remark}
    The reason we inteoduced the extra set of vector fields $(X, \vec{d})$ is to make the assumptions on $(W, \dv)$ from Section \ref{Main multi-parameter theorem} quantitative. We leave it to the reader to verify that the above item (i) and (ii) implies conditions (I)-(III) from Section \ref{Main multi-parameter theorem}. 
\end{remark}
\begin{definition}(\cite[Definition 3.15.1]{BS})\label{Introduction: Definition of multi-parameter unit-admissible constant}
    For a parameter $\iota$, we say that a constant $C$ is a $\iota$-multi-parameter unit-admissible constant if there exists $L\in \NN$ depending only on $\iota$ and upper bounds $q, \nu$ and $\max_{1\leq j\leq q} |d_{j}|_1 $ such that $C$ can be chosen to depend only $\iota$, upper bounds for $q,\nu$ and \\$\max_{1\leq j\leq q} |d_{j}|_1, \max_{1\leq j\leq q} \|X_{j}\|_{C^{L}(B^{n}(1);\RR)}, \max_{1\leq j,k,l\leq q}\|c_{j,k}^{l}\|_{C^{L}(B^{n}(1))}$, and $\|h\|_{C^{L}(B^{n}(1))}$ and lower bounds $>0$ for $\inf_{x\in B^{n}(1)} h(x)$ and \\ $\underset{j_1,...,j_n\in \{1,...,q\}}{\max} \underset{x\in B^{n}(1)}{\inf} |\det (X_{j_1}^1(x),...,X_{j_{n}}^{1}(x))|$. If $\iota_0$ is another parameter , we day $C=C(\iota_0)$ is an $\iota$- multi-parameter unit-admissible constant if $C$ is an $\iota$-multi-parameter unit -admissible constant, which can also depend on $i_0$. 
 \end{definition}
 \begin{remark}
     In definition \ref{Introduction: Definition of multi-parameter unit-admissible constant}, we defined the notion of a multi-parameter unit admissible constant as we want to keep track of the constants as we scale the vector fields. We will scale the vector fields infinitely many times, seeking a uniform estimate across all scales for the inequalities we will later prove in this article. These scalings will be made precise in Theorem \ref{Preliminaries: Theorem 3.15.5 of [BS]}.
 \end{remark}
 \begin{notation}
     Throughout this article we use the following notation. We write $A\lesssim_{\iota; \iota_0}B$ to denote $A\leq CB$, where $C=C(\iota_0)\geq 0$ is an $\iota$-multi-parameter unit-admissible constant. We write $A\approx_{\iota;\iota_0} B$ for $A\lesssim_{\iota;\iota_0}B$ and $B\lesssim_{\iota;\iota_0} A$. We write $A\lesssim_{\iota_0}B$ and $A\approx_{\iota_0}B$ for $A\lesssim_{0;\iota_0}B$ and $A\approx_{0;\iota_0}B$, respectively. 
 \end{notation}
Let $\lambda= (\lambda_1, \lambda_2,..., \lambda_{\nu}):=(1, \lambda_2, ..., \lambda_{\nu})\in (0,\infty)^{\nu}$ be such that for $1\leq j \leq r_{\mu}$, 
\begin{align*}
    W_{j}^{\mu}=\sum_{d_{l}^1\leq \lambda_{\mu}\db_{j}^{\mu}} c_{j, \mu}^{l}X_l^1, \ c_{j, \mu}^l \in C^{\infty}(B^n(1))
\end{align*}
satisfying $\|c_{j,\mu}^l\|_{C^L(B^n(1))}\lesssim_{L} 1, \ \forall L\in \NN$, which is always possible by assumption (ii) above. Now, define 
\begin{align*}
    \Lambda(x, \delta)&:= h(x)\underset{1\leq j_1,..., j_n\leq q_1}{\max}\ \left|\det(\delta^{\db_{j_1}^1}X_{j_1}^1(x)|...|\delta^{\db_{j_n}^1}X_{j_n}^1(x))\right|\\
    &=h(x)\underset{1\leq j_1,..., j_n\leq q_1}{\max}\ \sigma_{Leb}(\delta^{\db_{j_1}^1}X_{j_1}^1(x),...,\delta^{\db_{j_n}^1}X_{j_n}^1(x)).
\end{align*}
Next, we state a theorem that  helps us to quantitatively scale the vector fields $(W,\dv)$ and $(X, \vec{d})$ to the unit scale. 
\begin{notation}
    Given $C^{\infty}$ diffeomorphism $\Phi:\mathcal{M}\to \mathcal{N}$ for two manifolds $\mathcal{M}$ and $\mathcal{N}$ we write:
    \begin{itemize}
        \item If $f\in C^{\infty}(\mathcal{N})$, then $\Phi^*f=f\circ \Phi\in C^{\infty}(\mathcal{M})$.
        \item If $X\in C^{\infty}(\mathcal{N};T\mathcal{N})$, then $\Phi^*X\in C^{\infty}(\mathcal{M};T\mathcal{M})$ is defined by 
        \begin{align*}
            \Phi^* X(p)= (d\Phi)^{-1}(X(\Phi(p)))\in T_{p}\mathcal{M}.
        \end{align*}
    \end{itemize}
\end{notation}
\begin{theorem}\label{Preliminaries: Theorem 3.15.5 of [BS]}(\cite{BS},Theorem 3.15.5)
    Fix $\sigma\in (0,1)$. There exists $0$-multi-parameter unit- admissible constant $\xi_1=\xi_1(\sigma)$ such that the following hold: 
   
       For all $x\in \widebar{B^{n}(\sigma)},\delta\in (0,1]$, there exists a map $\Phi_{x,\delta}:B^{n}(1)\to B_{(X^{1},d^1)}(x,\delta)\cap B^{n}((1+\sigma)/2)$ such that:
       \begin{enumerate}[label=(\roman*)]
       \item $\Phi_{x,\delta}(0)=x$.
       \item $\Phi_{x,\delta}$ is a smooth coordinate system, that is, $\Phi_{x,\delta}(B^{n}(1))\subseteq B^{n}(1)$ is open and $\Phi_{x,\delta}: B^{n}(1)\to \widebar{B^{n}(1)}$ is a $C^{\infty}$ diffeomorphism.
       \item $\forall x\in \widebar{B^{n}(\sigma)},\delta\in (0,1],B_{(W^1,\db^1)}(x,\xi_1\delta)\subseteq B_{(X^1,d^1)}(x,\xi_1\delta)\subseteq \Phi_{x,\delta}(B^{n}(1/2))\subseteq \Phi_{x,\delta}(B^{n}(1))\subseteq B_{(X^1,d^1)}(x,\delta)$.

       Set $W_{j}^{\mu,x,\delta}:= \Phi_{x,\delta}^{*}\delta^{\lambda_{\mu}\db_{j}^{\mu}}W_{j}^{\mu}$ and $X_{k}^{\mu,x,\delta}:= \Phi_{x,\delta}^{*}\delta^{\lambda_{\mu}d_{k}^{\mu}}X_{k}^{\mu}$. Let 
       \begin{align*}
           (W^{\mu,x,\delta},\db^{\mu})&:= \{(W_{1}^{\mu,x,\delta},\db_{1}^{\mu}),...(W_{r_{\mu}}^{\mu,x,\delta},\db_{r_{\mu}}^{\mu})\}, \\
           (X^{\mu,x,\delta},d^{\mu})&:= \{(X_{1}^{\mu,x,\delta},d_{1}^{\mu}),...(W_{q_{\mu}}^{\mu,x,\delta},d_{q_{\mu}}^{\mu})\}.
       \end{align*}
       Define $h_{x,\delta}$ by $\Phi_{x,\delta}^{*}h\sigma_{Leb}= \Lambda(x,\delta)h_{x,\delta}\sigma_{Leb}$. Then $(W^{1,x,\delta},\db^1),...,(W^{\nu,x,\delta},\db^{\nu}), (X^{1,x,\delta},d^{1}),...\\(X^{\nu,x,\delta},d^{\nu}),$ and $h_{x,\delta}$ satisfy the following:
       \item $(W^{\mu,x,\delta},\db^{\mu})\subseteq (X^{\mu,x,\delta},d^{\mu})\subset Gen((W^{\mu,x,\delta}),\db^{\mu})$.
       \item $\forall L\in \NN, x\in \widebar{B^{n}(\sigma)}, \delta\in (0,1]$,
       \begin{align*}
           \|W_{j}^{\mu,x,\delta}\|_{C^{L}(B^{n}(1);\RR^{n})}&\lesssim_{L;\sigma} 1,\\
            \|X_{k}^{\mu,x,\delta}\|_{C^{L}(B^{n}(1);\RR^{n})}&\lesssim_{L;\sigma} 1.
       \end{align*}
       \item For all $x\in \widebar{B^n(\sigma)},\delta\in(0,1]$, 
       \begin{align*}
           \underset{j_1,...,j_{n}\in \{1,...,q_1\}}{\max}\underset{u\in B^{n}(1)}{\inf} \ \left|\det(X_{j_1}^{1,x,\delta}(u)|...|X_{j_n}^{1,x,\delta}(u)) \right| \gtrsim_{\sigma} 1.
       \end{align*}
              \item For all $x\in \widebar{B^n(\sigma)},\delta\in(0,1]$,
          \begin{align*}
            \|f\|_{C^{L}(B^{n}(1))} \approx_{L,\sigma} \sum_{|\alpha|\leq L} \|(X^{1,x,\delta})^{\alpha}f\|_{C^{L}(B^{n}(1))}. 
       \end{align*}
\item For $x\in \widebar{B^{n}(\sigma)}$ and $\delta \in (0,1]$, set 
\begin{align*}
    (X^{x,\delta},\vec{d})&:= \{(X_1^{x,\delta},\vec{d}_1),...(X_{q}^{x,\delta},\vec{d}_q)\}\\
    &:=(X^{1,x,\delta},d^1) \boxtimes (X^{2,x,\delta},d^2)\boxtimes ...\boxtimes (X^{\nu,x,\delta},d^\nu).
\end{align*}
Then 
\begin{align*}
    [X_{j}^{x,\delta}, X_{k}^{x,\delta}]= \sum_{\vec{d}_l\leq \vec{d}_j +\vec{d}_k} c_{j,k}^{l,x,\delta} X_{l}^{x,\delta}, \ c_{j,k}^{l,x,\delta} \in C^{\infty}(B^{n}(1)),
\end{align*}
where $\forall L \in \NN$,
\begin{align*}
    \|c_{j,k}^{l,x,\delta}\|_{C^{L}(B^{n}(1))}\lesssim_{L;\sigma} 1. 
\end{align*}
\item $\forall L\in \NN, \|h_{x,\delta}\|_{C^{L}(B^{n}(1))}\lesssim_{L,\sigma} 1. $
\item $h_{x,\delta}(u)\approx_{\sigma }1,\ \forall u\in B^{n}(1) $. 
   \end{enumerate}
\end{theorem}
\begin{remark}\label{Preliminaries: The Geometry: remark about theorem 3.15.5}
    Theorem \ref{Preliminaries: Theorem 3.15.5 of [BS]} shows that for each $x\in \widebar{B^{n}(\sigma)}$ and $\delta\in (0,1]$, the multi-parameter unit-admissible constants with respect to $(W^{x, \delta}, \dv), (X^{x, \delta}, \vec{d})$ and $h_{x, \delta}$ can be chosen to be uniform in $\delta$. This idea will be extensively used in the proof of Theorem \ref{Introduction: Main theorem: qualitative version}. 
\end{remark}
\subsection{Function spaces}\label{Function spaces}
In this section, we will introduce the appropriate function spaces as defined in \cite[Section 6]{BS}. We define the general Besov and Triebel Lizorkin spaces in the multi-parameter setting. 

We do not introduce the single-parameter spaces here. However, the single-parameter spaces can be obtained by reading the following definitions with $\nu=1$. 

Crucial to the definition of classical Besov and Triebel-Lizorkin spaces is the notion of "Littlewood-Paley projection operators". So, we need such a notion of operators adapted to $(W, \dv)$ before we can define the appropriate function spaces. For $j\in \RR^{\nu}$, we write $2^{-j\dv}W$ for $\{2^{-j\cdot \dv_1}W_1,...2^{-j\cdot \dv_r}W_r\}$. 
\begin{definition}(\cite[Definition 5.2.22]{BS})
    Set 
    \begin{align*}
        C^{\infty}_W (\mathcal{M}):= \{f: \mathcal{M}\to \CC\ |\  \underset{x\in U}{\sup} \sum_{|\alpha|\leq L} |W^{\alpha} f(x)|<\infty, \forall U\Subset \mathcal{M}\ \text{open}\ , \forall L\in \NN\}.
    \end{align*}
        
\end{definition}
    \begin{remark}
        $C^{\infty}_W(\mathcal{M})$ with the usual projective limit topology is a Fr\'echet space. 
    \end{remark}

\begin{definition}(\cite[Definition 5.2.23]{BS})
    Set 
    \begin{align*}
        C^{\infty}_{W,0}:= \{f\in C^{\infty}_W(\mathcal{M}): supp(f) \ \text{is compact}\}.
    \end{align*}
\end{definition}
\begin{remark}
    With the usual inductive topology on $C^{\infty}_{W, 0}(\mathcal{M})$ we get an LF space. 
\end{remark}
\begin{definition}(\cite[Definition 5.2.26]{BS})\label{Preliminaries: Function spaces: Definiton of pre-elementary operators}
    We say $\mathcal{E}\subset Hom(C^{\infty}_{W,0}(\mathcal{M}), C^{\infty}_W(\mathcal{M}))\times (0,1]^{\nu}$ is a bounded set of generalized $(W, \dv)$ pre-elementary operators supported in $\Omega$ if the following hold:
    \begin{itemize}
        \item $\bigcup_{(E, 2^{-j})\in \mathcal{E}} supp(E) \Subset \Omega\times \Omega$.
        \item For all ordered multi-indices $\alpha, \beta$ the operator 
        \begin{align}
            (2^{-j\dv}W)^{\alpha} E(2^{-j\dv}W)^{\beta}
        \end{align}
        extends to a bounded operator $L^1(\mathcal{M},Vol)\to L^1(\mathcal{M}, Vol)$, and
        \begin{align}
            \underset{(W, 2^{-j})\in \mathcal{E}}{\sup} \|(2^{-j\dv}W)^{\alpha} E (2^{-j\dv}W)^{\beta}\|_{L^1(\mathcal{M}, Vol)\to L^1(\mathcal{M}, Vol)}<\infty.
        \end{align}
        \item For all ordered multi-indices $\alpha, \beta$,
        \begin{align}
\underset{(W, 2^{-j})\in \mathcal{E}}{\sup} \|(2^{-j\dv}W)^{\alpha} E (2^{-j\dv}W)^{\beta}\|_{L^{\infty}(\mathcal{M}, Vol)\to L^{\infty}(\mathcal{m}, Vol)}<\infty.
        \end{align}
        \item For all ordered multi-indices $\alpha, \beta$ and for every countable set $\mathcal{E}' \subset \{(E_k, \delta_k): k\in \NN\} \subseteq \mathcal{E}$, set 
        \begin{align}
            T_{\mathcal{E}', \alpha, \beta}\{f_k\}_{k\in \NN} := \{(\delta_k^{\dv}W)^{\alpha} E_k (\delta_k^{\dv}W)^{\beta}\}_{k\in \NN}.
        \end{align}
        Then, for all $p\in (1, \infty)$, $q\in (1, \infty]$, 
        \begin{align*}
            \underset{\mathcal{E}'}{\sup} \|T_{\mathcal{E}', \alpha, \beta}\|_{L^p(\mathcal{M}, Vol; l^q(\NN))\to L^p(\mathcal{M}, Vol; l^q(\NN))}<\infty,
        \end{align*}
        where the supremum is taken over all such countable subsets $\mathcal{E}'\subseteq \mathcal{E}$. 
        \item For $(E, 2^{-j})\in \mathcal{E}$, we assume $E^*$ initially defined on $C^{\infty}(\mathcal{M})'$ restricts to an operator on $Hom(C^{\infty}_{W,0}(\mathcal{M}), C^{\infty}_{W}(\mathcal{M}))$. Furthermore, set 
        \begin{align*}
            \mathcal{E}^*:&= \{(E^*, 2^{-j}): (E, 2^{-j})\in \mathcal{E}\}\\
            &\subseteq Hom(C_{W,0}^{\infty}(\mathcal{M}), C^{\infty}_W(\mathcal{M}))\times (0,1]^{\nu}.
        \end{align*} 
        We also assume that all of the above also holds with $\mathcal{E}$ replaced with $\mathcal{E}^*$. 
    \end{itemize}
\end{definition}
\begin{definition}
    We define the set of bounded set of generalized $(W, \dv)$ elementary operators supported in $\Omega, \tilde{\mathcal{G}}_{\Omega}$ to be the largest set of subsets of $Hom(C^{\infty}_{W,0}(\mathcal{M}), C^{\infty}_W(\mathcal{M}))$ such that $\forall \mathcal{E}\in \tilde{\mathcal{G}}$: 
    \begin{itemize}
        \item $\mathcal{E}$ is a bounded set of generalized $(W, \dv)$ pre-elementary operators supported in $\Omega$. 
        \item For all $(E, 2^{-j})\in \mathcal{E}$, $\mu\in \{1, ..., \nu\}$, 
        \begin{align*}
            E= \sum_{|\alpha|_{\mu}, |\beta|_{\mu}\leq 1} 2^{-(2- |\alpha_{\mu}|-|\beta_{\mu}|)j_{\mu}}  (2^{-j_{\mu}\db^{\mu}}W^{\mu})^{\alpha_{\mu}} E_{\mu, \alpha_{\mu}, \beta_{\mu}} (2^{-j_{\mu}\db^{\mu}}W^{\mu})^{\beta_{\mu}},
        \end{align*}
        where 
        \begin{align*}
            \{(E_{\mu, \alpha_{\mu}, \beta_{\mu}}, 2^{-j}: (E, ^{-j})\in \mathcal{E}, \mu\in \{1,..., \nu\}, |\alpha_{\mu}|, |\beta_{\mu}|\leq 1)\in \tilde{\mathcal{G}}_{\Omega}\}.
        \end{align*}
    \end{itemize}
    For $\mathcal{E}\in \tilde{\mathcal{G}}_{\Omega}$, we say $\mathcal{E}$ is a bounded set of generalized $(W, \dv)$ elementary operators supported in $\Omega$. We say $\mathcal{E}$ is a bounded set of generalized $(W, \dv)$ elementary operators if there exists $\Omega \Subset \mathcal{M}$ such that $\mathcal{E}$ is a bounded set of generalized $(W, \dv)$ elementary operators supported in $\Omega$. 
\end{definition}

Next, we will state theorem about some basic properties of elementary operators, which we will extensively use throughout the article. 
\begin{theorem}\label{Preliminaries: Function spaces: properties of elementary operators}(\cite[Theorem 5.5.5]{BS})
    Let be a bounded set of generalized $(W, \dv)$ elementary operators supported in $\Omega$. Then, 
    \begin{enumerate}[label=(\alph*)]
        \item Suppose $\{a_l\}_{l\in \NN}$ is a sequence of complex numbers with $\sum |a_l|<\infty$. Then 
        \begin{align*}
            \left\{ \left(\sum_{l\in \NN} a_l E_l , 2^{-j}\right): \{(E_l, 2^{-j}): l\in \NN\} \subseteq \mathcal{E}\right\}
        \end{align*}
        is a bounded set of generalized $(W, \dv)$ elementary operators supported in $\Omega$. 
        \item If $\psi\in C^{\infty}(\Omega)$, then 
        \begin{align*}
            \{(Mult[\psi] E, 2^{-j}), (E Mult[\psi], 2^{-j}): (E, 2^{-j})\in \mathcal{E}\}
        \end{align*}
        is a bounded set of generalized $(W, \dv)$ elementary operators supported in $\Omega$. 
        \item Fix an ordered multi-index $\alpha$. Then, 
        \begin{align*}
            \{((2^{-j\dv}W)^{\alpha} E, 2^{-j}), (E(2^{-j\dv}W)^{\alpha}, 2^{-j}): (E, 2^{-j})\in \mathcal{E}\}
        \end{align*}
        is a bounded set of generalized $(W, \dv)$ elementary operators supported in $\Omega$. 
          \item For each $N\in \NN$ and $\mu\in \{1, ..., \nu\}$, every $(E, 2^{-j})\in \mathcal{E}$ can be written as 
          \begin{align*}
              E = \sum_{|\alpha_{\mu}|\leq N} 2^{|\alpha_{\mu}-N|j_{\mu}} (2^{-j_{\mu}\db^{\mu}}W^{\mu})^{\alpha_{\mu}}E_{\mu, \alpha_{\mu}}, 
          \end{align*}
          where $\{(E_{\mu, \alpha_{\mu}}, 2^{-j}): (E, 2^{-j})\in \mathcal{E}, \mu\in \{1,..., \nu\}, |\alpha_{\mu}|\leq N\}$ is a bounded set of generalized $(W, \dv)$ elementary operators supported in $\Omega$. Similarly, we have 
          \begin{align*}
              E = \sum_{|\alpha_{\mu}|\leq N} 2^{|\alpha_{\mu}-N|j_{\mu}} \tilde{E}_{\mu, \alpha_{\mu}}(2^{-j_{\mu}\db^{\mu}}W^{\mu})^{\alpha_{\mu}}, 
          \end{align*}
           where $\{(\tilde{E}_{\mu, \alpha_{\mu}}, 2^{-j}): (E, 2^{-j})\in \mathcal{E}, \mu\in \{1,..., \nu\}, |\alpha_{\mu}|\leq N\}$ is a bounded set of generalized $(W, \dv)$ elementary operators supported in $\Omega$. 
    \end{enumerate}
\end{theorem}
Now, we are ready to define the Besov and Triebel-Lizorkin spaces. 
\begin{notation}
    $\mathcal{V}$ denotes one of the spaces 
\begin{align*}
    \{l^q(\NN^{\nu}; L^p(\mathcal{M}, Vol): p,q \in [1, \infty])\}\bigcup \{L^p(\mathcal{M}, Vol; l^q(\NN^{\nu})); p\in (1, \infty), q\in (1, \infty]\}.
\end{align*}
For a sequence of distributions $\{f_j\}_{j\in \NN^{\nu}}\subset C_0^{\infty}(\mathcal{M})'$, when we consider $\|\{f_j\}_{j\in \NN^{\nu}}\|_{\mathcal{V}}$ we set this equal to $\infty$ if any of the $f_j$ does not agree with the integration against an $L^1_{loc}(\mathcal{M}, Vol)$ function. If all $f_j$ agree with integration against an $L^1_{loc}(\mathcal{M}, Vol)$ function, $\|\{f_j\}_{j\in \NN^{\nu}}\|_{\mathcal{V}}$ has the usual definition. 
\end{notation}
\begin{notation}\label{Function space: Notation for Besov and Triebel-Lizorkin space}
   Before we define the function space we will introduce a notation for them. For $\sv\in \RR^{\nu}$, $\mathfrak{X}^{\sv}(\mathcal{K},(W, \dv))$ will any one of the spaces
   \begin{align*}
       \{\mathcal{B}^{\sv}_{p,q}(\mathcal{K}, (W, \dv)): p,q \in [1,\infty]\}\bigcup \{\mathcal{F}^{\sv}_{p,q}(\mathcal{K}, (W, \dv)): p\in(1,\infty), q\in(1, \infty]\}.
   \end{align*}
   When $\mathfrak{X}^{s}(\mathcal{K}, (W, \dv))= \mathcal{B}^{\sv}_{p, q}(\mathcal{K}, (W, \dv))$ we take $\mathcal{V}$ to be $l^{q}(\NN^{\nu}; L^{p}(\mathcal{M}, Vol))$, and when $\mathfrak{X}^{s}(\mathcal{K}, (W, \dv))= \mathcal{F}^{\sv}_{p,q}(\mathcal{K}, (W, \dv))$ we take $\mathcal{V}$ to be $L^p(\mathcal{M}; Vol; l^{q}(\NN^{\nu}))$.
\end{notation}
\begin{definition}(\cite[Definition 6.3.3, 6.3.5]{BS})
    Let $\mathcal{E}$ be a  bounded set of generalized $(W, \dv)$ elementary operators. We set for $f\in C^{\infty}_{W, 0}(\mathcal{M})$' and $\sv\in \RR^{\nu}$, 
    \begin{align*}
        \|f\|_{\mathcal{V}, \sv, \mathcal{E}}:= \underset{\{(E_j, 2^{-j}): j\in \NN^{\nu}\}\subseteq \mathcal{E}}{\sup} \|\{2^{j\cdot \sv} E_j f\}\|_{\mathcal{V}}. 
    \end{align*}
    This defines an extended semi-norm on $C^{\infty}_{W,0}(\mathcal{M})'$. So, we define $\mathfrak{X}^{\sv}(\mathcal{K}, (W, \dv))$ to be the space of all those $f\in C^{\infty}_{W, 0}(\mathcal{M})'$ such that $supp(f)\subseteq \mathcal{K}$ and for every bounded set of generalized $(W, \dv)$ elementary operators, $\mathcal{E}$, we have $\|f\|_{\mathcal{V}, s, \mathcal{E}}< \infty$. 
 \end{definition}
 With the function spaces defined, one can study the boundedness properties of differential and pseudodifferential operators on these function space. The following are some of the results that we will use in this article.
 
 \begin{prop}\label{Preliminaries: Besov and Triebel-Lizorkin space: Prposition 6.5.9 of BS}\cite[Proposition 6.5.9]{BS}
    Let $\Omega \Subset \mathcal{M}$ be an open set with $\mathcal{K}\Subset \Omega$. Let $\vec{\kappa} \in \ZZ^{\nu}$ and let $\mathcal{P}$ be a $(W, \dv)$ partial differential operator on $\Omega$ with degree $\leq \kappa$. Then, for every $\sv\in \RR^{\nu}$, $\mathcal{P}:\mathfrak{X}^{\kappa}, (W, \dv)\to \mathfrak{X}^{s-\kappa}(\mathcal{K}, (W, \dv))$ is a bounded operator. 
\end{prop}
The following two corollaries can be see immediately from Proposition \ref{Preliminaries: Besov and Triebel-Lizorkin space: Prposition 6.5.9 of BS}. 
 \begin{corollary}\label{Preliminaries: Besov and Triebel-Lizorkin space: Corollary 6.5.10 of BS}(\cite[Corollary 6.5.10]{BS})
     Let $\phi\in C^{\infty}(\mathcal{M})$. Then, $Mult[\phi]$ is a bounded operator from $\mathfrak{X}^{s}(\mathcal{K}, (W, \dv))\to \mathfrak{X}^{s}(\mathcal{K}, (W, \dv))$. 
 \end{corollary}
\begin{corollary}\label{Preliminaries: Besov and Triebel-Lizorkin space: Corollary 6.5.11 of BS}(\cite[Corollary 6.5.11]{BS})
    For any ordered multi-index $\alpha$, 
    \begin{align*}
        W^{\alpha}: \mathfrak{X}^{\sv}\to \mathfrak{X}^{\sv-\deg_{\dv}(\alpha)} (\mathcal{K}, (W,\dv))
    \end{align*}
    is a bounded operator.
\end{corollary}
    \begin{remark}\label{Function spaces: remark about explicit norm: Equivalence of norms}Now, given a function $\psi\in C^{\infty}_{0}(\mathcal{M})$ with $\psi\equiv 1$ on a neighbourhood of $\mathcal{K}$ one can find a  bounded set of generalized $(W, \dv)$ elementary operators $\mathcal{D}_0$ given by $\mathcal{D}_0=\{(D_j, 2^{-j}: j\in \NN^{\nu})\}$ with $Mult[\psi]= \sum_{j\in \NN^{\nu}} D_j$(see \cite[Proposition 5.8.3, Proposition 5.5.10]{BS} for the existence of $\mathcal{D}_0$ and convergence of this sum). Observe that $\mathcal{D}$ is analogous to the defintion of classical Littlewood-Paley projections. Then, one can define a norm on $\mathfrak{X}^{\sv}(\mathcal{K}, (W, \dv))$ using $\mathcal{D}$ making it a Banach space. Apriori, it would seems like the choice of norm depends on the choice of $\psi$ and $\mathcal{D}$. However, one can show that this is not the case (see \cite[Proposition 6.3.7]{BS} for the equivalence of norms). Moreover, \cite[Corollary 6.4.4]{BS} shows that, for $\mathcal{E}$ a bounded set of generalized $(W, \dv)$ elementary operator, there exists $C\geq 0$ such that 
    \begin{align*}
        \|f\|_{\mathcal{V}, s, \mathcal{E}}\leq C \|f\|_{\mathcal{V}, s, \mathcal{D}_0}.
    \end{align*}
    \end{remark}
    \begin{remark}\label{Function spaces: remark about explicit norm: remark about explicit norm part 2}
To study fully nonlinear maximally subelliptic PDE one needs more than a mere equivalence of norms as in Remark \ref{Function spaces: remark about explicit norm: Equivalence of norms}. Hence, given a compact set $\mathcal{K}\subset \mathcal{M}$, and under the assumption of Theorem \ref{Introduction: Main theorem: qualitative version} we will need to fix a specific $\psi\in C_{0}^{\infty}(\Omega)$ with $\psi\equiv 1$ on a neighbourhood of $\mathcal{K}$ and a  bounded set of generalized $(W, \dv)$ elementary operators $\mathcal{D}=\{(D_j, 2^{-j}: j\in \NN^{\nu})\}$ such that $Mult[\psi]=\sum_{j\in \NN^{\nu}} D_j$ (see \cite[Section 6.11]{BS}). This construction depends on $(W, \dv), (X, \vec{d})$, relatively compact open sets $\Omega_0, \Omega_1, \Omega_2$ with $\mathcal{K}\Subset \Omega_0\Subset \Omega_1\Subset \Omega_2 \Subset \mathcal{M}$ and a $0$-multi-parameter unit-admissible constant $a$, but not on $\psi$. We do not provide the specific construction here as it is beyond the scope of this article. From now on, we will fix the norm given by the above choice that we made.
\end{remark} 
\begin{prop}(\cite[Proposition 7.1.1, Corollary 7.1.4]{BS})\label{Preliminaries: Function spaces: Proposition 7.1.1 and Corollary 7.1.4}
    Fix $\psi$ and $\mathcal{D}$ as in Remark \ref{Function spaces: remark about explicit norm: remark about explicit norm part 2}. Set $P_{j}:= \sum_{k\leq j} D_j$. Then, 
    \begin{enumerate}[label=(\roman*)]
        \item $\forall j\in \NN^{\nu}$ we have $P_j, D_j\in C_{0}^{\infty}(\Omega\times \Omega)$.
        \item $\{(P_{j}, 2^{-j}: j\in \NN^{\nu})\}$ is a bounded set of generalized $(W, \dv)$ pre-elementary operators supported in $\Omega$. 
        \item For every multi-index $\alpha$, 
        \begin{align*}
            \underset{j\in \NN^{\nu}}{\sup} \ \|(2^{-j\dv}W)^{\alpha} P_j\|_{L^{\infty}(\mathcal{M})\to L^{\infty}(\mathcal{M})}<\infty, \ \underset{j\in \NN^{\nu}}{\sup} \ \|(2^{-j\dv}W)^{\alpha} D_j\|_{L^{\infty}(\mathcal{M})\to L^{\infty}(\mathcal{M})}<\infty.
        \end{align*}
    \end{enumerate}
\end{prop}
\subsubsection{Singular integrals}
We will define the algebra of "singular integrals", which have properties similar to the Calder\'on-Zygmund singular integrals. We will use these singular integrals when we talk about the "parametrix" of linear maximally subelliptic operators(see Theorem \ref{Preliminaries: Theorem 8.1.1} part (vii)). We refer the reader the reader to \cite[Chapter 5]{BS} for a more extensive exposition to these singular integrals. 

\begin{definition}(\cite[Definition 5.2.17, Definition 5.11.1]{BS})
\begin{enumerate}
    \item    For $\sv\in \RR^{\nu}$, we let $\mathcal{A}^{\sv}(\Omega, (W, \dv))\subseteq Hom(C^{\infty}_{0}(\mathcal{M}), C^{\infty}_{0}(\mathcal{M})')$ be the set of all $T\in Hom(C^{\infty}_{W, 0}(\mathcal{M}), C^{\infty}_{W,0}(\mathcal{M})')$ such that there exists a bounded set of generalized $(W, \dv)$ elementary operators supported in $\Omega$, $\{(E_j, 2^{-j}): j\in \NN^{\nu}\}$, with $T=\sum_{j\in \NN^{\nu}}2^{j\cdot \sv} E_j$.
    \item We also define $\mathcal{A}_{loc}^{\sv}(W, \dv)$ to be the set of all $T\in Hom(C^{\infty}_{ 0}(\mathcal{M}), C^{\infty}_{0}(\mathcal{M})')$ such that $\forall \psi_1, \psi_2 \in C^{\infty}_{0}(\mathcal{M}), Mult[\psi_1]T Mult[\psi_2] \in \bigcup_{\Omega\Subset \mathcal{M}}\mathcal{A}^{\sv}(\Omega,(W, \dv))$. 
\end{enumerate}
\end{definition}
Next, we will state some properties of these singular integrals that we will use in this article. 
\begin{theorem}\label{Preliminaries: Function spaces: Singular integrals: Theorem 5.8.18}(\cite[Theorem 5.8.18]{BS})
    For all $\sv\in \RR^{\nu}$, the operators in $\mathcal{A}^{\sv}(\Omega, (W, \dv))$ are \textit{pseudo-local}. i.e, if $T\in \mathcal{A}^{\sv}(\Omega, (W, \dv))$ then 
    \begin{align*}
        T(x,y)|_{(x,y): x\neq y} \in C^{\infty}(\{(x,y)\in \mathcal{M}\times \mathcal{M}: x\neq y\}).
    \end{align*}
\end{theorem}
\begin{prop}\label{Preliminaries: Function spaces: Singular integrals: prop 5.8.9, corollary 5.8.10}(\cite[Proposition 5.8.9, Corollary 5.8.10]{BS})
    Let $\alpha$ be an ordered multi-index and $\phi\in C^{\infty}(\Omega)$. If $T\in \mathcal{A}^{\sv}(\Omega, (W, \dv))$, then 
    \begin{align*}
        Mult[\phi]W^{\alpha}T,\ T Mult[\phi]W^{\alpha}\in \mathcal{A}^{\sv+\deg_{\dv}(\alpha)}(\Omega, (W, \dv)).  
    \end{align*}
    Hence, for any $\mathcal{P}$ a $(W, \dv)$ partial differential operator on $\Omega$ of degree $\leq \kappa\in \ZZ^{\nu}$. Then, 
    \begin{align*}
        \mathcal{P}T, T\mathcal{P} \in \mathcal{A}^{\sv+\kappa}(\Omega, (W, \dv)). 
    \end{align*}
\end{prop}
\begin{prop}\label{Preliminaries: Function spaces: Singular integrals: prop 5.8.11}(\cite[Proposition 5.8.11]{BS})
    If $T\in C_{0}^{\infty}(\Omega\times \Omega)$ then $T\in \bigcap_{\sv\in \RR^{\nu}}\mathcal{A}^{\sv}(\Omega, (W, \dv))$. 
\end{prop}

  \begin{theorem}\label{Preliminaries: Function spaces: boundedness of singular integrals, Theorem 6.3.10 of BS}(\cite[Theorem 6.3.10]{BS})
   Let $\Omega$ be an open set in $\mathcal{K}$. Then, $\forall \vec{t}, \sv\in \RR^{\nu}$ the operators in $\mathcal{A}^{\vec{t}}(\Omega,(W,\dv))$ are bounded in $\mathfrak{X}^{\sv}(\mathcal{K},(W,\dv))\to \mathfrak{X}^{\sv-\vec{t}}(\mathcal{K},(W,\dv))$. 
\end{theorem}
\subsubsection{Zygmund-H\"older space}\label{Zygmund-H\"older Space}
Even though we defined a general class of multi-parameter Besov and Triebel-Lizorkin spaces adapted to $(W, \dv)$, we are concerned only about two special classes among these function spaces, namely the multi-parameter Zygmund-H\"older and Sobolev space. 

 Let $(W, \dv)= \{(W_1, \dv_1), ..., (W_r, \dv_r)\}$ be as in Section \ref{Main multi-parameter theorem}.
\begin{definition}(\cite[Definition 7.0.1]{BS})
    Let $\mathcal{K}\subset \mathcal{M}$ be a compact set. Then, for $\sv\in (0,\infty)^{\nu}$ we define $\cc^{\sv}(\mathcal{K}, (W, \dv)):= \mathcal{B}^{\sv}_{\infty, \infty}(\mathcal{K}, (W, \dv))$ (see Notation \ref{Function space: Notation for Besov and Triebel-Lizorkin space}).  We also define $(\cc^{\sv}(\mathcal{K},(W, \dv));\CC^L)$ for some $L\in \NN_+$ to be $\{(f_1,..., f_L): f_{i}\in \cc^{\sv}(\mathcal{K}, (W, \dv)) \forall \ i\in \{1,.., L\}\}$. 
\begin{remark}\label{Preliminaries: Function space: Zygmund Holder space: remark about classical space}
    When $\nu=1$ and $\{(W_1, \db_1), ..., (W_r, \db_r)\}:=\{(\partial_{x_1},1),..., (\partial_{x_n},1)\}$, $\cc^{m}(\mathcal{K}, (W, \db))$ is the classical Zygmund space for $m\in \NN$, and $\cc^{r}(\mathcal{K}, (W, \dv))$ is the classical H\"older space for $0<r<1$. 
\end{remark}
    Let $\psi$ and $\mathcal{D}$ be as in Remark \ref{Function spaces: remark about explicit norm: remark about explicit norm part 2}. Then, for $f\in \cc^{\sv}(\mathcal{K}, W, \dv)$ we define the norm 
    \begin{align}
        \|f\|_{\cc^{\sv}(W, \dv)}:= \|f\|_{\mathcal{B}_{\infty, \infty}^{\sv}(W, \dv)}= \underset{j\in \NN^{\nu}}{\sup}\ 2^{j\cdot s} \|D_j f\|_{L^{\infty}}.
    \end{align}
\end{definition}
\begin{remark}\label{Preliminaries: Function spaces: Zygmund Holder space: Inclusion into continuous functions}
    Since $\sv\in (0,\infty)^{\nu}$, one can show that $\cc^{\sv}(\mathcal{K}, (W, \dv))\subset C(\mathcal{M})$ and the inclusion is continuous(\cite[Corollary 7.2.2]{BS}). 
\end{remark}
To deal with fully nonlinear PDE, we also need product Zygmund-H\"older space $\cc^{\sv, t}(\mathcal{K}\times \RR^N, (W, \dv) \boxtimes \nabla_{\RR^N})$, which we will define now. Let 
\begin{align*}
    (W, \dv)\boxtimes \nabla_{\RR^N}:= (W, \dv)\boxtimes \{(\partial_{x_1},1),..., (\partial_{x_N},1)\}.
\end{align*}
 Let $\psi$ be as before. Consider $\{(\tilde{D}_l, 2^{-l}):l\in \NN^{\nu}\}$, where $\tilde{D}_l: \mathcal{S}(\RR^N)'\to \mathcal{S}(\RR^N)'$ be given by 
 \begin{align*}
     (\tilde{D}_lf)^{\widehat{}}(\xi):= \widehat{\psi}(2^{-l}\xi) \widehat{f}(\xi), 
 \end{align*}
 where $\wedge$ denotes the Fourier transform. Recall that this gives us a bounded set of elementary operators on $\RR^N$ and we also have $\sum_{l\in \NN} \tilde{D}_l= Id$. Set $\widehat{D}_{j,l}:= D_{j}\otimes \tilde{D_l}$(this is an abuse of notation as $\widehat{D}_{j,l}$ doesn't have anything to do with the Fourier transform).
 \begin{definition}(\cite[Definition 7.5.1]{BS})
     For $\sv\in (0,\infty)^{\nu}$ and $t>0$, we let $\cc^{\sv, t}(\mathcal{K}\times \RR^N, (W, \dv)\boxtimes \nabla_{\RR^N})$ denote the space of all $F\in C_0^{\infty}(\mathcal{M}\times \RR^N)'$ such that the following norm is finite: 
     \begin{align*}
         \|F\|_{\cc^{\sv, t}((W, \dv)\boxtimes \nabla_{\RR^N})}:= \underset{j\in \NN^{\nu}, l\in \NN}{\sup} 2^{j\cdot \sv} 2^{lt} \|\widehat{D}_{j,l} F\|_{L^{\infty}(\mathcal{M}\times \RR^N)}.
     \end{align*}
 \end{definition}
 For $L\in \NN_+$, we define the vector valued space $\cc^{\sv, t}(\mathcal{K}\times \RR^N, (W, \dv)\boxtimes \nabla_{\RR^N};\CC^L)$ to consist of those distributions $F=(F_1,..., F_L)$ with each $F_k\in \cc^{\sv, t}(\mathcal{K\times \RR^N}, (W, \dv)\boxtimes \nabla_{\RR^N})$. We set 
 \begin{align*}
     \|F\|_{\cc^{\sv, t}((W, \dv)\boxtimes \nabla_{\RR^N};\CC^L)}:= \sum_{k=1}^L\|F_k\|_{\cc^{\sv, t}((W, \dv)\boxtimes \nabla_{\RR^N})}.
 \end{align*}
 \begin{remark}
     $\cc^{\sv, t} (\mathcal{K}\times \RR^N, (W, \dv)\boxtimes \nabla_{\RR^N})$ is related to the $(\nu+1)$-parameter Zygmund-H\"older space $\cc^{(\sv,t)}$. The difference is that $F(x, \zeta)\in \cc^{\sv, t} (\mathcal{K}\times \RR^N, (W, \dv)\boxtimes \nabla_{\RR^N})$ is need not be compactly supported in the $\zeta$ variable. 
 \end{remark}
 Next, we will provide some properties of these product Zygmund-H\"older spaces. 
\begin{prop}(\cite[Proposition 7.5.9]{BS}) \label{Preliminaries: Zygmund H\"older space: smooth decompisition of CH space}

    Let $\sv\in (0,\infty)^{\nu}$ and $t>0$. For $F\in C_{0}^{\infty}(\mathcal{M}\times \RR^N)'$, the following are equivalent
    \begin{enumerate}[label=(\roman*)]
        \item    $F\in \cc^{\sv, t}(\mathcal{K}\times \RR^N, (W, \dv)\boxtimes \nabla_{\RR^N})$. 
        \item $supp(F)\subseteq \mathcal{K}\times \RR^N$ and there exists a sequence $\{F_{j,l}\}_{j\in\NN^{\nu}}\subset C^{\infty}(\mathcal{M}\times \RR^N)$ such that for every ordered multi-index $\alpha$ and evert multi-index $\beta$, 
        \begin{align*}
            \underset{j\in \NN^{\nu}, l\in \NN}{\sup} 2^{j\cdot\sv}2^{lt} \|(2^{-j\dv}W_x)^{\alpha}(2^{-l}\partial_{\zeta})^{\beta}F_{j, l}(x, \zeta)\|_{L^{\infty}(\mathcal{M}\times \RR^N)}<\infty,
        \end{align*}
        with $F=\sum_{j,l}F_{j,l},$ where the sum converges in $C(\mathcal{M}\times \RR^N)$.
    \end{enumerate}
\end{prop}
\begin{remark}\label{lemma 9.2.19 and 9.2.18}
    Lemma 9.2.19 of \cite{BS} says that smooth functions are in $\cc^{\sv}(W,\db)$. Now, Proposition \ref{Preliminaries: Zygmund H\"older space: smooth decompisition of CH space} tells us that every function in $\cc^{\sv, t}(\mathcal{K}\times \RR^N, (W, \dv)\boxtimes \nabla_{\RR^N})$ can be decomposed into smooth functions. 
\end{remark}
We also need the following two embedding theorems for product Zygmund-H\"older spaces later on in the proof. 
\begin{lemma}(\cite[Corollary 7.5.12]{BS})\label{Preliminaries: Zygmund H\"older space: Corollary 7.5.12 of [BS]}
    For $\sv\in (0,\infty)^{\nu}$ and $t>0$, we have:
    \begin{enumerate}[label=(\roman*)]
        \item The map $F(x, \zeta)\mapsto F(x,0)$ is continuous from $\cc^{\sv, t}(\mathcal{K}\times \RR^N, (W, \dv)\boxtimes \nabla_{\RR^N})\to\cc^{\sv}(\mathcal{K}, (W,\dv))$. 
        \item For $x_0\in \mathcal{M},$ the map $F(x, \zeta)\mapsto F(x_0, \zeta)$ is continuous from $\cc^{\sv, t}(\mathcal{K}\times \RR^N, (W, \dv)\boxtimes \nabla_{\RR^N}) \to \cc^{t}(\RR^N)$.
    \end{enumerate}
\end{lemma}
\begin{lemma}\label{Preliminaries: Functions spaces: Zygmund-Holder space: Prop 7.5.11} (\cite[Proposition 7.5.11]{BS})
    For $\sv\in (0,\infty)^{\nu}, \beta\in \NN^{\nu}$ a multi-index and $t>|\beta|$. Then, $F(x, \zeta)\mapsto \partial_{\zeta}^{\beta}F(x, \zeta)$ is a continuous map from $\cc^{\sv,t}(\mathcal{K}\times \RR^N, (W, \dv)\boxtimes \nabla_{\RR^N})\to \cc^{\sv,t-|\beta|}(\mathcal{K}\times \RR^N, (W, \dv)\boxtimes \nabla_{\RR^N})$
\end{lemma}
\subsubsection{Sobolev space}
In this section, we define the Sobolev spaces and prove a couple of technical lemmas required for our computation. See \cite[Section 6.9]{BS} for a detailed exposition of these non-isotropic Sobolev spaces. We will provide some basic properties of these space here. 
\begin{definition}
    Let $\mathcal{K}\subset \mathcal{M}$ be a compact set. Then, for $1<p<\infty, \sv\in \RR^{\nu}$ and $(W, \dv)$ as in Section \ref{Main multi-parameter theorem} we define $L^p_{\sv}(\mathcal{K}, (W, \dv)):= \mathcal{F}^{\sv}_{p, 2}(\mathcal{K}, (W, \dv))$ (see Notation \ref{Function space: Notation for Besov and Triebel-Lizorkin space}). 
\begin{remark}\label{Preliminaries: Function space: Sobolev space: remark about classical space}
     When $\nu=1$ and $\{(W_1, \db_1), ..., (W_r, \db_r)\}:=\{(\partial_{x_1},1),..., (\partial_{x_n},1)\}$, $L^p_s(\mathcal{K}, (W, \db))$ is the classical Sobolev space.
\end{remark}
    Let $\psi$ and $\mathcal{D}$ be as in Remark \ref{Function spaces: remark about explicit norm: remark about explicit norm part 2}. Then, for $f\in L^p_{\sv}(\mathcal{K}, (W, \dv))$ we define the norm 
    \begin{align}
        \|f\|_{L^p_{\sv}(W, \dv)}:= \|f\|_{\mathcal{F}_{p, 2}^{\sv}(W, \dv)}= \|\{2^{j\cdot \sv}D_j f\}_{j\in \NN^{\nu}}\|_{L^{p}(\mathcal{M};Vol; l^2(\NN^{\nu}))}.
    \end{align}
\end{definition}

 To get a sense for these Sobolev spaces, consider the single parameter setting $\nu=1$ with $(W, \db)= (W_1, \db_1), ..(W_r, \db_r)$ satisfies H\"ormander condition. Let $\kappa \in \NN_+$ such that $\db_j$ divides $\kappa$ for $1\leq j \leq r$, one can show that (see \cite[Corolary 6.2.14, Theorem 8.3.3]{BS})
 \begin{align}\label{Preliminaries: Sobolev space: lemma 8.3.3 (i) of [BS]}
     \|f\|_{L^{p}_{\kappa}(W, \db)} \approx \sum_{\deg_{\db}(\alpha)\leq \kappa}\|W^{\alpha} f\|_{L^{p}}\approx \|f\|_{L^p}+ \sum_{j=1}^r\|W_{j}^{n_j}f\|_{L^{p}}.
 \end{align}

  The next theorem allows one to control $L^p$ Sobolev norm by Zygmund-H\"older norms under sufficient conditions.
\begin{lemma}\label{Preliminaries: embedding sobolev space in Zygmund Holder space}
    Let $\mathcal{K}$ be a compact set in $\RR^n$ and let $f\in L^{p}_{\sv}(W,\dv)$, and $\text{supp}(f)\subset \mathcal{K}$. Then, there exists $\vec{t}_0\in (0,\infty)^{\nu}$ such that 
    \begin{align*}
        \|f\|_{L^{p}_{\sv}(W,\dv)}\lesssim \|f\|_{\cc^{\vec{t}}(W,\dv)},\ \forall \vec{t}\geq \vec{t}_0.
    \end{align*}
In particular, if the right hand-side is finite, the left hand-side is also finite. 
    \begin{proof}
        Since $f$ is compactly supported, using the definition of $L^{p}_{\sv}(W,\dv)$ it is easy to see that 
        \begin{align*}
            \|f\|_{L^{p}_{\sv}(W,\dv)}= \left\|\left(\sum_{j\in \NN^{\nu}}|2^{j\cdot \sv} D_{j} f|^2\right)^{1/2}\right\|_{L^{p}(\mathcal{K})}&\lesssim \left\|\left(\sum_{j\in \NN^{\nu}}\|2^{j\cdot \sv} D_{j} f\|_{L^{\infty}}^2\right)^{1/2}\right\|_{L^{p}(\mathcal{K})} \\
            &\lesssim \left(\sum_{j\in \NN^{\nu}}\|2^{-|j|_{1}}\{2^{j\cdot (\sv+(1,...,1))} D_{j} f\|_{L^{\infty}}\}_{l^{\infty}}^2\right)^{1/2}\\
            &\lesssim \|f\|_{\cc^{\sv+(1,...,1)}(W,\dv)}.
        \end{align*}
        So, we get the intended inequality by taking $
    \vec{t}= \sv+(1,...,1)$. 
    \end{proof}
\end{lemma}
Next, we will decompose elements of $L^{p}_{\sv}(\mathcal{K}, (W, \dv))$ using smooth functions as long as $\sv\in (0,\infty)^{\nu}$. 
\begin{prop}\label{Preliminaries: smooth function decomposition for Sobolev space}
   Let $1<p<\infty, \sv\in (0,\infty)^{\nu}$. Then, for $f\in C_{0}^{\infty}(\mathcal{M})'$, the following are equivalent.
   \begin{enumerate}
   \item $f\in L^{p}_{\sv}(\mathcal{K}, (W ,\dv))$.
       \item $supp(f)\subseteq \mathcal{K}$ and there exists a sequence $\{f_{j}\}\subset C_{0}^{\infty}(\mathcal{M})$ such that for every ordered multi-index $\alpha$
       \begin{align*}
       \|\{2^{j\cdot \sv} (2^{-j\dv}W)^{\alpha}f_j\}_{j\in \NN^{\nu}}\|_{L^{p}(\mathcal{M},Vol; l^2(\NN^{\nu}))}<\infty 
       \end{align*}
       and $f=\sum_{j\in \NN^{\nu}}f_j$ with convergence in $L^p(\mathcal{M},Vol)$.

    In this case, there exists $M=M(\sv, \nu)\in \NN$ such that 
    \begin{align}\label{Preliminiaries: Function spaces: Sobolev space: decomposition into smooth functions equation part (2) equation 2}
        \|f\|_{L^{p}_{\sv}(W,\dv)} \leq C\sum_{|\alpha|\leq M} \left\|\left(\sum_{j\in \NN^{\nu}} |2^{j\cdot \sv} (2^{-j\cdot \dv}W)^{\alpha} f_j |^2\right)^{1/2}\right\|_{L^{p}(\mathcal{M})}.
    \end{align}
    Furthermore, if $\mathcal{K}\subset \Omega\Subset \mathcal{M}$, $f_j$ may be chosen such that $f_j\in C_{0}^{\infty}, (\Omega)$, for every $j\in \NN^{\nu}$, and for every ordered multi-index $\alpha$, 
    \begin{align}\label{Preliminiaries: Function spaces: Sobolev space: decomposition into smooth functions equation part (2) equation 3}
       \left \|\left(\sum_{j\in \NN^{\nu}} |2^{j\cdot \sv} (2^{-j\cdot \dv}W)^{\alpha} f_j |^2\right)^{1/2}\right\|_{L^{p}(\mathcal{M})} \leq C_{\alpha} \|f\|_{L^{p}_{\sv}(W, \dv)}. 
    \end{align}
    Here, $C$ and $C_{\alpha}$ do not depend on $f$. 
   \end{enumerate}
\end{prop}
\begin{proof}
   Suppose (1) holds, i.e $f\in L^{p}_{\sv}(\mathcal{K}, (W, \dv))$. Then, by definition $supp(f)\subseteq \mathcal{K}$. Let $D_j$ be as before. Then, we set $f_j:=D_j f$. By proposition \ref{Preliminaries: Function spaces: Proposition 7.1.1 and Corollary 7.1.4} (i) we have $D_{j}\in C_0^{\infty}(\Omega\times \Omega)$ and therefore $f_j\in C_0^{\infty}(\Omega)$. We know that $\sum_{j\in \NN^{\nu}}f_j$ converges to $f$ in $C_{0}^{\infty}(\mathcal{M}, \mathcal{M})'$. So, it is sufficient to show that $\sum_{j\in \NN^{\nu}}f_j$ converges in $L^{p}(\mathcal{M},Vol)$. Since $\sv\in (\infty)^{\nu}$, 
   \begin{align*}
       \left\| \sum_{j\in \NN^{\nu}} f_{j}\right\|_{L^{p}(\mathcal{M}, Vol)}\leq \left\|\left(\sum_{j\in \NN^{\nu}} |f_j|^2\right)^{1/2}\right\|_{L^p(\mathcal{M}, Vol)}&\leq \left\|\left(\sum_{j\in \NN^{\nu}} |2^{j\cdot \sv}f_j|^2\right)^{1/2}\right\|_{L^p(\mathcal{M}, Vol)}\\&\lesssim \|f\|_{L^{p}_{\sv}(\mathcal{K}, (W, \dv))}.
   \end{align*}
   Hence $\sum_{j\in \NN^{\nu}}f_j$ converges to $f$ in $L^{p}(\mathcal{M}, Vol)$. Let $\alpha$ be an ordered multi-index and set $E_{j, \alpha}:= (2^{-j\dv}W)^{\alpha} D_j$. Since $\{(D_j, 2^{-j}: j\in \NN^{\nu})\}$ is a bounded set of $(w, \dv)$-generalized operators supported in $\Omega$, Proposition \ref{Preliminaries: Function spaces: properties of elementary operators} part (c), $\mathcal{E}_{\alpha}:=\{(E_{j, \alpha}, 2^{-j}): j\in \NN^{\nu}\}$ is a also a bounded set of generalized $(W, \dv)$ elementary operators supported in $\Omega$. Thus, with the help of Remark \ref{Function spaces: remark about explicit norm: Equivalence of norms}, we have 
   \begin{align*}
              \left \|\left(\sum_{j\in \NN^{\nu}} |2^{j\cdot \sv} (2^{-j\cdot \dv}W)^{\alpha} f_j |^2\right)^{1/2}\right\|_{L^{p}(\mathcal{M})} &=  \left\|\left(\sum_{j\in \NN^{\nu}} |2^{j\cdot \sv} E_{j, \alpha} f|^2\right)^{1/2}\right\|_{L^p(\mathcal{M})}\\
               &= \|f\|_{L^{p}(\mathcal{M},Vol; l^2(\NN^{\nu})),\sv, \mathcal{E}_{\alpha}}\lesssim \|f\|_{L^{p}_{\sv}(W, \dv)}.
\end{align*}
Thus (2) holds, and moreover we also have \eqref{Preliminiaries: Function spaces: Sobolev space: decomposition into smooth functions equation part (2) equation 3}. 
\begin{notation}
    For $a,b \in \NN^{\nu}$, we say that $a\leq b$ if $a_{\mu}\leq b_{\mu}$ for all $\mu\in \{1,..., \nu\}$. 
\end{notation}
   Suppose that (2) holds. Fix $M\geq \nu (|\sv|_{\infty}+1)$. From here on, when we write $L^{p}(\mathcal{M})$ we mean $L^{p}(\mathcal{M}, Vol)$ unless specified. Using the definition 
   \begin{align*}
       \|f\|_{L^{p}_{\sv}(W, \dv)} = \left\|\left(\sum_{j\in \NN^{\nu}} |2^{j\cdot \sv} D_j f|^2\right)^{1/2}\right\|_{L^p(\mathcal{M})}.
   \end{align*}
   We had also assumed that $f=\sum_{k\in \NN^{\nu}}$ with convergence in $L^{p}(\mathcal{M})$. Hence, we can find a subsequence $\{N_{l}\}_{l\in \NN}, N_{l}\in \NN^{\nu}$ such that $\sum_{k\in \NN^{\nu}} f_{k}$ converges to $f$ a.e as $l\to \infty$. Hence, by Fatou's lemma
   \begin{align}\label{Preliminaries: Function spaces: Sobolev space: Decomposition into smooth functions: Fatou's lemma}
       \left\|\left(\sum_{j\in \NN^{\nu}} |2^{j\cdot \sv} D_j f|^2\right)^{1/2}\right\|_{L^p(\mathcal{M})} \leq \underset{l\to \infty}{liminf} \left\|\left(\sum_{j\in \NN^{\nu}} |2^{j\cdot \sv} \sum_{0\leq k\leq N_{l}}D_j f_{k}|^2\right)^{1/2}\right\|_{L^p(\mathcal{M})}.
   \end{align}
   We will later on take the sum in $k$ outside using triangle inequality. So, by abuse of notation just write the right hand side in \eqref{Preliminaries: Function spaces: Sobolev space: Decomposition into smooth functions: Fatou's lemma} as
   \begin{align}\label{Preliminaries: Function spaces: Sobolev space: Decomposition into smooth functions: abuse of notation after Fatou's lemma}
          \left\|\left(\sum_{j\in \NN^{\nu}} |2^{j\cdot \sv} \sum_{k\in \NN^{\nu}}D_j f_{k}|^2\right)^{1/2}\right\|_{L^p(\mathcal{M})}. 
   \end{align}
\begin{notation}
    Given $j, k\in \ZZ^{\nu}$, we make the following notation
    \begin{enumerate}[label=(\roman*)]
        \item $|j|_{\infty}= \underset{\mu\in \{1,.., \nu\}}{\sup}|j_{\mu}|$, $|j|_1= \sum_{1}^{\nu}|j_{\mu}|$.
        \item $j\wedge k:=l \in \NN^{\nu}$ such that $l_{\mu}= \min \{j_{\mu}, k_{\mu}\}, \forall \mu \in \{1,..., \nu\}$ and  $j\vee k:=m \in \NN^{\nu}$ such that $m_{\mu}= \max \{j_{\mu}, k_{\mu}\}, \forall \mu \in \{1,..., \nu\}$.
    \end{enumerate}
\end{notation}
We will now split the sum in $k$ in \eqref{Preliminaries: Function spaces: Sobolev space: Decomposition into smooth functions: abuse of notation after Fatou's lemma} into two sums.
\begin{align}\label{Preliminaries: Function spaces: Sobolev space: Decomposition into smooth functions: splitting the sum in k}
     \left\|\left(\sum_{j\in \NN^{\nu}} |2^{j\cdot \sv}\sum_{|j \vee k-k|>0} \sum_{|j\vee k -k|=0}  D_j f_{k}|^2\right)^{1/2}\right\|_{L^p(\mathcal{M})}.
\end{align}
For the sum with $|j \vee k-k|>0$, we see that we there exists $\mu \in \{1,..., \nu\}$ with $|j\vee k -k|_{\infty}= j_{\mu}-k_{\mu}>0$. Fix this $\mu$ and pick $M\geq \nu(|\sv|_{\infty}+1)$. Now, using Theorem \ref{Preliminaries: Function spaces: properties of elementary operators} part (d) we write 
\begin{align*}
    D_{j}= \sum_{|\alpha_{\mu}|\leq M} 2^{j_{\mu}(|\alpha_{\mu}-M|)}E_{j, \mu, \alpha_{\mu}} (2^{-j_{\mu}\db^{\mu}}W^{\mu})^{\alpha_{\mu}}. 
\end{align*}
Observe that 
\begin{align*}
   2^{j_{\mu}(|\alpha_{\mu}-M|)}E_{j, \mu, \alpha_{\mu}} (2^{-j_{\mu}\db^{\mu}}W^{\mu})^{\alpha_{\mu}}= 2^{j\cdot \sv} 2^{(|\alpha_{\mu}|-M)j_{\mu}+ (k_{\mu}-j_{\mu})\deg_{\db^{\mu}}(\alpha_{\mu})} E_{j, \mu, \alpha_{\mu}} (2^{-k_{\mu\db^{\mu}}}W^{\mu})^{\alpha_{\mu}}.
\end{align*}
Using $|\alpha_{\mu}|\leq M$ and $\deg_{\db^{\mu}}(\alpha_{\mu})\geq |\alpha_{\mu}|$ we get the bound
\begin{align}\label{Preliminaries: Function spaces: Sobolev space: Decomposition into smooth functions:rearrangin j and k}
  2^{j\cdot \sv} 2^{(|\alpha_{\mu}|-M)j_{\mu}+ (k_{\mu}-j_{\mu})\deg_{\db^{\mu}}(\alpha_{\mu})} \leq 2^{j\cdot \sv} 2^{-M(j_{\mu}-k_{\mu})}&= 2^{j\cdot \sv} 2^{-M|j \vee k-k|_{\infty}}\nonumber \\
  &=  2^{(j+k)\cdot \sv} 2^{-k \cdot \sv}2^{-M|j \vee k-k|_{\infty}}. 
\end{align}
Using \eqref{Preliminaries: Function spaces: Sobolev space: Decomposition into smooth functions:rearrangin j and k} we can bound \eqref{Preliminaries: Function spaces: Sobolev space: Decomposition into smooth functions: splitting the sum in k} by 
\begin{align}\label{Preliminaries: Function spaces: Sobolev space: Decomposition into smooth functions: required quanity after applying j and k rearrangement}
      \Bigg\|\Bigg(\sum_{j\in \NN^{\nu}} \Big| &\sum_{|j\vee k -k|=0}2^{(j+k)\cdot \sv} 2^{-k \cdot \sv}2^{-M|j \vee k-k|_{\infty}}  D_j f_{k}\\&+\sum_{|j \vee k-k|>0}2^{(j+k)\cdot \sv}\sum_{|\alpha_{\mu}|\leq M} 2^{-k \cdot \sv}2^{-M|j \vee k-k|_{\infty}}  E_{j, \mu, \alpha_{\mu}} (2^{-k_{\mu}\db^{\mu}}W^{\mu})^{\alpha_{\mu}} f_k\Big|^2\Bigg)^{1/2}\Bigg\|_{L^p(\mathcal{M})}.
\end{align}
Observe that in the sum with $|j\vee k -k|$ we introduced a term $2^{-M|j\vee k -k|}$ as this is equal to 1. We then make a change of variable $k\mapsto k+j$, where we use the notation $f_k0$ for $k$ such that $\exists\  \theta \in \{1, ..., \nu\}$ with $k_{\theta}<0$. So, \eqref{Preliminaries: Function spaces: Sobolev space: Decomposition into smooth functions: required quanity after applying j and k rearrangement} becomes 
\begin{align}\label{Preliminaries: Function spaces: Sobolev space: Decomposition into smooth functions: after k replaced by k+j}
      \Bigg\|\Bigg(\sum_{j\in \NN^{\nu}} \Big| &\sum_{|j\vee (k+j) -(k+j)|=0}2^{(j+k)\cdot \sv} 2^{-k \cdot \sv}2^{-M|j \vee (k+j)-(k+j)|_{\infty}}  D_j f_{k+j}+\\&\sum_{|j \vee (k+j)-(k+j)|>0}\sum_{|\alpha_{\mu}|\leq M}2^{(j+k)\cdot \sv} 2^{-k \cdot \sv}2^{-M|j \vee (k+j)-(k+j)|_{\infty}}  E_{j, \mu, \alpha_{\mu}} (2^{-(k+j)_{\mu}\db^{\mu}}W^{\mu})^{\alpha_{\mu}} f_{k+j}\Big|^2\Bigg)^{1/2}\Bigg\|_{L^p(\mathcal{M})}\nonumber.
\end{align}
Since $M\geq \nu(|\sv|_{\infty}+1)$ we have the following estimate 
 \begin{align} \label{power of two estimate}
 2^{-k\cdot\sv} 2^{-M|j\vee (k+j)-(k+j)|_{\infty}}&\leq 2^{-k\cdot \sv} 2^{-(|\sv|_{\infty}+1) |j\vee (k+j)-(k+j)|_1}\nonumber\\&\leq  \prod_{\{\theta: k_{\theta}<0\}}2^{-(|\sv|_{\infty}+1)(-k_{\theta})}2^{-k_{\theta}s_{\theta}} \prod_{\{\theta: k_{\theta}\geq 0\} } 2^{-k_{\theta}s_{\theta}}\nonumber \\&\leq \prod_{\theta\in \{1,...,\nu\}} 2^{-|k_{\theta}|(1\wedge s_{\theta})}.
 \end{align}
 Applying \eqref{power of two estimate}, replacing both the sum $\sum_{|j\vee (k+j) -(k+j)|=0}$ and $\sum_{|j\vee (k+j) -(k+j)|>0}$ with $\sum_{k\in \ZZ^{\nu}}$ and triangle inequality to  \eqref{Preliminaries: Function spaces: Sobolev space: Decomposition into smooth functions: after k replaced by k+j} we get the bound 
 \begin{align}
       &\sum_{k\in \ZZ^{\nu}}\prod_{\theta\in \{1,...,\nu\}} 2^{-|k_{\theta}|(1\wedge s_{\theta})}\Bigg\|\Bigg(\sum_{j\in \NN^{\nu}} \Big| 2^{(j+k)\cdot \sv}  D_j f_{k+j}\Big|^2\Bigg)^{1/2}\Bigg\|_{L^p(\mathcal{M})}+\nonumber\\
       &\sum_{|\alpha_{\mu}|\leq M}\sum_{k\in \ZZ^{\nu}}\prod_{\theta\in \{1,...,\nu\}} 2^{-|k_{\theta}|(1\wedge s_{\theta})}\Bigg\|\Bigg(\sum_{j\in \NN^{\nu}}\Big|2^{(j+k)\cdot \sv}  E_{j, \mu, \alpha_{\mu}} (2^{-(k+j)_{\mu}\db^{\mu}}W^{\mu})^{\alpha_{\mu}} f_{k+j}\Big|^2\Bigg)^{1/2}\Bigg\|_{L^p(\mathcal{M})}\nonumber\\
       &=(I) + (II).
 \end{align}
Since $\{D_j\}_{j\in \NN^{\nu}}$ and $\{E_{j, \mu, \alpha_{\mu}}\}_{j\in \NN^{\nu}}$ are $(W, \dv)$-generalized elementary operators, by Definition \ref{Preliminaries: Function spaces: Definiton of pre-elementary operators} we get that 
\begin{align*}
    \Bigg\|\Bigg(\sum_{j\in \NN^{\nu}} \Big| 2^{(j+k)\cdot \sv}  D_j f_{k+j}\Big|^2\Bigg)^{1/2}\Bigg\|_{L^p(\mathcal{M})} &\lesssim \Bigg\|\Bigg(\sum_{j\in \NN^{\nu}} \Big| 2^{(j+k)\cdot \sv}   f_{k+j}\Big|^2\Bigg)^{1/2}\Bigg\|_{L^p(\mathcal{M})}\\
    &\leq \Bigg\|\Bigg(\sum_{j\in \NN^{\nu}} | 2^{j\cdot \sv}   f_{j}|^2\Bigg)^{1/2}\Bigg\|_{L^p(\mathcal{M})}\\
      &\leq \sum_{|\alpha|\leq M} \Bigg\|\Bigg(\sum_{j\in \NN^{\nu}}2^{j\cdot \sv}  | (2^{-j\dv}W)^{\alpha} f_{j}|^2\Bigg)^{1/2}\Bigg\|_{L^p(\mathcal{M})}
\end{align*}
and 
\begin{align*}
  &\sum_{|\alpha_{\mu}|\leq M}\Bigg\|\Bigg(\sum_{j\in \NN^{\nu}}\Big|2^{(j+k)\cdot \sv}  E_{j, \mu, \alpha_{\mu}} (2^{-(k+j)_{\mu}\db^{\mu}}W^{\mu})^{\alpha_{\mu}} f_{k+j}\Big|^2\Bigg)^{1/2}\Bigg\|_{L^p(\mathcal{M})} \\&\lesssim \
     \sum_{|\alpha_{\mu}|\leq M} \Bigg\|\Bigg(\sum_{j\in \NN^{\nu}}2^{(j+k)\cdot \sv}   \Big|(2^{-(k+j)_{\mu}\db^{\mu}}W^{\mu})^{\alpha_{\mu}} f_{k+j}\Big|^2\Bigg)^{1/2}\Bigg\|_{L^p(\mathcal{M})}\\
      &\leq \sum_{|\alpha|\leq M} \Bigg\|\Bigg(\sum_{j\in \NN^{\nu}}2^{j\cdot \sv}  | (2^{-j\dv}W)^{\alpha} f_{j}|^2\Bigg)^{1/2}\Bigg\|_{L^p(\mathcal{M})}.
\end{align*}
Since $\sum_{k\in \ZZ^{\nu}}\prod_{\theta\in \{1,...,\nu\}} 2^{-|k_{\theta}|(1\wedge s_{\theta})}<\infty$, we get
\begin{align*}
    (I)+(II) \lesssim  \sum_{|\alpha|\leq M} \Bigg\|\Bigg(\sum_{j\in \NN^{\nu}}2^{j\cdot \sv}  | (2^{-j\dv}W)^{\alpha} f_{j}|^2\Bigg)^{1/2}\Bigg\|_{L^p(\mathcal{M})}.
\end{align*}
\begin{remark}\label{Preliminaries: Function spaces: Sobolev space: remark about the proof technique in tame estimate for the sobolev space}
    The proof technique we used to prove (2) $\implies$ (1) will also be employed in the proof of Proposition \ref{tame estimate for composition subellitpic case proposition}. The proof technique is easier to understand from the perspective of Proposition \ref{Preliminaries: smooth function decomposition for Sobolev space} than Proposition \ref{tame estimate for composition subellitpic case proposition}.
\end{remark}
 \end{proof}

 \section{Tame estimate for Sobolev space}\label{Tame estimate for Sobolev space}
\numberwithin{equation}{section}
 
Tame estimates form an important tool in the analysis of analysis of nonlinear PDE. We will estimate the multi-parameter Sobolev norm of composition of two functions using the individual norms of the two functions. In particular, we prove a version of the following theorem 
 \begin{theorem}\label{Appendix B: Theorem 7.5.2 of BS}(\cite[Theorem 7.5.2]{BS})
    Fix $\sv\in (0,\infty)^{\nu}$ and $t>0$. Let $F\in \cc^{\sv, t} (\mathcal{K}\times \RR^N, (W, \dv)\boxtimes \nabla_{\RR^N})$, $u, v \in \cc^{\sv}(\mathcal{K}, (W, \dv);\RR^N)$, and $w\in \cc^{\sv}(\mathcal{K}, (W, \dv))$. Then:
    \begin{enumerate}[label=(\roman*)]
        \item Fix $L\in \NN_+$ and suppose $\partial_{\zeta}^{\beta}F(x,0)\equiv 0$, $\forall |\beta|<L$ and $t> \lfloor \sv|_1\rfloor+1+\nu +L$. Then
        \begin{align*}
            \|F(x,u(x))\|_{\cc^{\sv}(W, \dv)} \lesssim \|F\|_{\cc^{\sv, t} ((W, \dv)\boxtimes \nabla_{\RR^N})} (1+\|u\|_{L^{\infty}})^{\lfloor |\sv|_1+\nu+1\rfloor}\|u\|_{L^{\infty}}^{L-1} \|u\|_{\cc^{\sv}(W, \dv)}.
        \end{align*}
        \item Fix $L\in \NN, L\geq 2$ and suppose $\partial_{\zeta}^{\beta}F(x,0)\equiv 0, \forall 0< |\beta|< L$ and $t> \lfloor |\sv|_1+1+\nu+L\rfloor$. Then 
        \begin{align*}
            \|&F(x, u(x))-F(x, v(x))\|_{\cc^{\sv}(W, \dv)}\\
            &\lesssim \|F\|_{\cc^{\sv,t}((W, \dv)\boxtimes \nabla_{\RR^N})} (1+\|u\|_{L^{\infty}}+\|v\|_{L^{\infty}})^{\lfloor |\sv|_1\rfloor+\nu+1} (\|u\|_{L^{\infty}}+\|v\|_{L^{\infty}})^{L-2}\\
            &\times (\|u\|_{\cc^{\sv}(W, \dv)}+\|v\|_{\cc^{\sv}(W, \dv)})\|u-v\|_{L^{\infty}}\\
            &+ \|F\|_{\cc^{\sv, t}((W, \dv)\boxtimes \nabla_{\RR^N})} (\|u\|_{L^{\infty}}+\|v\|_{L^{\infty}})^{L-1} \|u-v\|_{\cc^{\sv}(W, \dv)}. 
        \end{align*}
    \end{enumerate}
\end{theorem}
 The following Proposition and its corollaries are one of the most important estimates for proving the regularity theorem. 
\begin{prop}\label{tame estimate for composition subellitpic case proposition}
	Fix $\sv\in (0,\infty)^{\nu}$ and let $M\in \NN$ such that $0< \nu(|\sv|_{\infty}+1) \leq M$. Let $F\in C^{\sv,M}(\mathcal{K}\times \mathbb{R}^N, (W,\dv)\boxtimes \nabla_{\RR^N})$, where $\mathcal{K}$ is a compact set in $\mathcal{M}$. Then, for any $1<p<\infty$, and $u\in L^{p}_{\sv}(\mathcal{K}, (W,\dv); \RR^N)\bigcap L^{\infty}(\RR^N)$ we have
	\begin{align}\label{tame estimate for composition ellitpic case}
	\|K\|_{L^{p}_{\sv}(\mathbb{R}^n,(W,\dv))L^{\infty}(\RR^N)} \lesssim ||F||_{\mathcal{C}^{\vec{s},M}((W,\dv)\boxtimes \nabla_{\RR^N}) }(1+||u||_{L^{p}_{\sv}(W,\dv)}) (1+||u||_{L^{\infty}})^{M+\nu-1},
	\end{align}
 where $K(x,\zeta)=F(x,u(x)+\zeta)$. The implicit constant in the inequality do not depend on $F,u$, but may depend on other ingredients in the result. 
	\end{prop}
  \begin{remark}\label{Tame estimate for sobolev space: remark about embedding to space of continuous functions}
 In this section we will freely use the fundamental theorem of calculus. This is justified using the fact that for $\rv, \sv\in(0,\infty)$ and $t>0$,  we have a continuous inclusions $\cc^{\rv}(\mathcal{K}, (W, \dv)) \subset C(\mathcal{M})$(see Remark \ref{Preliminaries: Function spaces: Zygmund Holder space: Inclusion into continuous functions}) and $\cc^{\sv, t}(\mathcal{K}\times \RR^N,(W, \dv)\boxtimes \nabla_{\RR^N})\subset C(\mathcal{M}\times\RR^N)$(this is an easy corollary of Proposition \ref{Preliminaries: Zygmund H\"older space: smooth decompisition of CH space} part (ii)). We will also use the same proof technique as in Proposition \ref{Preliminaries: smooth function decomposition for Sobolev space}, as mentioned in Remark \ref{Preliminaries: Function spaces: Sobolev space: remark about the proof technique in tame estimate for the sobolev space}. The details are easier to understand in Proposition \ref{Preliminaries: smooth function decomposition for Sobolev space} as the the setting is less intricate over in that setting. 
 \end{remark}
\begin{proof}
 Fix $\zeta \in \RR^N$ and we will derive estimates independent of $\zeta$ to get $L^{\infty}$ bound in the $\zeta$ variable. We will use the notation that $\NN$ contains $0$. By abuse of notation we will also denote $\|f\|_{L^{p}_{\sv}(\mathcal{K}, (W, \dv))}$ by $\|f\|_{L^{p}_{\sv}(W, \dv)}$, whenever $\mathcal{K}=supp(f)$. By definition 
\begin{align}\label{sobolev norm maximally subelliptic} \|F(x,u(x)+\zeta)\|_{L^{p}_{\sv}(W,\dv)}& = \left\|\left(\sum_{j\in \NN^{\nu}} \left|2^{j\cdot \sv}D_{j} F(x,u(x)+\zeta)\right|^2\right)^{1/2}\right\|_{L^{p}(\RR^n)}\end{align}
Now, we decompose the function $D_{j}F(x,u(+\zeta))$ using a telescoping series that converges in $C(\mathcal{K}\times \RR^N)$. Let $P_{k}:=\sum_{j\leq k} D_{j}$ for $k\in \NN^\nu$. For $E\subset \{1,...,\nu\}$ let $\ef_{E}=\sum_{\mu\in E} e_{\mu}$, where $e_1,..,e_{\nu}$ is the standard basis for $\RR^{\nu}$. \\
Let $k\in \NN^{\nu}\backslash\{0\}$. Without loss of generality, we assume $\{\mu: k_{\mu}\neq 0\}= \{1,..,\nu_0\}$ for some $1\leq \nu_0\leq \nu$. We write $\sigma =(\sigma_1,..,\sigma_{\nu_0})\in \RR^{\nu_0}$. For $E\subseteq \{1,..,\nu_0\}$, we write $\sigma_E:=\prod_{\mu\in E} \sigma_{\mu}$ and $E^{c}:= \{1,...,\nu_0\}\backslash E$. Then, 
\ba D_{j}F(x,u(x)+\zeta)&= \sum_{k\in \NN^{\nu}\backslash \{0\}}\sum_{E\subseteq \{1,..,\nu_0\}} (-1)^{|E|} D_{j}F(x, P_{k- e_E} u(x)+\zeta)\\
&= \sum_{k\in \NN^{\nu}\backslash \{0\}} H_{j,k}(x,\zeta),
\ea
where $H_{j,k}(x,\zeta)= \sum_{E\subseteq \{1,..,\nu_0\}}(-1)^{|E|} D_j F(z, P_{k-\ef_E}u(x)+\zeta)$. Now, by a careful application of the fundamental theorem of calculus the above equals
\ba 
\int_{[0,1]^\nu} \partial_{\sigma_1}\partial_{\sigma_2}...\partial_{\sigma_{\nu_0}} \left(D_j F\left(s,P_{k-\ef_{\{1,...,\nu_0\}}}+\sum_{E\subsetneq\{1,...,\nu_{0}\}}\sigma_{E^c} D_{k-\ef_E}u(x)+\zeta\right)\right)\ d\sigma. 
\ea
Let $1\leq L\leq \nu_{0}$ and $E_1,...,E_L,\tilde{E}_1,..,\tilde{E}_L\subset \{1,...,\nu_0\}$. Then, since $\nu_0\geq 1$
\ba \partial_{\sigma_1}...\partial_{\nu_0} \left(D_j F\left(x, x, P_{k-\ef_{\{1,...,\nu_0\}}}u(x)+\sum_{E\subsetneq \{1,...,\nu_0\}}\sigma_{E^c} D_{k-\ef_E} u(x)+\zeta\right)\right)
\ea
can be written as a finite sum of terms of the form 
\ba
\left(D_{\zeta}^LD_j F\right) \left(x, P_{k-\ef_{\{1,...,\nu_0\}}}u(x)+\sum_{E\subsetneq \{1,...,\nu_0\}} \sigma_{E^c} D_{k-\ef_E} u(x)+\zeta\right) \left[\sigma_{E_1}D_{k-\ef_{\tilde{E}_1}}u,..., \sigma_{E_L}D_{k-\ef_{\tilde{E}_L}}u\right].
\ea
Now, \eqref{sobolev norm maximally subelliptic} can be written as 
\begin{align}\label{sobolev norm maximally subelliptic rewritten} \left\|\left(\sum_{j\in \NN^{\nu}} \left|\sum_{k\in \NN^{\nu}\backslash \{0\}}2^{j\cdot \sv}H_{j,k}(x,\zeta)\right|^2\right)^{1/2}\right\|_{L^{p}(\RR^n)}=\left\|\left(\sum_{j\in \NN^{\nu}} \left|\sum_{|j\vee k -k|_{\infty}=0}\sum_{|j\vee k -k|_{\infty}>0}2^{j\cdot \sv}H_{j,k}(x,\zeta)\right|^2\right)^{1/2}\right\|_{L^{p}(\RR^n)}\end{align}
Consider the term above with $|j\vee k -k|_{\infty}>0$. We can find a $\mu$ (depending on $j$ and $k$) such that $|j\vee k -k|_{\infty}=j_{\mu}-k_{\mu}>0$. Hence, by Theorem \ref{Preliminaries: Function spaces: properties of elementary operators} part (d) we can write $D_j$ as 
\ba D_j= \sum_{|\alpha_{\mu}|\leq M}2^{(|\alpha_{\mu}|-M)j_{\mu}}  \left(2^{-j_{\mu}\db^{\mu}} W^{\mu}\right)^{\alpha_{\mu}}E_{j,\mu,\alpha_{\mu}},\ea
where $\{(E_{j,\mu,\alpha_{\mu}},2^{-j}):\mu\in\{1,...,\nu\}, |\alpha_{\mu}|\leq M\}$ is a bounded set of generalized $(W,\dv)$ elementary operators. Therefore,
\ba 
&\sum_{|j\vee k -k|_{\infty}>0}2^{j\cdot \sv}H_{j,k}(x,\zeta) \\
&\lesssim  \underset{\underset{E_1,..E_L,\tilde{E}_1,..\tilde{E}_L}{1\leq L\leq \nu_{0}}}{sup} \sum_{|j\vee k -k|_{\infty}>0}\sum_{|\alpha_{\mu}|\leq M}2^{j\cdot \sv}2^{(|\alpha_{\mu}|-M)j_{\mu}}\int_{[0,1]^{\nu}}\left(D_{\zeta}^L \left(2^{-j_{\mu}\db^{\mu}} W^{\mu}\right)^{\alpha_{\mu}}\right) \Bigg(\left(E_{j,\mu,\alpha_{\mu}} F\right)\\&\Big(x, P_{k-\ef_{\{1,...,\nu_0\}}}u(x)+\sum_{E\subsetneq \{1,...,\nu_0\}} \sigma_{E^c} D_{k-\ef_E} u(x)+\zeta\Big)\left[\sigma_{E_1}D_{k-\ef_{\tilde{E}_1}}u,..., \sigma_{E_L}D_{k-\ef_{\tilde{E}_L}}u\right]\Bigg) \ d\sigma \\
&\lesssim \underset{\underset{E_1,..E_L,\tilde{E}_1,..\tilde{E}_L}{1\leq L\leq \nu_{0}}}{sup} \sum_{|j\vee k -k|_{\infty}>0}\sum_{|\alpha_{\mu}|\leq M}2^{j\cdot \sv}2^{(|\alpha_{\mu}|-M)j_{\mu}+(k_{\mu}-j_{\mu})\deg_{\db^{\mu}}(\alpha_{\mu})}\int_{[0,1]^{\nu}}\left(D_{\zeta}^L \left(2^{-k_{\mu}\db^{\mu}} W^{\mu}\right)^{\alpha_{\mu}}\right) \\&\Bigg(\left(E_{j,\mu,\alpha_{\mu}}F \right)\Big(x, P_{k-\ef_{\{1,...,\nu_0\}}}u(x)+\sum_{E\subsetneq \{1,...,\nu_0\}} \sigma_{E^c} D_{k-\ef_E} u(x)+\zeta\Big)\left[\sigma_{E_1}D_{k-\ef_{\tilde{E}_1}}u,..., \sigma_{E_L}D_{k-\ef_{\tilde{E}_L}}u\right]\Bigg) \ d\sigma\ea
\ba &\lesssim \underset{\underset{E_1,..E_L,\tilde{E}_1,..\tilde{E}_L}{1\leq L\leq \nu_{0}}}{sup}  \sum_{|j\vee k -k|_{\infty}>0}\sum_{|\alpha_{\mu}|\leq M}2^{j\cdot \sv}2^{(|\alpha_{\mu}|-M)j_{\mu}+(k_{\mu}-j_{\mu})\deg_{\db^{\mu}}(\alpha_{\mu})}\int_{[0,1]^{\nu}}\Big|\left(D_{\zeta}^L \left(2^{-k_{\mu}\db^{\mu}} W^{\mu}\right)^{\alpha_{\mu}}\right) \\&\Bigg(\left(E_{j,\mu,\alpha_{\mu}}F \right)\Big(x, P_{k-\ef_{\{1,...,\nu_0\}}}u(x)+\sum_{E\subsetneq \{1,...,\nu_0\}} \sigma_{E^c} D_{k-\ef_E} u(x)+\zeta\Big)\left[\sigma_{E_1}D_{k-\ef_{\tilde{E}_1}}u,..., \sigma_{E_L}D_{k-\ef_{\tilde{E}_L}}u\right]\Bigg) \Big|\ d\sigma 
\ea
Now, we use the fat that $\deg_{\db^{\mu}}(\alpha_{\mu})\geq |\alpha_{\mu}|$ and $|\alpha_{\mu}|\leq M$ to get the bound
\ba 
& \lesssim \underset{\underset{E_1,..E_L,\tilde{E}_1,..\tilde{E}_L}{1\leq L\leq \nu_{0}}}{sup} \sum_{|j\vee k -k|_{\infty}>0}\sum_{|\alpha_{\mu}|\leq M}2^{j\cdot \sv}2^{-M(j_{\mu}-k_{\mu})}\int_{[0,1]^{\nu}}\Big|\left(D_{\zeta}^L \left(2^{-k_{\mu}\db^{\mu}} W^{\mu}\right)^{\alpha_{\mu}}\right) \\&\Bigg(\left(E_{j,\mu,\alpha_{\mu}}F \right)\Big(x, P_{k-\ef_{\{1,...,\nu_0\}}}u(x)+\sum_{E\subsetneq \{1,...,\nu_0\}} \sigma_{E^c} D_{k-\ef_E} u(x)+\zeta\Big)\left[\sigma_{E_1}D_{k-\ef_{\tilde{E}_1}}u,..., \sigma_{E_L}D_{k-\ef_{\tilde{E}_L}}u\right]\Bigg) \Big|\ d\sigma.
\ea
Now, we make a change of variable $k\mapsto k+j$ in \eqref{sobolev norm maximally subelliptic rewritten}. Therefore, to bound right hand side of \eqref{sobolev norm maximally subelliptic rewritten} it is sufficient to get a bound for the following independent of $1\leq L\leq \nu_0$ and $E_1,..E_L,\tilde{E}_1,...,\tilde{E}_L\subseteq \{1,...,\nu_0\}$:
\begin{align}\label{sobolev norm maximally subelliptic estimate 1}
&\Bigg\|\Bigg(\sum_{j\in \NN^{\nu}} \Bigg|\sum_{|j\vee (k+j) -(k+j)|_{\infty}=0} 2^{-k.\sv} 2^{-M|j\vee (k+j)-(k+j)|_{\infty}}2^{(j+k).\sv}\Big|H_{j,k+j}(x,\zeta) \Big|\nonumber\\&+\sum_{|j\vee (k+j) -(k+j)|_{\infty}>0}\sum_{|\alpha_{\mu}|\leq M}2^{-k\cdot \sv}2^{-M|j\vee (k+j)-(k+j)|_{\infty}}2^{(j+k)\cdot \sv}\times \nonumber\\&\int_{[0,1]^{\nu}}\Big|\left(D_{\zeta}^L \left(2^{-(k+j)_{\mu}\db^{\mu}} W^{\mu}\right)^{\alpha_{\mu}}\right)\Bigg(\left(E_{j,\mu,\alpha_{\mu}}F \right)\Big(x, P_{k+j-\ef_{\{1,...,\nu_0\}}}u(x)+\sum_{E\subsetneq \{1,...,\nu_0\}} \sigma_{E^c} D_{k+j-\ef_E} u(x)+\zeta\Big)\nonumber\\&\left[\sigma_{E_1}D_{k+j-\ef_{\tilde{E}_1}}u,..., \sigma_{E_L}D_{k+j-\ef_{\tilde{E}_L}}u\right]\Bigg) \Big|\ d\sigma\Bigg|^2\Bigg)^{1/2}\Bigg\|_{L^{p}(\RR^n)},
\end{align}
 where we use the convention that $D_{k}=D_0$ if there exists a $\theta\in \{1,..,\nu\}$ such that $k_{\theta}< 0$. 
 After applying \eqref{power of two estimate} and triangle inequality in $\sum_{j\in \NN^{\nu}}$ to \eqref{sobolev norm maximally subelliptic estimate 1} we get the bound 
 \begin{align}\label{splitting into I and II}
\lesssim &\sum_{k\in \ZZ^{\nu}\backslash \{0\}}\prod_{\theta\in \{1,...,\nu\}} 2^{-|k_{\theta}|(1\wedge s_{\theta})}\Bigg[\Bigg\|\Bigg(\sum_{j\in \NN^{\nu}} \Bigg| 2^{(j+k)\cdot\sv}H_{j,k+j}(x,\zeta) \Bigg|^2\Bigg)^{1/2}\Bigg\|_{L^{p}(\RR^n)}+\nonumber \\&\sum_{\mu=1}^{\nu}\sum_{|\alpha_{\mu}|\leq M}\Bigg\|\Bigg(\sum_{j\in \NN^{\nu}}\Bigg|2^{(j+k)\cdot \sv}\!\begin{aligned}[t]&\int_{[0,1]^{\nu}}\Big|\left(D_{\zeta}^L \left(2^{-(k+j)_{\mu}\db^{\mu}} W^{\mu}\right)^{\alpha_{\mu}}\right) \nonumber \\&\Bigg(\left(E_{j,\mu,\alpha_{\mu}}F \right)\Big(x, P_{k+j-\ef_{\{1,...,\nu_0\}}}u(x)+\sum_{E\subsetneq \{1,...,\nu_0\}} \sigma_{E^c} D_{k+j-\ef_E} u(x)+\zeta\Big)\nonumber
\\&\left[\sigma_{E_1}D_{k+j-\ef_{\tilde{E}_1}}u,..., \sigma_{E_L}D_{k+j-\ef_{\tilde{E}_L}}u\right]\Bigg)\Big|\ d\sigma\Bigg|^2\Bigg)^{1/2}\Bigg\|_{L^{p}(\RR^n)}\Bigg]\nonumber 
\end{aligned}
\\ &+\|F\|_{\cc^{\sv,M}((W,\dv)\boxtimes \nabla_{\RR^N})}\nonumber\\=:& \left(\sum_{k\in \ZZ^{\nu}\backslash \{0\}}\prod_{\theta\in \{1,...,\nu\}} 2^{-|k_{\theta}|(1\wedge s_{\theta})}\right)\left((I)+(II)+\|F\|_{\cc^{\sv,M}((W,\dv)\boxtimes \nabla_{\RR^N})}\right)\nonumber \\ \lesssim & (I)+(II)+\|F\|_{\cc^{\sv,M}((W,\dv)\boxtimes \nabla_{\RR^N})}.
 \end{align}
 Observe that we also took a sum in $\mu$ from $1$ to $\nu$ in $(II)$. This is because we had chosen our $\mu$ for $k$ such that $|j\vee(k+j)-(k+j)|_{\infty}= -k_{\mu}>0$. After we obtained the bounds from \eqref{power of two estimate}, the specifics of $\mu$ doesn't matter and hence we take a sum over $\mu$ to get an upper bound. The term $\|F\|_{\cc^{\sv,M}((W,\dv)\boxtimes \nabla_{\RR^N})}$ is to make up for the fact that we have omitted $k=(0,...,0)$ from the summation. We will first get a bound for $(II)$. Observe that 
 \ba 
 \Bigg|\left(D_{\zeta}^L  \left(2^{-(k+j)_{\mu}\db^{\mu}} W^{\mu}\right)^{\alpha_{\mu}}\right) &\Bigg(\left(E_{j,\mu,\alpha_{\mu}}F \right)\Big(x, P_{k+j-\ef_{\{1,...,\nu_0\}}}u(x)+\sum_{E\subsetneq \{1,...,\nu_0\}} \sigma_{E^c} D_{k+j-\ef_E} u(x)+\zeta\Big)\nonumber\\&\left[\sigma_{E_1}D_{k+j-\ef_{\tilde{E}_1}}u,..., \sigma_{E_L}D_{k+j-\ef_{\tilde{E}_L}}u\right]\Bigg)\Bigg| 
 \ea
 \ba 
\leq \sum_{\gamma_0\dotplus \gamma_1\dotplus...\gamma_L=\alpha_{\mu}} &\Bigg|\left(2^{-(k+j)_{\mu}\db^{\mu}} W^{\mu}\right)^{\gamma_0}\\&\left(D_{\zeta}^LE_{j,\mu,\alpha_{\mu}}F \right)\Big(x, P_{k+j-\ef_{\{1,...,\nu_0\}}}u(x)+\sum_{E\subsetneq \{1,...,\nu_0\}} \sigma_{E^c} D_{k+j-\ef_E} u(x)+\zeta\Big)\Bigg|\\
&\times \prod_{t=1}^{L}\left|\left(2^{-(k+j)_{\mu}\db^{\mu}} W^{\mu}\right)^{\gamma_t}D_{j-\ef{\tilde{E}_t}}u\right|=:A\times B. 
 \ea
 The idea is to bound $A$ and $B$ separately and apply Cauchy-Schwarz inequality in $(II)$. Observe that $A$ can be bounded by finitely many terms of the form 
 \begin{align}\label{estimate for A}
&\left| \left(\left(2^{-(k+j)_{\mu}\db^{\mu}} W^{\mu}\right)^{\gamma_0} D_{\zeta}^L E_{j,\mu,\alpha_{\mu}}\right)F \Big(x, P_{k+j-\ef_{\{1,...,\nu_0\}}}u(x)+\sum_{E\subsetneq \{1,...,\nu_0\}} \sigma_{E^c} D_{k+j-\ef_E} u(x)+\zeta\Big) \right| \nonumber\\&\times \left(1+\|u\|_{L^{\infty}}\right)^{|\gamma|_0} \nonumber\\
&\lesssim \left| \left(\left(2^{-(k+j)_{\mu}\db^{\mu}} W^{\mu}\right)^{\alpha_\mu} D_{\zeta}^L E_{j,\mu,\alpha_{\mu}}\right)F\Big(x, P_{k+j-\ef_{\{1,...,\nu_0\}}}u(x)+\sum_{E\subsetneq \{1,...,\nu_0\}} \sigma_{E^c} D_{k+j-\ef_E} u(x)+\zeta\Big) \right| \nonumber\\&\times \left(1+\|u\|_{L^{\infty}}\right)^{M} 
 \end{align}
 Here we also used the fact that $\sigma_{E_t}\in [0,1]$ for all $t\in \{1,..,L\}$ and $|\gamma_0|\leq |\alpha_{\mu}|$. 

 Since $|j-(j-\ef_{\tilde{E}_t)}|_{\infty}\leq 1$, Theorem \ref{Preliminaries: Function spaces: properties of elementary operators} part (a) and (c) implies that 
 \begin{align}\label{elementary operators for A and B}
     \{\left(\left(2^{-(k+j)_{\mu}\db}W^{\mu}\right)^{\gamma_t} D_{j-\ef_{\tilde{E}_t}},2^{-j}\right): j\in \NN^{\nu}\}
 \end{align}
 is a bounded set of generalized $(W,\dv)$ elementary operators. Applying \eqref{elementary operators for A and B} for $t\in\{2,..,L\}$ and Lemma \ref{Preliminaries: Function spaces: Proposition 7.1.1 and Corollary 7.1.4} part (iii) we see that 
 \begin{align}\label{second estimate for B}
     \prod_{t=2}^{L} \left\|\left(2^{-(j+k)_{\mu}\db}W^{\mu}\right)^{\gamma_t} D_{j-\tilde{E}_t}u\right\|_{L^{\infty}} \leq \prod_{t=2}^{L} \|u\|_{\infty} \leq (1+\|u\|_{L^{\infty}})^{\nu-1}
 \end{align}
 Putting \eqref{estimate for A}, \eqref{elementary operators for A and B}, \eqref{second estimate for B} and applying Cauchy-Schwarz inequality to $(II)$ gives us that 
 \ba 
 (II) \lesssim \|F\|_{\cc^{\sv,M}((W,\dv)\boxtimes \nabla_{\RR^N})}(1+||u||_{L^{p}_{\sv}(W,\dv)}) (1+||u||_{L^{\infty}})^{M+\nu-1}.
\ea
(I) is easier to bound using similar ideas. We have 
\ba 
(I) \lesssim \|F\|_{\cc^{\sv,M}((W,\dv)\boxtimes \nabla_{\RR^N})} \left(1+\|u\|_{L^{p}_{\sv}(W,\dv)}\right).
\ea
This completes the proof.
\end{proof}
\begin{remark}
   We also do not claim that the exponents are sharp in Proposition \ref{tame estimate for composition subellitpic case proposition}. It is possible that one can improve the exponents, however, we do not explore that here.  
\end{remark}
The main aim of Proposition \ref{tame estimate for composition subellitpic case proposition} was to prove the following corollary. 
 \begin{corollary}\label{tame estimate for composition maximally subellitpic case proposition corollary}
     Let $F$,$u$ and $M$ be as in Proposition \ref{tame estimate for composition subellitpic case proposition}, then for $v\in \left(L^{p}_{\sv}\bigcap L^{\infty}\right)(\mathcal{K}, (W,\dv); \RR^N)$ and $w\in  \left(L^{p}_{\sv}\bigcap L^{\infty}\right)(\mathcal{K}, (W,\dv))$ we have
     \begin{enumerate}
     \item 
     \ba
     \|F(x,u(x))\|_{L^{p}_{\sv}(\mathbb{R}^n,(W,\dv))} \lesssim ||F||_{\mathcal{C}^{\vec{s},M}((W,\dv)\boxtimes \nabla_{\RR^N}) }(1+||u||_{L^{p}_{\sv}(W,\dv)}) (1+||u||_{L^{\infty}})^{M+\nu-1}.
     \ea
         \item $
         \|u\cdot v\|_{L^{p}_{\sv}(W,\dv)} \lesssim \|u\|_{L^{\infty}}\|v\|_{L^{\infty}}+ ||u||_{L^{p}_{\sv}(W,\dv)}||v||_{L^{\infty}}+ ||u||_{L^{\infty}}||v||_{L^{p}_{\sv}(W,\dv)}$
         \item  \ba 
    & \|F(x,u(x)+\zeta)w(x)\|_{L^{p}_{\sv}(\mathbb{R}^n,(W,\dv))L^{\infty}(\RR^N)} \\&\lesssim \|F\|_{\cc^{\sv,M}((W,\dv)\boxtimes\nabla_{\RR^N})} \left(\|w\|_{L^{\infty}}+\|w\|_{L^{\infty}}\|u\|_{L^{p}_{\sv}(W,\dv)}+\|w\|_{L^{p}_{\sv}(W,\dv)}\right)\left(1+\|u\|_{L^{\infty}}\right)^{M+\nu-1}.
     \ea
         \item \ba 
           & \|F(x,u(x)+\zeta)-F(x,v(x)+\zeta)\|_{L^{p}_{\sv}(W,\dv)L^{\infty}(\RR^N)} \\ &\lesssim \|D_{\zeta}F\|_{\cc^{\sv,M}} \left(\|u-v\|_{L^{\infty}}+ \|u-v\|_{L^{\infty}}(\| u\|_{L^{p}_{\sv}(W,\dv)}+ \| v\|_{L^{p}_{\sv}(W,\dv)})+ \|u-v\|_{L^{p}_{\sv}(W,\dv)}\right)\\
      & \times \left(1+ \|u\|_{L^{\infty}}+\|v\|_{L^{\infty}}\right)^{M+\nu-1}
         \ea
        \item Fix $L\in \NN_{+}$ and suppose $\partial_{\zeta}^{\beta}F(x,0)\equiv 0$ for all $|\beta|< L$, ans suppose $\nu(|\sv|_{\infty}+1)+L \leq M$. Then  
        \begin{align*}
            \|F(x,u(x))\|_{L^{p}_{\sv}(W,\dv)}\lesssim &||F||_{\cc^{\sv,M}((W,\dv)\boxtimes \nabla_{\RR^N})}(1+||u||_{L^{\infty}})^{M+\nu-1}||u||_{L^{\infty}}^{L-1}\\&\times\left(\|u\|_{L^{\infty}}+(1+\|u\|_{L^{\infty}})||u||_{L^{p}_{\sv}(W,\dv)}\right)
        \end{align*}
        \item Fix an $L\in \NN_+$ and suppose that $L\geq 2$, and $\partial_{\zeta}^{\beta}F(x,0)\equiv 0$ for all $\beta$ such that $0< |\beta|<L$ and $\nu(\|\sv\|_{\infty}+1) +L\leq M$, then 
        \ba 
   & \|F(x,u(x))-F(x,v(x))\|_{L^{p}_{\sv}(W,\dv)} \\&\lesssim ||F||_{\cc^{\sv,M}((W,\dv)\boxtimes \nabla_{\RR^N})}(1+|| u||_{L^{\infty}}+||v||_{L^{\infty}})^{M+\nu-2}(|| u||_{L^{\infty}}+||v||_{L^{\infty}})^{L-2}\\&\times\left(\|u\|_{L^{\infty}}+\|v\|_{L^{\infty}}+(1+\| u\|_{L^{\infty}}+\|v\|_{L^{\infty}})(|| u||_{L^{p}_{\sv}(W,\dv)}+||v||_{L^{p}_{\sv}(W,\dv)})\right)\\
    &+ \|F\|_{\cc^{\sv,M}((W,\dv)\boxtimes \nabla_{\RR^N})} (||u||_{L^{\infty}}+||v||_{L^{\infty}})^{L-1}\left\|u-v\right\|_{L^{p}_{\sv}(W,\dv)}
        \ea 
     \end{enumerate}
     The implicit constants in the inequalities do not depend on $F,u,v$ or $w$. 
 \end{corollary}
 \begin{proof}
 \begin{enumerate}
 \item The estimate follows from Proposition \ref{tame estimate for composition subellitpic case proposition} by plugging in $\zeta=0$. 
     \item We first consider the case where $\|u\|_{L^{\infty}}=\|v\|_{L^{\infty}}=1$. Fix $\psi_{0}\in C_0^{\infty}(\RR^n)$ with $\psi_0\equiv 1$ on $\mathcal{K}$ and $\phi(\zeta_1,\zeta_2)\in C_{0}^{\infty}(\RR^N\times \RR^N)$ with $\phi(\zeta_1,\zeta_2)\equiv 1$ if $|\zeta_1|, |\zeta_2|\leq 1$. Set $F(x,\zeta)= \psi_0(x) \phi(\zeta_1, \zeta_2) \zeta_1\cdot\zeta_2$. It is easy to see that $F\in C^{\infty,\infty}(\text{supp}(\psi_0)\times \RR^{2N})$.

     If $\|u\|_{L^{\infty}}=\|v\|_{L^{\infty}}=1$, then $u(x).v(x)= F(x,u(x),v(x))$ and Proposition \ref{tame estimate for composition subellitpic case proposition} implies 
     \begin{align}\label{product nonlinearity}
     \|u \cdot v\|_{L^{p}_{\sv}(W,\dv)}&= \|F(x,u(x),v(x))\|_{L^{p}_{\sv}(W,\dv)}\nonumber\\
     &\lesssim \|F\|_{\cc^{\sv,M}((W,\dv)\boxtimes \nabla_{\RR^N})} \left(1+\|u\|_{L^{p}_{\sv}(W,\dv)}+\|v\|_{L^p_{\sv}(W,\dv)}\right)\nonumber\\
     &\times \left(1+\|u\|_{L^{\infty}}+\|v\|_{L^{\infty}}\right)^{M+\nu-1}\nonumber\\
     &\lesssim 1+ \|u\|_{L^p_{\sv}(W,\dv)} + \|v\|_{L^{p}_{\sv}(W,\dv)}.
     \end{align}
     For general $u,v$(not identically zero) in $\left(L^{p}_{\sv}\bigcap L^{\infty}\right)(\mathcal{K}, (W,\dv); \RR^N)$, we can apply \eqref{product nonlinearity} to $u/\|u\|_{L^{\infty}}$ and $v/\|v\|_{L^{\infty}}$ to get
     \ba 
       \|u \cdot v\|_{L^{p}_{\sv}(W,\dv)} \lesssim \|u\|_{L^{\infty}}\|v\|_{L^{\infty}}+ ||u||_{L^{p}_{\sv}(W,\dv)}||v||_{L^{\infty}}+ ||u||_{L^{\infty}}||v||_{L^{p}_{\sv}(W,\dv)}
     \ea
     \item First we prove the result with $\|w\|_{L^{\infty}}=1$. Let $\phi(\zeta')\in C_{0}^{\infty}(\RR)$ satisfy $\phi\equiv 1$ for $|\zeta'|\leq 1$. Set 
     \ba 
     \tilde{F}(x,\zeta,\zeta'):= F(x,\zeta) \phi(\zeta')\zeta'.
     \ea
     Clearly $\tilde{F}\in C^{\infty,\infty}(\mathcal{K}\times \RR^{N+1})$. When $\|w\|_{L^{\infty}=1}$, we have $F(x,u(x)+\zeta)w(x)=\tilde{F}(x,u(x)+\zeta, w(x)+\zeta')|_{\zeta'=0}$. Hence,
     \begin{align}\label{tame estimate corollary part 2 inequality estimate}
&\|F(x,u(x)+\zeta)w(x)\|_{L^{p}_{\sv}(\mathbb{R}^n,(W,\dv))L^{\infty}(\RR^N)} \nonumber\\&= \|\tilde{F}(x,u(x)+\zeta, w(x)+\zeta')|_{\zeta'=0}\|_{L^{p}_{\sv}(\mathbb{R}^n,(W,\dv))L^{\infty}(\RR^N)}\nonumber\\
& \lesssim \|\tilde{F}(x,u(x)+\zeta, w(x)+\zeta')\|_{L^{p}_{\sv}(\mathbb{R}^n,(W,\dv))L^{\infty}(\RR^{N+1})}\nonumber\\
& \lesssim \|\tilde{F}\|_{\cc^{\sv,M}} \left(1+\|u\|_{L^{p}_{\sv}(W,\dv)}+\|w\|_{L^{p}_{\sv}(W,\dv)}\right)\left(1+\|u\|_{L^{\infty}}+\|w\|_{L^{\infty}}\right)^{M+\nu-1}\nonumber\\
&\lesssim \|\tilde{F}\|_{\cc^{\sv,M}} \left(1+\|u\|_{L^{p}_{\sv}(W,\dv)}+\|w\|_{L^{p}_{\sv}(W,\dv)}\right)\left(1+\|u\|_{L^{\infty}}\right)^{M+\nu-1}
     \end{align}
     For a general $w$ (not identically zero) in $\left(L^{p}_{\sv}\bigcap L^{\infty}\right)(\mathcal{K}, (W,\dv))$, we apply \eqref{tame estimate corollary part 2 inequality estimate} to $w/\|w\|_{L^{\infty}}$ to see that 
     \ba 
    & \|F(x,u(x)+\zeta)w(x)\|_{L^{p}_{\sv}(\mathbb{R}^n,(W,\dv))L^{\infty}(\RR^N)} \\&\lesssim \|\tilde{F}\|_{\cc^{\sv,M}((W,\dv)\boxtimes \nabla_{\RR^{N+1}})} \left(\|w\|_{L^{\infty}}+\|w\|_{L^{\infty}}\|u\|_{L^{p}_{\sv}(W,\dv)}+\|w\|_{L^{p}_{\sv}(W,\dv)}\right)\left(1+\|u\|_{L^{\infty}}\right)^{M+\nu-1}.
     \ea
     \item Using (2) we get that 
      \ba 
     & \|F(x,u(x)+\zeta)-F(x,v(x)+\zeta)\|_{L^{p}_{\sv}(W,\dv)L^{\infty}(\RR^N)} \\
      &= \left\|\int_{0}^{1} D_{\zeta}F(x,\sigma u(x)+(1-\sigma)v(x)+\zeta)(u(x)-v(x))\ d\sigma\right\|_{L^{p}_{\sv}(W,\dv)L^{\infty}(\RR^N)}\\
      &\leq \int_{0}^{1}\left\| D_{\zeta}F(x,\sigma u(x)+(1-\sigma)v(x)+\zeta)(u(x)-v(x))\right\|_{L^{p}_{\sv} (W,\dv)L^{\infty}(\RR^N)}\ d\sigma\\
      &\leq \underset{\sigma\in [0,1]}{\sup}\left\| D_{\zeta}F(x,\sigma u(x)+(1-\sigma)v(x)+\zeta)(u(x)-v(x))\right\|_{L^{p}_{\sv} (W,\dv)L^{\infty}(\RR^N)}\\
      &\lesssim \underset{\sigma\in [0,1]}{\sup}\|D_{\zeta}F\|_{\cc^{\sv,M-1}} \left(\|u-v\|_{L^{\infty}}+ \|u-v\|_{L^{\infty}}\|\sigma u+ (1-\sigma) v\|_{L^{p}_{\sv}(W,\dv)}+ \|u-v\|_{L^{p}_{\sv}(W,\dv)}\right)\\
      & \times \left(1+ \|u\|_{L^{\infty}}+\|v\|_{L^{\infty}}\right)^{M+\nu-1}\\
        &\lesssim \|D_{\zeta}F\|_{\cc^{\sv,M}} \left(\|u-v\|_{L^{\infty}}+ \|u-v\|_{L^{\infty}}(\| u\|_{L^{p}_{\sv}(W,\dv)}+ \| v\|_{L^{p}_{\sv}(W,\dv)})+ \|u-v\|_{L^{p}_{\sv}(W,\dv)}\right)\\
      & \times \left(1+ \|u\|_{L^{\infty}}+\|v\|_{L^{\infty}}\right)^{M+\nu-1}
      \ea
      \item We prove this by induction on $L$. We also assume that $\|u\|_{L^{\infty}}=1$ for now. We begin with the base case $L=1$. When $L=1$, using the fact that $F(x,0)\equiv 0$, we have $F(x,u(x))= F(x,u(x))-F(x,0)$. From here, the result when $L=1$ follows from (3) when $\zeta=0$. 

      We assume (4) for some $L\in \NN_+$, and prove it for $L+1$; thus we assume the hypothesis with $L$ replaced by $L+1$. Using the fact that $F(x,0)=0$ and (1), we have
      \ba 
      \|F(x,u(x))\|_{L^{p}_{\sv}(W,\dv)}&= \|F(x,u(x))-F(x,0)\|\\
      &= \|\int_{0}^{1}D_{\zeta}F(x,\sigma u(x))u(x)\ d\sigma\|_{L^{p}_{\sv}(W,\dv)}\\
      &\leq  \int_{0}^{1}\|D_{\zeta}F(x,\sigma u(x))u(x)\|_{L^{p}_{\sv}(W,\dv)}\ d\sigma\\
      & \leq \underset{\sigma\in [0,1]}{\sup} \|D_{\zeta}F(x,\sigma u(x))u(x)\|_{L^{p}_{\sv}(W,\dv)}\\
      &\leq \underset{\sigma\in [0,1]}{\sup} \|D_{\zeta}F(x,\sigma u(x))\|_{L^{p}_{\sv}(W,\dv)}\|u\|_{L^{\infty}}+ \underset{\sigma\in [0,1]}{\sup}\|D_{\zeta}F(x,\sigma u(x))\|_{L^{\infty}}\|u\|_{L^p_{\sv}(W,\dv)}\\
      & + \underset{\sigma\in [0,1]}{\sup}\|u\|_{L^{\infty}}\|D_{\zeta}F(x,\sigma u(x))\|_{L^{\infty}} \\
      &:=(I)+(II)+(III),
      \ea
      where we applied to (1) to various summands of the matrix multiplication $D_{\zeta} F(x,\sigma u(x))u(x)$. 

      Since $M-1\geq \nu(\|\sv\|_{\infty}+1) $ we can apply (3) for $M-1$ along with Lemma \ref{Preliminaries: Functions spaces: Zygmund-Holder space: Prop 7.5.11} to get
\begin{align} \label{bound for (I) composition corollary 4}
(I)&= \underset{\sigma\in [0,1]}{\sup} \|D_{\zeta}F(x,\sigma u(x))\|_{L^{p}_{\sv}(W,\dv)}\|u\|_{L^{\infty}} \nonumber\\
& \lesssim \underset{\sigma\in [0,1]}{\sup} \Bigg [||D_{\zeta}F||_{\cc^{\sv,M-1}((W,\dv)\boxtimes \nabla_{\RR^N})}(1+||\sigma u||_{L^{\infty}})^{M+\nu-2}||\sigma u||_{L^{\infty}}^{L-1}\nonumber\\&\times\left(\|\sigma u\|_{L^{\infty}}+(1+\|\sigma u\|_{L^{\infty}})||\sigma u||_{L^{p}_{\sv}(W,\dv)}\right)\Bigg] \|u\|_{L^{\infty}}\nonumber\\
& \lesssim ||F||_{\cc^{\sv,M}((W,\dv)\boxtimes \nabla_{\RR^N})}(1+||u||_{L^{\infty}})^{M+\nu-2}||u||_{L^{\infty}}^{L}\left(\|u\|_{L^{\infty}}+(1+\|u\|_{L^{\infty}})||u||_{L^{p}_{\sv}(W,\dv)}\right) \end{align} 
For (II) we use the fact that $\partial_{\zeta}^{\beta}D_{\zeta}F(x,0)\equiv 0,\ \forall |\beta| <L $, the fact that $M> L+1$ and Lemma \ref{Preliminaries: Functions spaces: Zygmund-Holder space: Prop 7.5.11} and Remark \ref{Tame estimate for sobolev space: remark about embedding to space of continuous functions} to see that
\ba 
 \underset{\sigma\in [0,1]}{\sup}\|D_{\zeta}F(x,\sigma u(x))\|_{L^{\infty}} &\lesssim \underset{\sigma\in [0,1]}{\sup}\sum_{|\beta|=L} \left\|\partial_{\zeta}^L D_{\zeta}F(x,\zeta)\right\|\|\sigma u(x)\|_{L^{\infty}}^L\\
 &\lesssim \|F\|_{\cc^{\sv,M}((W,\dv)\boxtimes \nabla_{\RR^N})}\|u\|_{L^{\infty}}^{L}. 
\ea
So we conclude that 
\begin{align}\label{bound for (II) composition corollary 4}
(II)\lesssim \|F\|_{\cc^{\sv,M}((W,\dv)\boxtimes \nabla_{\RR^N})}\|u\|_{L^{\infty}}^{L} \|u\|_{L^p_{\sv}(W,\dv)}. 
\end{align}
 Using the same idea that we used to bound (II) we get 
 \begin{align}\label{bound for (III) composition corollary 4}
     (III) \lesssim \|F\|_{\cc^{\sv,M}((W,\dv)\boxtimes \nabla_{\RR^N})}\|u\|_{L^{\infty}}^{L+1}\lesssim \|F\|_{\cc^{\sv,M}((W,\dv)\boxtimes \nabla_{\RR^N})}\|u\|_{L^{\infty}}^{L}\left(\|u\|_{L^{\infty}}+ \|u\|_{L^{p}_{\sv}(W,\dv)}\right).
 \end{align}
 Combining \eqref{bound for (I) composition corollary 4}, \eqref{bound for (II) composition corollary 4} and \eqref{bound for (III) composition corollary 4} we get the required estimate. 
 \item From (3), we have
 \begin{align*}
     &\|F(x,u(x))-F(x,v(x))\|_{L^p_{\sv}(W,\dv)}\\
     &= \left\|\int_{0}^{1}D_{\zeta}F(x,\sigma u(x)+(1-\sigma)v(x))(u(x)-v(x))\ d\sigma\right\|_{L^{p}_{\sv}(W,\dv)}\\
     &\leq \int_{0}^{1}\left\|D_{\zeta}F(x,\sigma u(x)+(1-\sigma)v(x))(u(x)-v(x))\right\|_{L^{p}_{\sv}(W,\dv)}\ d\sigma\\
     &\leq \underset{\sigma\in [0,1]}{\sup}\left\|D_{\zeta}F(x,\sigma u(x)+(1-\sigma)v(x))(u(x)-v(x))\right\|_{L^{p}_{\sv}(W,\dv)}\\
     &\lesssim  \underset{\sigma\in [0,1]}{\sup}\left\|D_{\zeta}F(x,\sigma u(x)+(1-\sigma)v(x))\right\|_{L^{p}_{\sv}(W,\dv)} \left\|u-v\right\|_{L^{\infty}}\\
     &+ \underset{\sigma\in [0,1]}{\sup}\left\|D_{\zeta}F(x,\sigma u(x)+(1-\sigma)v(x)\right\|_{L^{\infty}} \left\|u-v\right\|_{L^{p}_{\sv}(W,\dv)}\\
     &=: (V)+(VI),
 \end{align*}
 where we applied (3) to various summands of the matrix multiplication $D_{\zeta}F(x,\sigma u(x)+(1-\sigma)v(x))(u(x)-v(x))$. 

 For (V) we have $\partial_{\zeta}^{\beta}\equiv 0, \ \forall |\beta|\leq L-1$, and we may therefore apply (4) with $L$ replaced by $L-1$ and $M$ replaced by $M-1$ to see that 
 \begin{align*}
    & \underset{\sigma\in [0,1]}{\sup}\left\|D_{\zeta}F(x,\sigma u(x)+(1-\sigma)v(x))\right\|_{L^{p}_{\sv}(W,\dv)} \\
   &\lesssim  \underset{\sigma\in [0,1]}{\sup} ||D_{\zeta}F||_{\cc^{\sv,M-1}((W,\dv)\boxtimes \nabla_{\RR^N})}(1+||\sigma u+(1-\sigma)v||_{L^{\infty}})^{M+\nu-2}||\sigma u+(1-\sigma)v||_{L^{\infty}}^{L-2}\\&\times\left(\|\sigma u+(1-\sigma)v\|_{L^{\infty}}+(1+\|\sigma u+(1-\sigma)v\|_{L^{\infty}})||\sigma u+(1-\sigma)v||_{L^{p}_{\sv}(W,\dv)}\right)\\
   &\lesssim ||F||_{\cc^{\sv,M}((W,\dv)\boxtimes \nabla_{\RR^N})}(1+|| u||_{L^{\infty}}+||v||_{L^{\infty}})^{M+\nu-2}(|| u||_{L^{\infty}}+||v||_{L^{\infty}})^{L-2}\\&\times\left(\|u\|_{L^{\infty}}+\|v\|_{L^{\infty}}+(1+\| u\|_{L^{\infty}}+\|v\|_{L^{\infty}})(|| u||_{L^{p}_{\sv}(W,\dv)}+||v||_{L^{p}_{\sv}(W,\dv)})\right),
 \end{align*}
 where the final estimate used Lemma \ref{Preliminaries: Functions spaces: Zygmund-Holder space: Prop 7.5.11}. Hence
 \begin{align}\label{Tame estimate Corollary estimate for (V)}
     (V) \lesssim &||F||_{\cc^{\sv,M}((W,\dv)\boxtimes \nabla_{\RR^N})}(1+|| u||_{L^{\infty}}+||v||_{L^{\infty}})^{M+\nu-2}(|| u||_{L^{\infty}}+||v||_{L^{\infty}})^{L-2}\nonumber \\&\times\left(\|u\|_{L^{\infty}}+\|v\|_{L^{\infty}}+(1+\| u\|_{L^{\infty}}+\|v\|_{L^{\infty}})(|| u||_{L^{p}_{\sv}(W,\dv)}+||v||_{L^{p}_{\sv}(W,\dv)})\right) ||u-v||_{L^{\infty}}.
 \end{align}
 For (VI) we use the fact that $\partial_{\zeta}^{\beta}D_{\zeta}f(x,0)\equiv 0, \ \forall |\beta|<L-1$, to see that 
 \ba 
& \underset{\sigma\in [0,1]}{\sup}\left\|D_{\zeta}F(x,\sigma u(x)+(1-\sigma)v(x))\right\|_{L^{\infty}} \\
 &\lesssim \underset{\sigma\in [0,1]}{\sup} \sum_{|\beta|=L-1}\left\|\partial_{\zeta}^{\beta}D_{\zeta}F(x,\sigma u(x)+(1-\sigma)v(x))\right\|_{L^{\infty}} ||\sigma u+(1-\sigma)v||^{L-1}\\
 &\lesssim \|F\|_{\cc^{\sv,M}((W,\dv)\boxtimes \nabla_{\RR^N})} (||u||_{L^{\infty}}+||v||_{L^{\infty}})^{L-1},
 \ea
 where in the last inequality we used $M>L$ and Lemma \ref{Preliminaries: Functions spaces: Zygmund-Holder space: Prop 7.5.11} and Remark \ref{Tame estimate for sobolev space: remark about embedding to space of continuous functions}. Therefore
 \begin{align}\label{Tame estimate Corollary estimate for VI}
     (VI)\lesssim \|F\|_{\cc^{\sv,M}((W,\dv)\boxtimes \nabla_{\RR^N})} (||u||_{L^{\infty}}+||v||_{L^{\infty}})^{L-1}\left\|u-v\right\|_{L^{p}_{\sv}(W,\dv)}.
 \end{align}
 Now, combining \eqref{Tame estimate Corollary estimate for (V)} and \eqref{Tame estimate Corollary estimate for VI} we get the required estimate. 
 \end{enumerate}
 \end{proof}

\section{Reduction I} \label{Reduction I}
\numberwithin{equation}{section}

 Theorem \ref{Preliminaries: Theorem 3.15.5 of [BS]} applied to $\sigma=7/8$, gives us a scaling map $\Phi_{x,\delta}: B^{n}(1)\to B^{n}(1)$ for $x\in B^{n}(1), 0<\delta<1$. We refer the reader to \cite[Section 3.15.3]{BS} for further details regarding $\Phi_{x,\delta}$. The main quantitative regularity theorem is the following:
 \begin{theorem}\label{Introduction: Main Theorem}
     Suppose all of the above assumptions hold. Then, there exists an $(p,\sv,\rv,\kappa, \sigma, M)-$ multi-parameter unit admissible constant $\delta=\delta(A,C_0, D_1,D_2)\in (0,1)$ such that the following holds. Let $\psi_0\in C_{0}^{\infty}(B^{n}(2/3))$. Then $$(\psi_0\circ \Phi_{0,\delta})u\in L^{p}_{\sv+\kappa \ef_1} (\widebar{B^n(7/8)}, (W,\dv); \CC^{D_2})$$ and 
     \begin{align*}
         \|(\psi_0\circ \Phi_{0,\delta}^{-1})u\|_{L^{p}_{\sv+\kappa \ef_1}(W,\dv)}\leq K,
     \end{align*}
     where $K=K(A,C_0, D_1,D_2, \psi_0)\geq 0$ is an $(p,\sv, \rv, \kappa, \sigma, M)-$ multi-parameter unit-admissible constant. 
  \end{theorem}

The rest of the article will be dedicated to the proof of Theorem \ref{Introduction: Main Theorem}. We see Theorem \ref{Introduction: Main theorem: qualitative version} as a corollary of Theorem \ref{Introduction: Main Theorem}.
   \begin{proof}[Proof of Theorem~\ref{Introduction: Main theorem: qualitative version}]
     Observe that the hypothesis of Theorem \ref{Introduction: Main theorem: qualitative version} are invariant under diffeomorphisms. By choosing a coordinate chart $\Phi: B^n(2)\tilde{\to} \Phi(B^n(2))\subseteq \mathcal{M}$, where $\Phi(0)=x_0$ and $\Phi(B^n(2))$ is a small neighbourhood of $x_0\in \mathcal{M}$, writing the above assumptions in coordinate system, and restricting the estimates to $B^n(1)$, we see that all the hypothesis of Theorem \ref{Introduction: Main theorem: qualitative version} hold. Then, we get the required result as a corollary. 
 \end{proof}
With the power of the the tame estimates we proved in Section \ref{Tame estimate for Sobolev space} we can start the proof of Theorem \ref{Introduction: Main Theorem}. 

 Now, we will see a weaker version of Theorem \ref{Introduction: Main Theorem}, whose hypothesis will be easier to verify in the final step of the proof. The idea of the first reduction is to reduce the weaker version of the PDE to a perturbation linear PDE, for which we prove the regularity theorem in the second reduction. 

Let $\epsilon_2\in (0,1]$ be a small number, to be chosen later. Suppose 
\begin{align*}
    R_{3}(x,\zeta)\in \cc^{\sv, M} (\widebar{B^{n}(7/8)}\times \RR^{N}, (W,\dv)\boxtimes \nabla_{\RR^N};\RR^{D_1}),\ g\in L^{p}_{\sv}(\widebar{B^{n}(7/8)}; (W,\dv); \RR^{D_1})
\end{align*}
with $M\geq \nu(|\sv|_{\infty}+1)+2$ and 
\ba 
\|R_3\|_{\cc^{\sv,M}((W,\dv)\boxtimes \nabla_{\RR^N})}, \|g\|_{L^{p}_{\sv}(W,\dv)}, \|g\|_{\cc^{\rv}(W,\dv)}\leq \epsilon_2.
\ea
Let $G(\zeta)\in \cc^{M+1} (\RR^N;\RR^{D_1})$ and take $C_{9}\geq 0$ with 
\begin{align*}
    \|G\|_{\cc^{M+1}}\leq C_9
\end{align*}
and $u\in \cc^{\rv+\kappa\ef_1}(\widebar{B^n(7/8)}, (W,\dv); \RR^{D_2})$ with 
\begin{align*}
     \|u\|_{\cc^{\rv+\kappa\ef_1}(W,\dv)}\leq \epsilon_2. 
\end{align*}
We assume that $dG(\mathfrak{W}_1^{\kappa} u(0))\mathfrak{W}_1^{\kappa}$ is maximally subelliptic of degree $\kappa$ with respect to $(W,\dv)$ in the sense that there exists $A_1\geq 0$ such that 
\begin{align}\label{Reduction II: maximal subellipticity definition}
    \sum_{j=1}^{r_1}\|(W_{j}^1)^{n_j} v\|_{L^2(B^{n}(1),h\sigma_{Leb})} \leq A_1 (\|dG (\mathfrak{W}_1^{\kappa}u(0))\mathfrak{W}_{1}^{\kappa}v \|_{L^2(B^{n}(1),h\sigma_{Leb})} + \|v\|_{L^2(B^{n}(1),h\sigma_{Leb})}),
\end{align}
for all $v\in C_{0}^{\infty}(B^n(1);\CC^{D_2})$. Finally, we suppose that $u$ satisfies the following non-linear PDE:
\begin{align}\label{Reduction I: non linear PDE}
    G(\Wf_1^{\kappa}u(x)) +R_{3}(x,\Wf_1^{\kappa}u(x))=g(x)\ \forall x\in B^n(13/16).
\end{align}
\begin{prop}\label{Reduction I: main prposition of second reduction}
    There exists $\epsilon_2\in (0,1]$ sufficiently small such that if the above holds $\forall \psi_0\in C_{0}^{\infty}(B^{n}(2/3))$ we have
    \begin{align*}
        \psi_0 u \in L^{p}_{\sv+\kappa\ef_1} ((W,\dv);\RR^{D_2})
    \end{align*}
    and $C_{10}\geq 0$ with
    \begin{align*}
        \|\psi_0\|_{L^{p}_{\sv+\kappa\ef_1}(W,\dv)}\leq C_{10}. 
    \end{align*}
    Here, $\epsilon_2=\epsilon_2(C_9,A_1, D_1,D_2)\in (0,1]$ and $C_{10}=C_{10}(C_9, A_1,D_1,D_2,\psi_0)\geq 0$ are $(p,\kappa, \sv, \rv, \sigma,M)-$multi-parameter unit admissible constants.  
\end{prop}
\begin{proof}
     Lemma \eqref{Preliminaries: Sobolev space: lemma 8.3.3 (i) of [BS]} combined with the fact that $h\approx 1$ shows that 
    \begin{align*}
        \sum_{\deg_{\db^1}(\alpha)\leq \kappa} \|(W^1)^{\alpha} v\|_{L^{2}(B^{n}(1), h\sigma_{Leb})} \approx \sum_{j=1}^{r_1} \|(W_{j}^1)^{n_j} v\|_{L^2(B^n(1),h\sigma_{Leb})} +\|v\|_{L^2(B^{n}(1), h\sigma_{Leb})},
    \end{align*}
    for all $v\in C_{0}^{\infty}(B^{n}(31/32); \CC^{D_2})$. Combining  this with \eqref{Reduction II: maximal subellipticity definition} gives us that there exists $A_{2}\approx 1$ with 
    \begin{align}\label{Reduction I: maximal subellipticity rewritten}
        \sum_{\deg_{\db^1}({\alpha})\leq \kappa}\|(W^1)^{\alpha} v\|_{L^{2}(B^{n}(1), h\sigma_{Leb})} \leq A_2 (\|dG (\mathfrak{W}_1^{\kappa}u(0))\mathfrak{W}_{1}^{\kappa}v \|_{L^2(B^{n}(1),h\sigma_{Leb})} + \|v\|_{L^2(B^{n}(1),h\sigma_{Leb})}),
    \end{align}
    for all $v\in C_{0}^{\infty}(B^{n}(31/32);\CC^{D_2})$. We claim that if $\epsilon_2$ is sufficiently small, then 
   \begin{align}\label{Reduction I: maximal subellipticity rewritten with improved constant}
        \sum_{\deg_{\db^1}({\alpha})\leq \kappa}\|(W^1)^{\alpha} v\|_{L^{2}(B^{n}(1), h\sigma_{Leb})} \leq 2 A_2 (\|dG (0)\mathfrak{W}_{1}^{\kappa}v \|_{L^2(B^{n}(1),h\sigma_{Leb})} + \|v\|_{L^2(B^{n}(1),h\sigma_{Leb})}),
    \end{align}
    for all $v\in C_{0}^{\infty}(B^{n}(31/32);\CC^{D_2})$. We know that 
    \begin{align*}
        \|\Wf_1^{\kappa} u\|_{L^{\infty}} \lesssim \|\Wf_1^{\kappa} u\|_{\cc^{\rv}(W,\dv)}\lesssim \|u\|_{\cc^{\rv+\kappa\ef_1}(W,\dv)}\lesssim \epsilon_2.
    \end{align*}
    Thus since, $M+1\geq 2$
    \begin{align*}
        |dG(\Wf_{1}^{\kappa}u(0))-dG(0)|\lesssim \sum_{|\beta|\leq 2} \|\partial_{\zeta}^{\beta}G\|_{L^{\infty}(\RR^N)} |\Wf_{1}^{\kappa} u(0)| \lesssim \|G\|_{\cc^{M+1}(\RR^{N})} \|\Wf_1^{\kappa} u\|_{L^{\infty}}\lesssim \epsilon_2. 
    \end{align*}
    Plugging in these inequalities into \eqref{Reduction I: maximal subellipticity rewritten} we get
    \begin{align*}
        \sum_{\deg_{\db^1}({\alpha})\leq \kappa}\|(W^1)^{\alpha} v\|_{L^{2}(B^{n}(1), h\sigma_{Leb})} \leq &A_2 (\|dG (0)\mathfrak{W}_{1}^{\kappa}v \|_{L^2(B^{n}(1),h\sigma_{Leb})} + \|v\|_{L^2(B^{n}(1),h\sigma_{Leb})}) +
        \\ &A2 C_{11} \epsilon_2 \|\Wf_1^{\kappa} v\|_{L^{2}(B^{n}(1), h\sigma_{Leb}}
    \end{align*}
for some $C_{11}\approx 1$, $\forall v\in C_{0}^{\infty}(B^{n}(31/32);\CC^{D_2})$. Now, take $\epsilon_2\in (0,1\wedge \frac{1}{2} A_2 C_{11}]$ to get \eqref{Reduction I: maximal subellipticity rewritten with improved constant}. 

Now we see that
\begin{align*}
      |G(\Wf_{1}^{\kappa}u(0))-G(0)|\lesssim \sum_{|\beta|\leq 1} \|\partial_{\zeta}^{\beta}G\|_{L^{\infty}(\RR^N)} |\Wf_{1}^{\kappa} u(0)| \lesssim \|G\|_{\cc^{M+1}(\RR^{N})} \|\Wf_1^{\kappa} u\|_{L^{\infty}}\lesssim \epsilon_2. 
      \end{align*}
      Also, 
      \begin{align*}
          |G(\Wf_{1}^{\kappa}u(0))| &\leq |G(\Wf_1^{\kappa} u(0)) +R_{3}(0,\Wf_{1}^{\kappa}u(0))| +\|R_3\|_{L^{\infty}}\\
          &\leq \|g\|_{L^{\infty}}+ \|R_3\|_{L^{\infty}}\\
          &\lesssim \|g\|_{\cc^{\rv}(W,\dv)}+ \|R_3\|_{\cc^{\sv,M}(W,\dv)\boxtimes \nabla_{\RR^{N}}}.
      \end{align*}
Combining these, we get
\begin{align*}
    |G(0)|\lesssim \epsilon_2. 
\end{align*}
Define, $E(\zeta)$ by $G(\zeta)=G(0)+ d_{\zeta} G(0)\zeta +E(\zeta)$, so that $E(0)=0$, $d_{\zeta}E(0)=0$, $E\in \cc^{M+1}(\RR^N)$, 
\begin{align*}
    \|E\|_{\cc^{M+1}}\lesssim 1.
\end{align*}
Set $\mathcal{P}:=d_{\zeta}G(0)\Wf_1^{\kappa}$. Rewriting \eqref{Reduction I: non linear PDE} gives 
\begin{align}
    \mathcal{P}u(x)=-G(0)- R_3(x,\Wf_1^{\kappa} u(x))-E(\Wf_1^{\kappa} u(x)) \ \forall x\in B^{n}(13/16).
\end{align}
Fix $\psi \in C_0^{\infty}(B^n(7/8))$ with $\psi_1\equiv 1$ on $B^{n}(13/16)$. Set 
\begin{align*}
    R_{1}(x,\zeta)&:= -\psi_1(x)(G(0)+R_3(x, \zeta)),\\
    R_{2}(x,\zeta)&:= -\psi_1(x) E(\zeta). 
\end{align*}
i.e
\begin{align}\label{Reduction I: writing the nonlinear PDE in term of R_1 and R_2}
    \mathcal{P}u(x)= R_1(x,\Wf_1^{\kappa}u(x))+R_{2}(x,\Wf_1^{\kappa}u(x))\ \forall x\in B^{n}(13/16).
\end{align}
We can also get the following estimates for $R_1$ and $R_2$
\begin{align*}
    \|R_1\|_{\cc^{\sv, M }((W,\dv)\boxtimes \nabla_{\RR^N}) }\lesssim |G(0)|+ \|R_3\|_{\cc^{\sv,M}(W,\dv)\boxtimes\nabla_{\RR^N}}\lesssim \epsilon_2,
\end{align*}
and 
\begin{align*}
    \|R_2\|_{\cc^{\sv,M+1}(W,\dv)\boxtimes \nabla_{\RR^N}}\lesssim \|E\|_{\cc^{M+1}(\RR^N)}\lesssim 1.
\end{align*}
Hence, there exists $C_{12}\approx 1$ with 
\begin{align*}
  \|R_1\|_{\cc^{\sv, M }((W,\dv)\boxtimes \nabla_{\RR^N}) }&\leq C_{12}\epsilon_2, \\
  \|R_2\|_{\cc^{\sv,M+1}(W,\dv)\boxtimes \nabla_{\RR^N}}&\leq C_{12}. 
\end{align*}
Also, 
\begin{align*}
    |d_{\zeta}G(0)|\leq \|G\|_{\cc^{M+1}(\RR^n)}\leq 1. 
\end{align*}
i.e there exists $C_{13}\approx 1$ with 
\begin{align*}
    |d_{\zeta}G(0)|\leq C_{13}.
\end{align*}
\begin{remark}
    In short we were able to reduce \eqref{Reduction I: non linear PDE} to look like a perturbation of a linear operator given in \eqref{Reduction I: writing the nonlinear PDE in term of R_1 and R_2}. So, we will now prove the regularity theorem for this perturbed operator. 
\end{remark}
\end{proof}
\section{Scaling}
Before we are able to use the Reduction I, we have to introduce appropriate notion of scaling. The main goal of this section is to introduce scaled vector fields and norms to prove scaled estimates for the proof of regularity theorem for the perturbed linear operator obtained in the previous section. 
\subsection{Scaled vector fields and norms}
Let $(W, \dv)$ and $(X, \vec{d})$ be as in Section \ref{Multi-parameter unit scale}. Let $\lambda=(1, \lambda_1,...\lambda_{\nu})$ be as Section \ref{The Geometry}. Let $\delta\in (0,1]$ and $x\in \widebar{B^{n}(7/8)}$. Recall the definition of the following vector scaled vector fields from Theorem \ref{Preliminaries: Theorem 3.15.5 of [BS]}.
\begin{align*}
    (W^{\mu, x, \delta}, \db^{\mu})&:= \{(W_1^{\mu, x, \delta}, \db_1^{\mu}),..., (W_{r_{\mu}}^{\mu, x, \delta}, \db_{r_{\mu}}^{\mu})\}\subset C^{\infty}(B^n(1);TB^n(1))\times 
    \NN_+,\\
     (X^{\mu, x, \delta}, d^{\mu})&:= \{(X_1^{\mu, x, \delta}, d_1^{\mu}),..., (X_{q_{\mu}}^{\mu, x, \delta}, d_{q_{\mu}}^{\mu})\}\subset C^{\infty}(B^n(1);TB^n(1))\times 
    \NN_+,
\end{align*}
and $h_{x, \delta}\in C^{\infty}(B^n(1))$ be as in Theorem \ref{Preliminaries: Theorem 3.15.5 of [BS]} with $\sigma =7/8$. Now, define 
\begin{align*}
    (W^{x, \delta}, \dv)&:= (W^{1, x, \delta}, \db^1)\boxtimes (W^{2, x, \delta}, \db^2)\boxtimes...\boxtimes (W^{\nu, x, \delta}, \db^{\nu})  \\&\subset C^{\infty}(B^n(1);TB^n(1))\times 
(\NN^{\nu}\backslash \{0\}),\\
   (X^{x, \delta}, \vec{d})&:= (X^{1, x, \delta}, d^1)\boxtimes (X^{2, x, \delta}, d^2)\boxtimes...\boxtimes (X^{\nu, x, \delta}, d^{\nu})  \\&\subset C^{\infty}(B^n(1);TB^n(1))\times 
(\NN^{\nu}\backslash \{0\}),
\end{align*}
For each $\delta \in (0,1]^{\nu}$, we also consider the following vector fields with formal degrees. 
\begin{align*}
    (\delta^{\lambda\dv}W, \dv)&:= (\delta^{\lambda_1 \db^1}W^1, \db^1)\boxtimes (\delta^{\lambda_2\db^2}W^2, \db^2)\boxtimes...\boxtimes (\delta^{\lambda_{\nu}\db^{\nu}}W^{\nu}, \db^{\nu})\\
    &\subset C^{\infty}(B^n(1); TB^n(1))\times (\NN^{\nu}\backslash \{0\}),\\
     (\delta^{\lambda\vec{d}}W, \vec{d})&:= (\delta^{\lambda_1 d^1}X^1, d^1)\boxtimes (\delta^{\lambda_2d^2}X^2, d^2)\boxtimes...\boxtimes (\delta^{\lambda_{\nu}d^{\nu}}X^{\nu}, d^{\nu})\\
    &\subset C^{\infty}(B^n(1); TB^n(1))\times (\NN^{\nu}\backslash \{0\}).
\end{align*}
Observe that by definition (see Theorem \ref{Preliminaries: Theorem 3.15.5 of [BS]} part (vi)), $(W^{x, \delta}, \dv)= (\Phi_{x, \delta}^*\delta^{\lambda\dv}, \dv)$ and $(X^{x, \delta}, \vec{d})=(\Phi^*_{x, \delta}\delta^{\lambda\vec{d}}X, \vec{d})$.

For the rest of the article we fix the compact set $\mathcal{K}= \widebar{B^n(7/8)}$ and we had remarked about(see Remark\ref{Function spaces: remark about explicit norm: remark about explicit norm part 2}) an explicit construction of $D_j$ given $(X, \dv)$. 
Similarly, given $x\in \widebar{B^n(7/8)}, \delta\in (0,1]$ we can define $D_j^{x, \delta}$ that corresponds to $(X^{x, \delta}, \vec{d})$, and $D_j^{\delta}$ corresponding to $(\delta^{\lambda \dv}X, \vec{d})$. 
Now, one can define the function spaces as in Section \ref{Function spaces} with $(W^{x, \delta}, \dv), (X^{x, \delta}, \vec{d})$  and $(\delta^{\lambda\dv}W, \dv), (\delta^{\lambda \vec{d}}X\vec{d} )$. Now, we will prove some lemmas comparing the norms of these functions spaces. 

\begin{lemma}\label{Scaling: Vector Fields and Norms: Similar to Lemma 9.2.6 of [BS]}. 
     Fix $c_0\in (0,1]$ and $\sv\in (0,\infty)^{\nu}$. There exists an $\sv$-multi-parameter unit admissible constant $C_1=C_1(c_0)\geq 1$ and $(p, \sv)$-multi-parameter unit admissible constant $C_2=C_2(c_0)\geq 1$ such that $\forall \delta \in [c_0,1], \forall f\in L^p_{\sv}(\widebar{B^{n}(7/8)}, (W,\dv))$, 
    \begin{align}
C_2^{-1} \|f\|_{L^p_{\sv}(W, \dv)} \leq \|f\|_{L^p_{\sv}(\delta^{\lambda\dv}W, \dv)} \leq C_2 \|f\|_{L^p_{\sv}(W, \dv)}.\label{Scaling: Vector fields and norms: Lemma 9.2.6 : equation 2}
    \end{align}
\end{lemma}
\begin{proof}
    This lemma is a simple corollary of \cite[Proposition 6.3.7]{BS} and can be seen by keeping tracking of the constants in that proof. 
\end{proof}
\begin{lemma}\label{Scaling: Vector Fields and Norms: Similar to Lemma 9.2.7 of [BS]}
    For $x\in \widebar{B^{n}(7/8)} ,\delta\in (0,1], \sv \in (0,\infty)^{\nu},$ and $\phi\in C_{0}^{\infty}(B^{n}(7/8))$, we have 
    \begin{align}
 \|\phi \Phi_{x,\delta}^* f\|_{L^{p}_{\sv}(W^{x,\delta}, \db)}= \Lambda(x, \delta)^{-1/p}\|(\phi\circ \phi_{x,\delta}^{-1}) f\|_{L^{p}_{\sv}(\delta^{\lambda\dv} W, \dv)},\label{Scaling: Vector fields and norms: Lemma 9.2.7 : equation 2}
    \end{align}
    $\forall f\in C_{0}^{\infty}(B^{n}(1))'$, where 
    \begin{align*}
        \Lambda(x, \delta):= \underset{j_1, ..., j_n\in \{1,..., q_1\}}{\max}\ Vol(x)(\delta^{d_{j_1}^1}X_{j_1}^1(x),...,\delta^{d_{j_n}^1}X_{j_n}^1(x) )
    \end{align*}
\end{lemma}
\begin{proof}
    This Lemma is a consequence of \cite[Proposition 6.11.4]{BS}. 
\end{proof}
 \subsection{Scaled bump functions}
To prove estimates in different scales we need bump functions in different scales. This is what we will introduce in this chapter. We will also see how the norms compare in the support of these scaled bump functions.

The version of the following theorem was first proven by Nagel and Stein in \cite{NS}, who showed the existence of bump functions at every sub-Riemannian scale to get appropriate scaled estimates. We will be following \cite[Theorem 9.2.8]{BS} for this. The idea is to apply Theorem \ref{Preliminaries: Theorem 3.15.5 of [BS]} with $\sigma=7/8$. We use the same 0-multi-parameter unit-admissible constant $\xi_1\in (0,1]$ from that theorem. 
\begin{theorem}\label{Scaling: theorem 9.2.8 [BS]}
    There exist $0$-multi-parameter unit-admissible constants $\xi_2, \xi_3\in (0,1]$ with $\xi_3< \xi_2< \xi_1$ such that $\forall x\in B^{n}(7/8)$ and $\delta\in (0,1]$, there exists $\phi_{x,\delta}\in C_{0}^{\infty}(B^{n}(1))$ and, setting $\xi_4:= \xi_3^4/\xi_2^3\in (0,\xi_3)$, we have $\forall x\in B^{n}(7/8)$, $\delta\in (0,1]$:
    \begin{enumerate}[label=(\roman*)]
        \item $supp(\phi_{x,\delta})\subseteq B_{(W^1,\db^1)}(x,\delta)$.
        \item $0\leq \phi_{x,\delta}\leq 1$
        \item $\psi_{x,\delta}\equiv 1$ on a neighbourhood of $\overline{B_{(W^1,\db^1)}(x,\xi_3 \delta)}$.
        \item $\psi_{x,\frac{\xi_1}{\xi_2}\delta}\prec \psi_{x,\delta}$.
        \item For $\xi_0\in (0,1]$, $supp(\phi_{x,\xi_0\delta})\subseteq \Phi_{x,\delta}(B^{n}(1/2))$. In particular, 
        \begin{align*}
            supp(\phi_{x,\frac{\xi_2}{\xi_3}\xi_4\delta}), supp(\phi_{x,(\frac{\xi_2}{\xi_3})^2\xi_4\delta}),  supp(\phi_{x,(\frac{\xi_2}{\xi_3})^3\xi_4\delta}) \subseteq \phi_{x,\delta}(B^{n}(1/2)). 
        \end{align*}
        \item For every $\alpha$
        \begin{align*}
            \underset{\delta\in (0,1]}{\underset{x\in B^{n}(7/8)}{\sup}}\ \|(\delta^{\lambda\vec{d}}X)^{\alpha} \phi_{x,\delta}\|_{C(B^{n}(1))}\leq C_{\alpha},
        \end{align*}
        where $C_{\alpha}$ is an $\alpha$-multi-parameter unit admissible constant. 
        \item For $\xi_0\in (0,1]$ and $L\in \NN$,
        \begin{align*}
            \|\phi_{x,\xi_0\delta}\circ \Phi_{x,\delta}\|_{C^{L}(B^{n}(1))}\leq C_{L} \xi_0^{-L \max_{1\leq j \leq q_1} d_j^1},
        \end{align*}
        where $C_{L}\geq 0$ is an $L$-multi-parameter unit-admissible constant.
        \item There exists a $0$-multi-parameter unit-admissible constant $N_1\in \NN_+$ such that $\forall \sv\in (0,\infty)^{\nu}$, there exists $C_{\sv}\geq 0$ and $\sv$-multi-parameter unit-admissible constant such that $\forall x\in B^{n}(3/4)$ and $\delta\in (0,1]$ with $B_{(W^, \db^1)}(x,\delta)\subseteq B^{n}(3/4)$, there exist $x_1,..., x_{N_1}\in B^{n}(3/4)$ with 
        \begin{align}
            &B_{(W^1, \db^1)} \left(x_j, \frac{\xi_2^3\xi_4}{\xi_3^3} \delta\right)\subseteq B^{n}(3/4),\label{Scaling: scaled estimate (viii) inclusion of balls} \\
            & \|\phi_{x,\delta}u\|_{\cc^{\sv}({\delta^{\lambda\dv}}W, \dv)} \leq C_{\sv} \sum_{j=1}^{N_1} \|\phi_{x_j,\xi_4 \delta} u\|_{\cc^{\sv}({\delta^{\lambda\dv}}W, \dv)}\label{Scaling: scaled estimate (viii) for zygmund holder space}\\
          &   \|\phi_{x,\delta}u\|_{L^{p}_{\sv}(\delta^{\lambda\dv} W, \dv)}\leq C_{\sv} \sum_{j=1}^{N_1} \|\phi_{x_j,\xi_4\delta}u\||_{L^{p}_{\sv}(\delta^{\lambda\dv} W, \dv)}\label{Scaling: scaled estimate (viii) for sobolev space},
        \end{align}
        for all $u\in C_{0}^{\infty}(B^{n}(1))'$.
        \item There exist $0$-multi-parameter unit-admissible constants $N_2\in \NN_+$ and $\delta_1\in (0,1]$ and $x_1,..., x_{N_2}\in B^{n}(3/4)$ with $B_{(W^1, \db^1)}(x,\delta_1) \subseteq B^{n}(3/4)$, $1\leq j \leq N_2$, such that $\forall \sv\in (0,\infty)^{\nu}$, $\forall \psi\in C_{0}^{\infty}(B^{n}(2/3))$, there exists an $\sv$- multi-parameter unit -admissible constant $C_{\sv}= C_{\sv}(\psi) \geq 0$ such that 
        \begin{align}
          &  \|\psi u\|_{\cc^{\sv}(W,\dv)}\leq C_{\sv}\sum_{j=1}^{N_2} \|\phi_{x_j, \delta_1}u\|_{\cc^{\sv}(\delta_{1}^{\lambda\dv}W,\dv)}\label{Scaling: scaled estimate (ix) for zygmund holder space}\\
            &  \|\psi u\|_{L^{p}_{\sv}(W,\dv)}\leq C_{\sv}\sum_{j=1}^{N_2} \|\phi_{x_j, \delta_1}u\|_{L^{p}_{\sv}(\delta_{1}^{\lambda\dv}W,\dv)}\label{Scaling: scaled estimate (ix) for sobolev space}
        \end{align}
        \item Fix $\xi_0\in (0,1]$ and $\sv\in (0,\infty)^{\nu}$. Then
        \begin{align}
           & \|\phi_{x, \xi_{0}\delta}u\|_{\cc^{\sv}(\delta^{\lambda\dv}W, \dv)} \approx\|\phi_{x, \xi_0 \delta}u\|_{\cc^{\sv}((\xi_0\delta)^{\lambda\dv}W, \dv)} \label{Scaling: scaled estimate (x) for zygmund holder space}\\
           &\|\phi_{x, \xi_{0}\delta}u\|_{L^p_{\sv}(\delta^{\lambda\dv}W, \dv)} \approx\|\phi_{x, \xi_0 \delta}u\|_{L^p_{\sv}((\xi_0\delta)^{\lambda\dv}W, \dv)}\label{Scaling: scaled estimate (x) for sobolev space},
        \end{align}
        $\forall u\in C_{0}^{\infty}(B^n(1))'$, where the implicit constants are $\sv$- multi-parameter unit admissible constants which depend on $\xi_0$. 
    \end{enumerate}
\end{theorem}

\begin{proof}
    The proof of the Theorem \ref{Scaling: theorem 9.2.8 [BS]} is same as in \cite[Theorem 9.2.8]{BS}, except for $L^{p}_{\sv}$ estimates in (viii), (ix) and (x). These are extra statements that we need in our case, which was not stated in \cite[Theorem 9.2.8]{BS}. The proof of the $L^{p}_{\sv}$ statements in (viii) and (ix) uses Lemma \ref{Scaling: Vector Fields and Norms: Similar to Lemma 9.2.7 of [BS]} . We do not prove these statements here as the proof proceeds essentially the same structure as in Theorem Theorem \ref{Scaling: theorem 9.2.8 [BS]} and is not particularly enlightening. However, \eqref{Scaling: scaled estimate (x) for sobolev space} requires a little work and hence we will provide the proof here.

    We will show that for any two numbers $\gamma_1, \gamma_2\in [\xi_0, 1]$
    \begin{align}\label{Scaling: Bump functions: part (x) required estimate in the proof}
        \|\phi_{x, \xi_{0}\delta}u\|_{L^p_{\sv}((\gamma_1\delta)^{\lambda\dv}W, \dv)} \lesssim_{\sv;\xi_0}\|\phi_{x, \xi_0 \delta}u\|_{L^p_{\sv}((\gamma_2\delta)^{\lambda\dv}W, \dv)}.
    \end{align}
    This will prove the required estimate. 
    By (i), we have $supp(\phi_{x_, \xi_0\delta})\subseteq B_{(W^1, \db^1)}(x, \xi_0\xi_2\delta)=B_{(W^1, \db^1)}(x,\gamma_2\delta\xi_2\xi_0/\gamma_2)$. Hence, 
    \begin{align*}
        supp(\phi_{x, \xi_0\delta}\circ \Phi_{x, \gamma_2\delta})\subseteq B_{(W^{1, x, \gamma_2\delta}, \db^1)}(0, \xi_0\xi_2/\gamma_2) \subseteq B_{(W^{1, x, \gamma_2\delta}, \db^1)}(0, \xi_2) \subseteq B^{n}(7/8).
    \end{align*}
    Using this and Lemma \ref{Scaling: Vector Fields and Norms: Similar to Lemma 9.2.7 of [BS]}, we see that proving \eqref{Scaling: Bump functions: part (x) required estimate in the proof} is same as proving 
    \begin{align}
        \|(\phi_{x, \xi_0\delta}\circ \Phi_{x, \gamma_1\delta})u \circ \Phi_{x, \gamma_1\delta}\|_{L^p_{\sv}(W^{1, x, \gamma_1\delta}, \db)}\lesssim_{\sv; \xi_0}\|(\phi_{x, \xi_0\delta}\circ \Phi_{x, \gamma_2\delta})u \circ \Phi_{x, \gamma_2\delta}\|_{L^p_{\sv}(W^{1, x, \gamma_2\delta}, \db)}
    \end{align}
    Without loss of generality assume that $\|(\phi_{x, \xi_0\delta}\circ \Phi_{x, \gamma_2\delta})u \circ \Phi_{x, \gamma_2\delta}\|_{L^p_{\sv}(W^{1, x, \gamma_2\delta}, \db)}<\infty$, else the required inequality holds trivially. Now, Proposition \ref{Preliminaries: smooth function decomposition for Sobolev space} shows that there exists a sequence $\hat{v}_j\in C^{\infty}(B^n(1))$, $j\in \NN^{\nu}$, such that 
    \begin{align*}
        \sum_{j\in \NN^{\nu}} \hat{v}_j = (\phi_{x, \xi_0\delta}\circ \Phi_{x, \gamma_2\delta})u\circ \Phi_{x, \gamma_2\delta}
    \end{align*}
    and 
    \begin{align*}
       \left \|\left(\sum_{j\in \NN^{\nu}} |2^{j\cdot \sv} (2^{-j\cdot \dv}W^{1, x, \gamma_2\delta})^{\alpha} \hat{v}_j |^2\right)^{1/2}\right\|_{L^{p}(B^n(1))} \lesssim_{\alpha} \|(\phi_{x, \xi_0\delta}\circ \Phi_{x, \gamma_2\delta})u\circ \Phi_{x, \gamma_2\delta}\|_{L^{p}_{\sv}(\overline{B^n(7/8)}, W^{1, x, \gamma_2\delta}, \dv)},
    \end{align*}
    for all multi-index $\alpha$. Let $v_{j}:= \hat{v}_j\circ \Phi_{x, \gamma_2\delta}^{-1}$, we have $v_j\in C^{\infty}_0(\mathcal{M}), \sum_{j\in \NN^{\nu}} v_j= \phi_{x, \xi_0\delta}u$, and 
    \begin{align*}
      \left \|\left(\sum_{j\in \NN^{\nu}} |2^{j\cdot \sv} ((\gamma_2\delta)^{\lambda\dv}2^{-j\cdot \dv}W)^{\alpha} v_j |^2\right)^{1/2}\right\|_{L^{p}(\mathcal{M})} \lesssim_{\alpha} \|(\phi_{x, \xi_0\delta}\circ \Phi_{x, \gamma_2\delta})u\circ \Phi_{x, \gamma_2\delta}\|_{L^{p}_{\sv}( W^{1, x, \gamma_2\delta}, \dv)}.
    \end{align*}
    Since $\gamma_1\approx_{\xi_0}\gamma_2$ we get 
    \begin{align*}
       \left \|\left(\sum_{j\in \NN^{\nu}} |2^{j\cdot \sv} ((\gamma_1\delta)^{\lambda\dv}2^{-j\cdot \dv}W)^{\alpha} v_j |^2\right)^{1/2}\right\|_{L^{p}(\mathcal{M})} \lesssim_{\alpha, \xi_0} \|(\phi_{x, \xi_0\delta}\circ \Phi_{x, \gamma_2\delta})u\circ \Phi_{x, \gamma_2\delta}\|_{L^{p}_{\sv}( W^{1, x, \gamma_2\delta}, \dv)}.   
    \end{align*}
    Now, we let $\tilde{v}_j:= v_j\circ \Phi_{x, \gamma_1\delta}$ we see that $\tilde{v}_j\in C^{\infty}(B^n(1))$ with 
  \begin{align*}
       \left \|\left(\sum_{j\in \NN^{\nu}} |2^{j\cdot \sv} (2^{-j\cdot \dv}W^{1, x, \gamma_1\delta})^{\alpha} \tilde{v}_j |^2\right)^{1/2}\right\|_{L^{p}(B^n(1))} \lesssim_{\alpha, \xi_0} \|(\phi_{x, \xi_0\delta}\circ \Phi_{x, \gamma_2\delta})u\circ \Phi_{x, \gamma_2\delta}\|_{L^{p}_{\sv}( W^{1, x, \gamma_2\delta}, \dv)},
    \end{align*}
    for all multi-index $\alpha$, and we also have 
    $\sum_{j\in \NN^{\nu}} \tilde{v}_j= (\phi_{x, \xi_0\delta}\circ \Phi_{x, \gamma_1\delta})u\circ \Phi_{x, \gamma_1\delta}$. We use Proposition \ref{Preliminaries: smooth function decomposition for Sobolev space} again to get 
    \begin{align*}
       \|(\phi_{x, \xi_0\delta}\circ \Phi_{x, \gamma_1\delta})u\circ \Phi_{x, \gamma_1\delta}\|_{L^{p}_{\sv}( W^{1, x, \gamma_1\delta}, \dv)}  \lesssim_{\sv; \xi_0} \|(\phi_{x, \xi_0\delta}\circ \Phi_{x, \gamma_2\delta})u\circ \Phi_{x, \gamma_2\delta}\|_{L^{p}_{\sv}( W^{1, x, \gamma_2\delta}, \dv)}.
    \end{align*}
    Hence we are done. 
\end{proof}

\subsection{Scaled estimates}
Recall that $\kappa \in \NN_+$ is such that $\db_j^1$ divides $\kappa, \forall 1\leq j\leq r_1$, and $n_j =\kappa /\db_j^1\in \NN_+$. Let $\mathcal{P}$ be a partial differential operator of the form 
\begin{align*}
    \mathcal{P}=\sum_{\deg_{\db^1}(\alpha)\leq \kappa} a_{\alpha} (W^1)^{\alpha},
\end{align*}
where $a_{\alpha}\in \MM^{D_1\times D_2}(\RR)$ are constant matrices. Fix $B\geq 0$ such that $\max_{\alpha}|\alpha| \leq B$.
\begin{prop}\label{Scaling:Scaled Estimates: scaled maximally subelliptic estimate}
    Suppose $\mathcal{P}$ is maximally subelliptic of degree $\kappa$ with respect to $(W,\db^1)$, in the sense that $\exists A\geq 0$ with 
    \begin{align*}
        \sum_{j=1}^{r_1} \|(W_j^1)^{n_j}f\|_{L^2(B^n(1), h\sigma_{Leb})}\leq A (\|\mathcal{P}f\|_{L2(B^{n}(1), h\sigma_{Leb})}+\|f\|_{L^2(B^{n}(1), h\sigma_{Leb})}),
    \end{align*}
    $\forall f\in C_{0}^{\infty}(B^{n}(31/32);\CC^{D_2})$. Fix $\sv\in \RR^{\nu}$. For every $x\in B^n(7/8)$ and $\delta\in (0,1]$, we have 
\begin{align}
\|\phi_{x,\xi_4\delta}u\|_{L^{p}_{\sv+\kappa\ef_1}(\delta^{\lambda \dv}W, \dv)}\leq C (\|\phi_{x, \frac{\xi_2\xi_4}{\xi_3}}\delta^{2\kappa} \mathcal{P}^*\mathcal{P}u\|_{L^{p}_{\sv-\kappa\ef_1}(\delta^{\lambda\dv}W, \dv)}+ \|\phi_{x, \frac{\xi_2\xi_4}{\xi_3}\delta}u\|_{L^{p}_{\sv-\vec{N}}(\delta^{\lambda\dv}W, \dv)})\quad \forall \ \vec{N}\in [0,\infty)^{\nu},
    \end{align}
    $\forall u\in C_0^{\infty} (B^n(1); \CC^{D_2})'$ with $\mathcal{P}^*\mathcal{P}u\in L^p_{\sv-\kappa\ef_1}(\widebar{B^{n}(7/8)},(W,\dv))$. Here $C=C(A,B,D_1,D_2)\geq 0$ is an $(p,\sv, \kappa)$-multi-parameter unit-admissible constant. 
\end{prop}
\begin{proof}
    By Theorem \ref{Preliminaries: Theorem 8.1.1} (i) $\implies$ (iv), $\pp^*\pp$ is maximally subelliptic of degree $2\kappa$ wtih respect to $(W, \dv)$, in the sense that 
    \begin{align}
        \sum_{j=1}^{r_1} \|(W_j^1)^{2n_j} f\|_{L^2(B^n(1), h\sigma_{Leb})}\lesssim \|\pp^*\pp\||_{L^2(B^n(1), h\sigma_{Leb})} +\|f\||_{L^2(B^n(1), h\sigma_{Leb})},
    \end{align}
    $\forall f\in C_{0}^{\infty}(B^n(15/16))$.

    For $\xi_0\in \{\xi_4, \xi_2\xi_4/\xi_3\}$, set $\psi_{x, \xi_0\delta}:= \phi_{x, \xi_0\delta}\circ \Phi_{x, \delta}$. Theorem \ref{Scaling: theorem 9.2.8 [BS]} (v) and (vi) shows that $\psi_{x, \xi_0\delta}\in C_0^{\infty}(B^n(1))$ and for every $L\in \NN, \|\psi_{x, \xi_0\delta}\|_{C^L}\lesssim_L 1$, where the implicit constant is an $L$-multi-parameter unit-admissible constant. Theorem \ref{Scaling: theorem 9.2.8 [BS]} (iv) shows that $\psi_{x, \xi_4\delta}\prec \psi_{x, \frac{\xi_2\xi_4}{\xi_3}\delta}$.

    Set $\mathcal{L}^{x, \delta}:= \Phi_{x, \delta}^* \delta^{2\kappa} \pp^* \pp (\Phi_{x,\delta})_*$. By Lemma \ref{Scaling: Vector Fields and Norms: Similar to Lemma 9.2.7 of [BS]} we can prove the required estimate if we can prove 
    \begin{align}
        \|\psi_{x, \xi_4\delta}v\|_{L^{p}_{\sv+\kappa\ef_1}(W^{x, \delta}, \dv)}\lesssim \|\psi_{x, \frac{\xi_2\xi_4}{\xi_3}\delta}\mathcal{L}^{x, \delta}v\|_{L^{p}_{\sv-\kappa\ef_1}(W^{x, \delta}, \dv)} + \|\psi_{x, \frac{\xi_2\xi_4}{\xi_3}\delta} v\|_{L^{p}_{\sv-\vec{N}}(W^{x, \delta}, \dv)}, 
    \end{align}
    $\forall \vec{N}\in [0, \infty)^{\nu}, v\in C_0^{\infty}(B^n(1))'$ with $\psi_{x, \frac{\xi_2\xi_4}{\xi_4}\delta}v\in L^{p}_{\sv-\vec{N}}(W^{x, \delta}, \dv)$. 
\end{proof}
\begin{prop}\cite[Proposition 9.2.10]{BS}\label{Scaling: Scaled Esimtates: Proposition 9.2.10 of [BS]}
    Let $E_{\infty}\in C^{\infty}(B^{n}(1)\times B^n(1))$ and fix $\xi_0\in (0,1]$. Then, $\forall \sv\in \RR^{\nu}$, there exists $L\in \NN, C_{\sv}$, such that $\forall x\in B^{n}(7/8), \delta\in (0,1]$, 
    \begin{align*}
        \|\phi_{x, \xi_0\delta}E_{\infty}\|_{\mathcal{B}_{\infty, \infty}^{\sv}(\delta^{\lambda \dv}W, \dv)}\leq C_{\sv} \|E_{\infty}\|_{C^{L}(B^n(1)\times B^n(1))}
\|u\|_{L^{\infty}},    \end{align*}
$\forall u\in L^{\infty}(B^n(1)),$ where $L=L(\sv)\in \NN$ is a $0$-multi-parameter unit-admissible constant and $C_{\sv}=C_{\sv}(\xi_0)\geq 0$ is an $\sv$-multi-parameter unit-admissible constant. 
\end{prop}
\begin{corollary}\label{scaling: Scaled Estimates: Corollary of Proposition 9.2.10 of [BS]}
    Let $E_{\infty}, \xi_0$ be as in Proposition \ref{Scaling: Scaled Esimtates: Proposition 9.2.10 of [BS]}. Then, 
    Then, $\forall \sv\in \RR^{\nu}$, there exists $L\in \NN, C_{\sv}$, such that $\forall x\in B^{n}(7/8), \delta\in (0,1]$, 
    \begin{align*}
        \|\phi_{x, \xi_0\delta}E_{\infty}\|_{L^{p}_{\sv}(\delta^{\lambda \dv}W, \dv)}\leq C_{\sv} \|E_{\infty}\|_{C^{L}(B^n(1)\times B^n(1))}
\|u\|_{L^{\infty}},    \end{align*}
$\forall u\in L^{\infty}(B^n(1)),$ where $L=L(\sv)\in \NN$ is a $0$-multi-parameter unit-admissible constant and $C_{p,\sv}=C_{p,\sv}(\xi_0)\geq 0$ is an $(p,\sv)$-multi-parameter unit-admissible constant. 
\end{corollary}
 \begin{proof}
        By using Lemma \ref{Preliminaries: embedding sobolev space in Zygmund Holder space} we see that  
        \begin{align*}
                   \|\phi_{x, \xi_0\delta}E_{\infty}\|_{L^p_{\sv}(\delta^{\lambda \dv}W, \dv)}   \leq C_{p,\sv} \|\phi_{x, \xi_0\delta}E_{\infty}\|_{\mathcal{B}_{\infty, \infty}^{\sv+(1,...,1)}(\delta^{\lambda \dv}W, \dv)}.
        \end{align*}
        Then the required estimate follows immediately. 
 \end{proof}
 
\begin{prop}\label{Scaling: Scaled Estimates: Proposition 9.2.11 of [BS]}
   Fix $\xi_0\in (0,1]$. Then, $\forall \sv\in \RR^{\nu}, \exists C_{\sv}\geq 0$ such that $\forall x\in B^{n}(7/8), \delta\in (0,1]$, 
   \begin{align*}
       \|\phi_{x, \xi_0\frac{\xi_3}{\xi_2}}\delta^{\kappa} \pp u\|_{L^p_{\sv}(\delta^{\lambda\dv} W, \dv)}\leq C_{\sv} \|\phi_{x, \xi_0\delta}u\|_{L^{p}_{\sv+\kappa \ef_1}(\delta^{\lambda\dv}W, \dv)}, 
   \end{align*}
   $\forall u\in C_0^{\infty}(B^n(1); \CC^{D_2})'$ with $\phi_{x, \xi_0\delta}u\in L^{p}_{\sv+\kappa\ef_1}(W, \dv)$. Similarly,
   \begin{align*}
       \|\phi_{x, \xi_0\frac{\xi_3}{\xi_2}}\delta^{\kappa} \pp^* u\|_{L^p_{\sv}(\delta^{\lambda\dv} W, \dv)}\leq C_{\sv} \|\phi_{x, \xi_0\delta}u\|_{L^{p}_{\sv+\kappa \ef_1}(\delta^{\lambda\dv}W, \dv)},
   \end{align*}
    $\forall u\in C_0^{\infty}(B^n(1); \CC^{D_2})'$ with $\phi_{x, \xi_0\delta}u\in L^{p}_{\sv+\kappa\ef_1}(W, \dv)$. Here, $C_{\sv}(\xi_0, D_1, D_2)\geq 0$ is an $(p,\sv, \kappa)$- multi-parameter unit-admissible constant. 
\end{prop}
\begin{proof}
    Without loss of generality, we take $B=1$ and it suffices to consider only the case $D_1=D_2=1$.

    Both $\pp$ and $\pp^*$ are of the form 
    \begin{align*}
        \tilde{\pp}= \sum_{\deg_{\db^1}(\alpha)\leq \kappa} b_{\alpha} (W^1)^{\alpha},
    \end{align*}
    where $b_{\alpha}\in C^{\infty}(B^n(1))$ and for every $L, \|b_{\alpha}\|_{C^l(B^n(1))}\lesssim_L 1$, where the implicit constant is an $L$-multi-parameter unit-admissible constant. 

    For $x\in B^{n}(7/8), \delta\in (0,1]$, set
    \begin{align*}
        \tilde{\pp}^{x, \delta}:= \Phi_{x,\delta}^* \delta^{\kappa} \tilde{\pp} (\phi_{x,\delta})_* =\sum_{\deg_{\db^1}(\alpha)\leq \kappa} b_{\alpha}^{x, \delta} (W^{1,x, \delta}),
    \end{align*}
    where $b_{\alpha}^{x, \delta}=\delta^{\kappa -\deg_{\db^1}(\alpha)}b_{\alpha}\circ \phi_{x, \delta}$. Using Theorem \ref{Preliminaries: Theorem 3.15.5 of [BS]} (j), we have $\forall L\in \NN$, 
    \begin{align*}
        \|b_{\alpha}^{x, \delta}\|_{C^{L}(B^n(1))}&\approx_L \sum_{|\beta|\leq L} \|(X^{1,x,\delta})^{\alpha} b_{\alpha}^{x,\delta})\|_{C(B^n(1))}\\
        &= \sum_{|\beta|\leq L} \delta^{\beta+\kappa-\deg_{\db^1}(\alpha)} \|(X^1)^{\alpha}b_{\alpha}\|_{C(\phi_{x,\delta}(B^{n}(1))}\lesssim_L 1.
    \end{align*}
    Combining the above and using the fact that $\psi_{x, \xi_0\frac{\xi_3}{\xi_2}\delta}\tilde{\pp}^{x, \delta}\prec \psi_{x, \xi_0}$(Theorem \ref{Scaling: theorem 9.2.8 [BS]} (iv)) it follows from Proposition \ref{Preliminaries: Besov and Triebel-Lizorkin space: Prposition 6.5.9 of BS} that 
    \begin{align}\label{Scaling: Scaled Estimates: Estimate in Lemma 9.2.11 in [BS]}
        \|\psi_{x, \xi_0\frac{\xi_3}{\xi_2}\delta}\tilde{\pp}^{x, \delta} v\|_{L^{p}_{\sv}(W^{x, \delta}, \dv)}&= \|\psi_{x, \xi_0\frac{\xi_3}{\xi_2}\delta}\tilde{\pp}^{x, \delta} \psi_{x, \xi_0\delta}v\|_{L^{p}_{\sv}(W^{x,\delta}, \dv)}\nonumber\\
        &\lesssim_{\sv, \kappa} \|\psi_{x, \xi_0\delta}v\|_{L^{p}_{\sv+\kappa\ef_1}(W^{x, \delta}, \dv)},
    \end{align}
    $\forall v\in C_{0}^{\infty}(B^n(1))'$ with $\psi_{x, \xi_0\delta}v\in L^{p}_{\sv+\kappa\ef_1}(W^{x, \delta}, \dv)$. 

    Using Lemma \ref{Scaling: Vector Fields and Norms: Similar to Lemma 9.2.7 of [BS]} and \eqref{Scaling: Scaled Estimates: Estimate in Lemma 9.2.11 in [BS]}, 
    \begin{align*}
        \|\psi_{x, \xi_0\frac{\xi_3}{\xi_2}\delta}\delta^{\kappa} \tilde{\pp}u\|_{L^{p}_{\sv}(\delta^{\lambda\dv}W, \dv)}&= \Lambda(x, \delta)^{1/p}\|\phi_{x, \xi_0\frac{\xi_3}{\xi_2}\delta} \tilde{\pp}^{x, \delta} \Phi_{x, \delta}^* u\|_{L^{p}_{\sv}(W^{x, \delta}, \dv)}\\
        &\lesssim_{\sv, \kappa}\Lambda(x, \delta)^{1/p} \|\psi_{x, \xi_0\delta}\Phi_{x, \delta}^* u\|_{L^{p}_{\sv+\kappa\ef_1}(W^{x, \delta}, \dv)}\\
        &= \|\phi_{x, \xi_0\delta}u\|_{L^{p}_{\sv+\kappa\ef_1}(\delta^{\lambda\dv}W, \dv)},
    \end{align*}
    which completes the proof. 
\end{proof}
\begin{lemma}\label{Scaling: Scaled Estimates: Lemma 9.2.12 in [BS]}
Let $1<p<\infty$, $\sv\in (0,\infty)^{\nu}, t>0$. Then for  $\psi\in C_{0}^{\infty}(B^n(7/8))$, 
\begin{align}
    &\|\psi\Phi_{x,\delta}^* u\|_{L^p_{\sv}(W^{x,\delta},\dv)}\leq C_{p,\sv} \|\psi\|_{C^{L}}\|u\|_{L^p_{\sv}(W,\dv)}\quad \forall u\in L^p_{\sv}(\widebar{B^{n}(7/8)}, (W,\dv)) \label{Scaling: Scaled Estimates: Lemma 9.2.12 in [BS] Sobolev space},
\end{align}
where $L\in \NN$ and $C_{p,\sv}$ are $(p,\sv)$-multi-parameter unit-admissible constants. Similarly, for $F\in \cc^{\sv, t}((W, \dv)\boxtimes\nabla_{\RR^N})$ and $\psi\in C_{0}^{\infty}(B^{n}(7/8))$,
    \begin{align} \label{Scaling: Scaled Estimates: Lemma 9.2.12 for function in two variables}
        \|\psi(\cdot) F(\phi_{x,\delta}(\cdot), \zeta)\|_{\cc^{\sv, t}(W^{x, \delta}, \dv)\boxtimes \nabla_{\RR^N}}\leq C_{\sv, t} \|\psi\|_{C^L}\|F\|_{\cc^{\sv, t} (W, \dv)\boxtimes\nabla_{\RR^N}},
    \end{align}
    where $L\in \NN$ and $C_{\sv, t}=C_{\sv, t}(N)\geq 0$ are $(\sv, t)$-multi-parameter unit-admissible constants. 
\end{lemma}
\begin{remark}
    A similar version of \eqref{Scaling: Scaled Estimates: Lemma 9.2.12 for function in two variables} can also be written down for Sobolev space. However, we do not mention it here as we do not require it for our purpose. 
\end{remark}
\begin{proof}
    We will only prove \eqref{Scaling: Scaled Estimates: Lemma 9.2.12 in [BS] Sobolev space}. For a proof of \eqref{Scaling: Scaled Estimates: Lemma 9.2.12 for function in two variables} we refer the reader to \cite[Lemma 9.2.12]{BS}. 

    For $ u\in L^p_{\sv}(\widebar{B^{n}(7/8)}, (W,\dv))$, by Proposition \ref{Preliminaries: smooth function decomposition for Sobolev space}, $\exists \{u_j\}\subset C_0^{\infty}(B^n(1))$ such that $u =\sum_{j\in \NN^{\nu}} u_j$ and for every $\alpha$, 
    \begin{align*}
        \left\|\left(\sum_{j\in \NN^{\nu}}|2^{j\cdot \sv} (2^{-j\cdot \dv}W)^{\alpha}u_j|^2\right)^{1/2}\right\|_{L^p(\mathcal{M})} \lesssim_{\sv, \alpha}  \|u\|_{L^p_{\sv}(W,\dv)}
    \end{align*}
Since $\delta\leq 1$, 
    \begin{align*}
  \left  \|\left(\sum_{j\in \NN^{\nu}} |2^{j\cdot \sv} (2^{-j\cdot \dv}W^{x,\delta})^{\alpha} \Phi_{x,\delta}^*u_j|^2\right)^{1/2}\right\|_{L^{p}(B^n(1))}&\leq \left\|\left(\sum_{j\in \NN^{\nu}} |2^{j\cdot \sv} \delta^{\deg_{\db^1} (\alpha)}(2^{-j\cdot \dv}W)^{\alpha} u_j|^2\right)^{1/2}\right\|_{L^p(B^{n}(1))}\\
&\leq \left\|\left(\sum_{j\in \NN^{\nu}}|2^{j\cdot \sv} (2^{-j\cdot \dv}W)^{\alpha}u_j|^2\right)^{1/2}\right\|_{L^p(\mathcal{M})} \lesssim_{\sv, \alpha}  \|u\|_{L^p_{\sv}(W,\dv)}
    \end{align*}
    It then follows that, for every $L\in \NN$
    \begin{align*}
    \sum_{|\alpha|\leq L}  \left  \| \left(\sum_{j\in \NN^{\nu}} \left|2^{j\cdot \sv} (2^{-j\cdot \dv} W^{x,\delta})^{\alpha} \psi \Phi_{x,\delta}^* u_j\right|^2\right)^{1/2}\right\|_{L^{p}(B^{n}(1))}\lesssim_{\sv, L}\|u\|_{L^p_{\sv}(W, \dv)}.
    \end{align*}
    Since $\psi \Phi_{x,\delta}^* u=\sum_{j\in \NN^{\nu}} \psi \Phi_{x, \delta}^* u_j$ \eqref{Scaling: Scaled Estimates: Lemma 9.2.12 in [BS] Sobolev space} follows from Proposition \ref{Preliminaries: smooth function decomposition for Sobolev space}. 
\end{proof}
\begin{corollary}\label{Scaling: Scaled Estimates: bounding scaled vector field norm by usual norm}
    Fix $\xi_0\in (0,1]$ and $\sv\in (0,\infty)^{\nu}$. Then, $\forall x\in B^n(7/8), \delta\in (0,1]$
    \begin{align}
       \|\phi_{x,\xi_0\delta}u\|_{L^p_{\sv}(\delta^{\lambda \dv}W, \dv)} \leq C_{p,\sv} \|u\|_{L^{p}_{\sv}(W,\dv)}\quad  \forall \ u \in L^p_{\sv}(\widebar{B^{n}(7/8)}, (W,\dv)),
        \end{align}
        where $C_{p,\sv}=C_{p,\sv}(\xi_0)\geq 0 $ is a $(p,\sv)$-multi-parameter unit-admissible constant respectively. 
\end{corollary}
\begin{prop}\label{Scaling: Scaled Estimates: Proposition similar to 9.2.14 in [BS]}
    Fix $\sv\in (0,\infty)^{\nu}, M>0, \xi_0\in (0,1], N\in \NN_+.$ Fix $F(x, \zeta)\in L^p_{\sv, M}(\widebar{B^{N}(7/8)}\times \RR^N, (W, \dv)\nabla_{\RR^N})$ and $u\in L^{p}_{\sv}(\widebar{B^{n}(7/8), (W, \dv)}; \RR^N)$. Then, $\forall x\in B^{n}(7/8)$ and $\delta\in (0,1]$, the following hold:
    \begin{enumerate}[label=(\roman*)]
        \item If $M\geq \nu(|\sv|_{\infty}+1)$, then 
        \begin{align*}
            \|\phi_{x, \xi_0\delta}(\cdot)&F(\cdot, \phi_{x, \xi_0\delta}(\cdot)u(\cdot))\|_{L^{p}_{\sv}(\delta^{\lambda\dv}W, \dv)}\\
           & \leq C_{pm\sv, M} \|F\|_{\cc^{\sv,M}((W, \dv)\nabla_{\RR^N})} (1+ \|\phi_{x, \xi_0\delta}u\|_{L^{p}_{\sv}(W, \dv)}) (1+ \|\phi_{x, \xi_0\delta}u\|_{L^{\infty}})^{M+\nu-1},
        \end{align*}
        where $C_{\sv, M}=C_{p,\sv, M}(\xi_0, N)\geq 0$ is an $(\sv, M)$-multi-parameter unit-admissible constant. 
        \item Fix $L
        \in \NN_+$ and suppose that $\partial_{\zeta}^{\beta}F(x, 0)\equiv 0, \forall |\beta|< L,$ and $M\geq \nu(|\sv|_{\infty}+1)+L$. Then
        \begin{align*}
              \|\phi_{x, \xi_0\delta}(\cdot)&F(\cdot, \phi_{x, \xi_0\delta}(\cdot)u(\cdot))\|_{L^{p}_{\sv}(\delta^{\lambda\dv}W, \dv)}\\       
              \leq& C_{p,\sv, M, L}  \|F\|_{\cc^{\sv,M}((W, \dv)\nabla_{\RR^N})} (1+ \|\phi_{x, \xi_0\delta}u\|_{L^{\infty}})^{M+\nu-1}\|\phi_{x, \xi_0\delta}u\|_{L^{\infty}}^{L-1}\\&\times (\|\phi_{x, \xi_0\delta}u\|_{L^{\infty}} + (1+\|\phi_{x, \xi_0\delta}u\|_{L^{\infty}}) \|\phi_{x, \xi_0\delta}u\|_{L^{p}_{\sv}(\delta^{\lambda\dv}W, \dv)}),       \end{align*}
              where $C_{p,\sv, M, L}=C_{p,\sv, M, L}(\xi_0, N)\geq 0$ is an $(\sv, M, L)-$ multi-parameter unit admissible constant. 
    \end{enumerate}
\end{prop}
\begin{proof}
    (i) Let $\psi_{x_0, \xi_0\delta}:= \phi_{x, \xi_0\delta}\circ \phi_{x, \delta}$, 
    \begin{align*}
        \|\phi_{x, \xi_0\delta}&(\cdot) F(\cdot, \phi_{x, \xi_0\delta}u(\cdot))\|_{L^{p}_{\sv}(\delta^{\lambda\dv}W, \dv)}\\
        =& \|\psi_{x, \xi_0\delta}(\cdot)F(\Phi_{x, \delta}(\cdot), \psi_{x, \xi_0\delta}(\cdot) \phi_{x, \delta}^*u(\cdot))\|_{L^p_{\sv}(W^{x, \delta}, \dv)}\\
        \lesssim_{\sv, M}&  \|\psi_{x, \xi_0\delta}(\cdot) F(\Phi_{x, \delta}(\cdot), \zeta)\|_{\cc^{\sv, M}((W^{x, \delta}, \dv)\boxtimes \nabla_{\RR^N})} (1+ \|\psi_{x, \xi_0\delta} \Phi_{x, \delta}^* u\|_{L^{p}_{\sv}(W^{x,\delta}, \dv)}) \\
        &\times (1+\|\psi_{x, \xi_0\delta}\Phi_{x,\delta}^* u\|_{L^{\infty}})^{M+\nu-1}.
    \end{align*}
   Using definition of $\psi_{x, \xi_0\delta}$
    \begin{align*}
        \|\psi_{x, \xi_0\delta} \phi_{x,\delta}^* u\|_{L^{\infty}}=\|\phi_{x,\delta} u\|_{L^{\infty}},
    \end{align*}
    and Lemma \ref{Scaling: Vector Fields and Norms: Similar to Lemma 9.2.7 of [BS]} we have
    \begin{align*}
        \|\psi_{x, \xi_0\delta} \Phi_{x,\delta}^* u\|_{L^{p}_{\sv}(W^{x,\delta}, \dv)}= \|\phi_{x, \xi_0\delta} u\|_{L^{p}_{\sv}(\delta^{\lambda\dv}W, \dv)}.
    \end{align*}
    So, if we show that 
    \begin{align}
        \|\psi_{x, \xi_0\delta}(\cdot)F(\Phi_{x,\delta}(\cdot), \zeta)\|_{\cc^{\sv, M}((W^{x, \delta}, \dv)\nabla_{\RR^N})} \lesssim_{\sv, M} \|F\|_{\cc^{\sv, M}((W, \dv)\nabla_{\RR^N})}.
    \end{align}
    The above inequality follows from \eqref{Scaling: Scaled Estimates: Lemma 9.2.12 for function in two variables} with $\psi=\psi_{x, \xi_0\delta}= \phi_{x, \xi_0\delta}\circ \phi_{x, \delta}$, where we use Theorem \ref{Scaling: theorem 9.2.8 [BS]} (v) and (vii) to see that $\phi_{x, \xi_0\delta}\circ \phi_{x, \delta}\in C_{0}^{\infty}(B^n(1/2))$ and $\|\phi_{x, \xi_0\delta}\circ \Phi_{x, \delta}\|_{C^L}\lesssim 1$, $\forall L\in \NN$. 
\end{proof}
 \section{Reduction II}\label{Reduction II}
\numberwithin{equation}{section}
 Let 
\ba \pp = \sum_{\deg_{\db^1}(a_{\alpha})\leq \kappa } (W^1)^{\alpha},\ a_{\alpha}\in \MM^{D_1\times D_2}(\RR)\ea
be a maximally subelliptic operator of degree $\kappa $ with respect to $(W^1,\db^1)$ in the sense that 
\ba \sum_{j=1}^{r_1} \|(W_j^1)^{n_j}u\|_{L^2(B^n(1), h\sigma_{Leb})}\leq A \left(\|\pp u\|_{L^2(B^n(1), h\sigma_{Leb})}+\|u\|_{L^2(B^n(1), h\sigma_{Leb})}\right),\ea
$\forall \ u \in C_0^{\infty}(B^n(31/32); \CC^{D_2})$. Here $a_{\alpha}\in \MM^{D_1\times D_2}(\RR)$ are constant matrices. Fix $B\geq 0$ with $|a_{\alpha}|\leq B,\ \forall a_{\alpha}$.\\

Let $\epsilon_1\in (0,1]$ be a small number to be chosen later and fix $C_{15}\geq 0$. Suppose that $R_1\in \cc^{\sv,M}(\widebar{B^{n}(7/8)}\times \RR^N, (W,\dv)\boxtimes \nabla_{\RR^N};\RR^{D_1})$, and $R_2\in \cc^{\sv, M+1}(\widebar{B^{n}(7/8)}\times \RR^N, (W,\dv)\boxtimes \nabla_{\RR^N};\RR^{D_1})$, where $\nu(|\sv|_{\infty}+1)+2\leq M\in \NN_+$. We also assume that 
\ba \|R_1\|_{\cc^{M,M}((W,\dv)\boxtimes \nabla_{\RR^N})}\epsilon_1, \|R_2\|_{\cc^{M,M}((W,\dv)\boxtimes \nabla_{\RR^N})}\leq C_{15}.\ea
We also suppose that 
\ba \partial_{\zeta}^{\beta}R_2(x,0)\equiv 0, \ \forall |\beta|\leq 1.\ea
Let $u\in \cc^{\rv+\kappa \mathfrak{e}_1}(\widebar{B^n(7/8)}, (W,\dv); \RR^{D_2})$ and $g\in \left(\cc^{\rv}\cap L^{p}_{\sv}\right)(\widebar{B^n(7/8)}, (W,\dv); \RR^{D_1})$ with 
\ba \|u\|_{\cc^{\rv+\kappa\ef_1}(W,\dv)}\leq \epsilon_1, \|g\|_{L^{p}_{\sv}(W,\dv)}, \|g\|_{\cc^{\rv}(W,\dv)}\leq \epsilon_1\ea
such that $u$ satisfies the nonlinear PDE:
\ba \pp u(x) = R_1(x,\ \mathfrak{W}_1^{\kappa} u(x)) + R_2(x, \Wf_1^{\kappa}u(x))-g(x),\ \forall x\in B^n(13/16).\ea
\begin{prop}
   There exists $\epsilon_1\in (0,1]$ sufficiently small such that if the above holds, then $\forall \psi_0\in C_0^{\infty} (B^n(2/3))$, we have
    \ba \psi_0 u \in L^p_{\sv +\kappa\ef_1}(\widebar{B^n(7/8)}, (W,\dv); \RR^{D_2})\ea
    and 
    \ba \|\psi_0 u\|_{L^p_{\sv +\kappa \ef_1}(W,\dv)}\leq C_{16}.\ea
    Here, $\epsilon_1= \epsilon_1(C_1,A,B,D_1,D_2)\in (0,1]$ and $C_{16}=C_{16}(C_{15}, A, B, D_1,D_2,\psi_0)\geq 0$ are $(p,\kappa, \sv, \rv, M, \sigma)$- multi-parameter unit-admissible constants. 
\end{prop}
\begin{proof}
Without loss of generality we assume that $\sv\geq \rv$, as we may replace $\rv$ with $\rv\wedge \sv$. Fix $\psi_1,\psi_2,\psi_3\in C_{0}^{\infty}(B^{n}(13/16))$ with $\psi_3\prec \psi_2\prec \psi_1$ and $\psi_3\equiv 1$ in a neighbourhood of $\widebar{B^n(3/4)}$.

   Since $\pp$ is maximally subelliptic, \ref{Preliminaries: Theorem 8.1.1} part (viii) implies that there exists \\$T\in \mathcal{A}_{loc}^{-2\kappa\ef_1}((W,\dv); \CC^{D_2},\CC^{D_2})$ such that 
   \begin{align}
       T\pp^*\pp, \pp^*\pp T \equiv I \ \mod C^{\infty}(B^n(31/32)\times B^n(31/32); \MM^{D_1\times D_2}(\CC)).
   \end{align}
 Define
\begin{align}\label{definition of V infinity}
    V_{\infty}:&= \psi_2 T\pp^*\psi_1\left(R_1(x,\ \mathfrak{W}_1^{\kappa} u(x)) + R_2(x, \Wf_1^{\kappa}u(x))-g(x)\right)\nonumber\\
    &=\psi_2 T\pp^*\psi_1\pp u\nonumber\\
    & = \psi_2 T \pp^*\pp u + \psi_2 T \pp^*(1-\psi_1)\pp u \nonumber\\&= \psi_2 u + \psi_2 e_{\infty}
 \end{align}
 where by the pseudo-locality(see Theorem \ref{Preliminaries: Function spaces: Singular integrals: Theorem 5.8.18}) of $T$ we have $e_{\infty}\in C_0^{\infty}(B^n(7/8))$. Moreover, using $supp(u)\subset \widebar{B^{n}(7/8)}$, Proposition \ref{Preliminaries: Function spaces: Singular integrals: prop 5.8.11}, Theorem \ref{Preliminaries: Function spaces: boundedness of singular integrals, Theorem 6.3.10 of BS} and Lemma \ref{Preliminaries: embedding sobolev space in Zygmund Holder space} we see that 
 \ba
 \|e_{\infty}\|_{\cc^{\rv+\kappa \ef_1}(W,\dv)}, \|e_{\infty}\|_{\lps}\lesssim \|e_{\infty}\|_{\cc^{\sv+\kappa\ef_1+(1,...,1)}}\lesssim \|u\|_{\cc^{\rv+ \kappa\ef_1}(W,\dv)}\lesssim \epsilon_1.
 \ea 
 Since, $T\pp^* \psi_1\pp\in \mathcal{A}_{loc}^{0_{\nu}}((W,\dv); \CC^{D_2}, \CC^{D_1})$ ($0_{\nu}:=(0,..,0)\in \NN^{\nu}$) and using the fact that $supp(u)\subset \overline{B^n(7/8)}$, Proposition \ref{Preliminaries: Function spaces: Singular integrals: prop 5.8.9, corollary 5.8.10} and Theorem \ref{Preliminaries: Function spaces: boundedness of singular integrals, Theorem 6.3.10 of BS} shows that 
 \ba 
 \|V_{\infty}\|_{\cc^{\rv+\kappa\ef_1}(W,\dv)}= \|\psi_2 T\pp^* \pp u\|_{\cc^{\rv+\kappa \ef_1}(W,\dv)} \lesssim \|u\|_{\cc^{\rv+\kappa \ef_1}(W,\dv)}\lesssim \epsilon_1. 
 \ea 
 So, there exists $C_{17}\approx 1$ with 
 \begin{align}\label{Reduction II: V_{infty} estimate}
 \|V_{\infty}\|_{\cc^{\rv+\kappa\ef_1}(W,\dv)} \leq C_{17} \epsilon_1. 
 \end{align}
 Set $H:= u-V_{\infty}$. We have $H\in \mathcal{C}^{\rv+\kappa\ef_1}(\overline{B^n(7/8))}, (W,\dv))$ and
 \ba 
 ||H||_{\crk} \leq \|u\|_{\crk}+\|V_{\infty}\|_{\crk} \lesssim \epsilon_1.
 \ea
 Also, since $\psi_3 H= \psi_3 e_{\infty}$, using Corollary \ref{Preliminaries: Besov and Triebel-Lizorkin space: Corollary 6.5.10 of BS} and \ref{Preliminaries: embedding sobolev space in Zygmund Holder space} we have 
 \begin{align}\label{Reduction II: psi_3H estimate}
\|\psi_3 H\|_{\lps} \lesssim \|\psi_3 H\|_{L^{p}_{\sv+\kappa\ef_1+ (1,...,1)}(W,\dv)}\lesssim \|e_{\infty}\|_{\lps} \lesssim \epsilon_1.
 \end{align}
 For $V\in \cc^{\rv+\kappa\ef_1}(\overline{B^{n}(7/8)}, (W,\dv))$, set 
 \ba 
  \mathcal{T}(V):= Mult [\psi_2] T\pp^*\psi_1(\cdot) (R_1(\cdot, \mathfrak{W}_1^{\kappa}(V+H)(\cdot))+ R_2(\cdot, \mathfrak{W}_1^{\kappa}(V+H)(\cdot))).
 \ea
 Now, it follows from the fact that $H\in \cc^{\rv+\kappa\ef_1}(\overline{B^n(7/8)}, (W,\dv))$, Theorem \ref{Preliminaries: Function spaces: boundedness of singular integrals, Theorem 6.3.10 of BS}, and Corollary \ref{tame estimate for composition maximally subellitpic case proposition corollary}. We also have from \eqref{definition of V infinity} that $\mathcal{T}(V_{\infty})= V_{\infty}$. 

 Now, let $D\in [C_3\vee 1, 1/\epsilon_1]$, to be chosen later; we will choose $D\approx 1$. Define 
 \ba 
 \mathcal{M}_{D,\epsilon_1,\rv}:= \{V\in \cc^{\rv+\kappa\ef_1}(\overline{B^{n}(7/8)}, (W,\dv)): \|V_{\crk}\|\leq D\epsilon_1\}. 
 \ea
 Note that $\mathcal{M}_{D,\epsilon_1,\rv}$ is a closed subspace of the Banach space $\cc^{\rv+\kappa\ef_1}(\overline{B^{n}(7/8)}, (W,\dv))$, and therefore it is a complete metric space induced by the norm $\|\cdot\|_{\crk}$. Furthermore, since $D\geq C_{17}$ by choice, \eqref{Reduction II: V_{infty} estimate} shows that $V_{\infty}\in \mathcal{M}_{D,\epsilon_1,\rv}$. We will show that $\mathcal{T}: \mathcal{M}_{D,\epsilon_1, \rv}\to \mathcal{M}_{D,\epsilon_1, \rv}$ and that is a strict contraction for appropriate choice of $D$ and $\epsilon_1$. This will show that $V_{\infty}$ is a unique fixed point of $\mathcal{T}$ in $\mathcal{M}_{D,\epsilon_1, \rv}$. \\

 For $V\in \mathcal{M}_{D,\epsilon_1,\rv}$, we have using Corollary \ref{Preliminaries: Besov and Triebel-Lizorkin space: Corollary 6.5.11 of BS} and \eqref{Reduction II: psi_3H estimate}, $D\geq 1$ that 
 \begin{align}\label{Reduction II: estimate for Wf_1^{kappa}(V+H) in zygmund Holder space}
     \|\Wf_1^{\kappa} (V+H)\|_{\crk} &\lesssim \|V+H\|_{\crk}\nonumber\\
     &\lesssim \|V\|_{\crk}+ \|H\|_{\crk}\nonumber\\
     &\lesssim D\epsilon_1+ \epsilon_1 \nonumber \\
     &\lesssim D\epsilon_1. 
 \end{align}
 In particular, using Remark \ref{Preliminaries: Function spaces: Zygmund Holder space: Inclusion into continuous functions}, we have 
 \begin{align}\label{Reduction II: estimate for Wf_1^{kappa} in L^{infty} norm}
     \|\Wf_1^{\kappa}(V+H)\|_{L^{\infty}}\lesssim \|\Wf_1^{\kappa}(V+H) \|_{\cc^{\rv}(W,\dv)}\lesssim D\epsilon_1.  
 \end{align}
 By applying Theorem \ref{Preliminaries: Function spaces: boundedness of singular integrals, Theorem 6.3.10 of BS} and Theorem \ref{Appendix B: Theorem 7.5.2 of BS} part (i) to $R_1$ and $R_2$ we get 
 \ba 
 &\|\mathcal{T} (V)\|_{\crk} \\
 =& \|\psi_2 T\mathcal{P}^* \psi_1 (R_1(x,\Wf_1^{\kappa}(V+H)(x))+ R_2(x,\Wf_1^\kappa(V+H)(x)-g(x))\|_{\crk}\\
 \lesssim & \|R_1(x,\Wf_1^{\kappa}(V+H)(x))\|_{\crk}+\| R_2(x,\Wf_1^\kappa(V+H)(x)\|_{\crk}+\|g(x)\|_{\crk}\\
 \lesssim & \|R_1\|{\cc^{\rv, \lfloor|\rv|_1\rfloor+1+\nu+\sigma} ((W,\dv)\boxtimes \nabla_{\RR^N})} (1+\|\Wf_1^{\kappa}(V+H)\|_{\cc^{\rv}(W,\dv)}) (1+\|\Wf_1^{\kappa} (V+H)\|_{L^{\infty}})^{\lfloor |\rv|_1\rfloor+\nu}+\\
 & \|R_2\|_{\cc^{\rv, \lfloor|\rv|_1\rfloor+3+ \nu+\sigma} ((W,\dv)\boxtimes \nabla_{\RR^N})} (1+\|\Wf_1^{\kappa}(V+H)\|_{\cc^{\rv}(W,\dv)})^{\nu +\lfloor|\rv|_1\rfloor+1} \|\Wf_1^{\kappa}(V+H)\|_{L^{\infty}}\times\\&\|\Wf_1^{\kappa}(V+H)\|_{\cc^{\rv}(W,\dv)}
 +\|g\|_{\cc^{\rv}(W,\dv)}\\
 \lesssim & \epsilon_1 (1+D\epsilon_1)(1+D\epsilon_1)^{\lfloor |\rv|_1\rfloor+1} + (1+D\epsilon_1)^{\lfloor |\rv|_1\rfloor+\nu+1}(D\epsilon_1)^2 +\epsilon_1\\
 \lesssim & \epsilon_1+D^2\epsilon_1^2+\epsilon_1,
 \ea
 where we used \eqref{Reduction II: estimate for Wf_1^{kappa}(V+H) in zygmund Holder space}, \eqref{Reduction II: estimate for Wf_1^{kappa} in L^{infty} norm}, and $D\leq 1/\epsilon_1$. So, there exists $C_{18}\approx 1$ with 
 \begin{align}\label{Reduction II: mathcal(T) estimate using constant C_{18}}
     \|\mathcal{T}(V)\|_{\crk}\leq C_{18}(2\epsilon_1 +D^2\epsilon_1^2),\quad \forall \ V\in \mathcal{M}_{D,\epsilon_1, \rv}.  
 \end{align}
 We take $D:= \max\{C_{17}, 3C_{18}, 1\}$ so that $D\approx 1$, and we take $\epsilon_1\in (0,1/D^2]$ to be chosen later. Then, $C_{18}(\epsilon_1+\epsilon_1^2D^2)\leq D\epsilon_1$ and \eqref{Reduction II: mathcal(T) estimate using constant C_{18}} implies that 
 \begin{align*}
     \|\mathcal{T}(V)\|_{\crk}\leq D\epsilon_1, \quad \forall \ V\in \mathcal{M}_{D,\epsilon_1, \rv}.
 \end{align*}
 i.e, $\mathcal{T}: \mathcal{M}_{D,\epsilon_1, \rv}\to \mathcal{M}_{D,\epsilon_1, \rv}$. Now, that we have chosen $D\approx 1$, \eqref{Reduction II: estimate for Wf_1^{kappa}(V+H) in zygmund Holder space} and \eqref{Reduction II: estimate for Wf_1^{kappa} in L^{infty} norm} becomes 
 \begin{align}\label{Reduction II: Wf_1^{kappa}(V+H) estimate rewritten}
     \|\Wf_1^{\kappa}(V+H)\|_{\cc^{\rv}(W,\dv)}, \|\Wf_1^{\kappa}(V+H)\|_{L^{\infty}}\lesssim \epsilon_1, \quad \forall V\in \mathcal{M}_{D,\epsilon_1, \rv}. 
 \end{align}
 Using \eqref{Reduction II: Wf_1^{kappa}(V+H) estimate rewritten}, Theorem \ref{Preliminaries: Function spaces: boundedness of singular integrals, Theorem 6.3.10 of BS} and Theorem \ref{Appendix B: Theorem 7.5.2 of BS} part (ii) we have 
 \begin{align*}
     \|\mathcal{T}&(V_1)-\mathcal{T}(V_2)\|_{\crk}\\
     =&\|\psi_2 T \mathcal{P}^* \psi_1 (R_1(x, \Wf_1^{\kappa} (V_1+H)(x))-R_1(x, \Wf_1^{\kappa})(V_2+H)(x))\\+&R_2(x, \Wf_1^{\kappa}(x,\Wf_1^{\kappa}(V_1+H)(x))-R_2(x, \Wf_1^{\kappa}(V_2+H)(x)))\|_{\crk}\\
     \lesssim& \|R_1(x, \Wf_1^{\kappa}(V_1+H)(x))-R_1(x, \Wf_1^{\kappa}(V_2+H)(x))\|_{\cc^{\rv}(W,\dv)}+ \\
     &\|R_2(x, \Wf_1^{\kappa}(V_1+H)(x))-R_2(x, \Wf_1^{\kappa}(V_2+H)(x))\|_{\cc^{\rv}(W,\dv)}\\
     \lesssim& \|R_1\|_{\cc^{\rv, \lfloor |\rv|_1\rfloor+2+\nu +\sigma}((W,\dv)\boxtimes \nabla_{\RR^N})} (\|\Wf_1^{\kappa}(V_1-V_2)\|_{\cc^{\rv}(W,\dv)}\\
     &+(\|\Wf_1^{\kappa}(V_1+H)\|_{\cc^{\rv}(W,\dv)}+\|\Wf_1^{\kappa}(V_2+H)\|_{\cc^{\rv}(W,\dv)})\|\Wf_1^{\kappa}(V_1-V_2)\|_{L^{\infty}})\\
     &\times (1+ \|\Wf_1^{\kappa}(V_1+H)\|_{L^{\infty}} +\|\Wf_1^{\kappa} (V_2+H)\|_{L^{\infty}})^{\lfloor |\rv|_1\rfloor+\nu}\\
     &+ \|R_2\|_{\cc^{\rv, \lfloor |\rv|_1\floor+3+\nu+\sigma} (W,\dv)\boxtimes \nabla_{\RR^N}} (1+ \|\Wf_1^{\kappa}(V_1+H)\|_{L^{\infty}} +\|\Wf_1^{\kappa} (V_2+H)\|_{L^{\infty}})^{\lfloor |\rv|_1\rfloor+\nu+1}\\
     &\times (\|\Wf_1^{\kappa}(V_1+H)\|_{\cc^{\rv}(W,\dv)}+\|\Wf_1^{\kappa}(V_2+H)\|_{\cc^{\rv}(W,\dv)})\|\Wf_1^{\kappa}(V_1-V_2)\|_{L^{\infty}}\\
     &+ \|R_2\|_{\cc^{\rv, \lfloor |\rv|_1\floor+3+\nu+\sigma} (W,\dv)\boxtimes \nabla_{\RR^N}} 
      (\|\Wf_1^{\kappa}(V_1+H)\|_{L^{\infty}}+\|\Wf_1^{\kappa}(V_2+H)\|_{L^{\infty}})\|\Wf_1^{\kappa}(V_1-V_2)\|_{L^{\infty}}\\
      \lesssim & \epsilon_1(\|V_1-V_2\|_{\crk}+ \epsilon_1\|\Wf_1^{\kappa}(V_1-V_2)\|_{L^{\infty}})(1+\epsilon_1)^{\lfloor |\rv|_1\rfloor+\nu}\\
      &+ (1+\epsilon_1)^{\lfloor |\rv|_1+\nu+1\rfloor}\epsilon_1 \|\Wf_1^{\kappa}(V_1-V_2)\|_{L^{\infty}}+ \epsilon_1 \|\Wf_1^{\kappa}(V_1-V_2)\|_{L^{\infty}}\\
      \lesssim & \epsilon_1 \|V_1-V_2\|_{\crk}+\epsilon_1 \|\Wf_1^{\kappa}(V_1-V_2)\|_{L^{\infty}}\\
      \lesssim & \epsilon_1 \|V_1-V_2\|_{\crk},
 \end{align*}
 where we used Remark \ref{Preliminaries: Function spaces: Zygmund Holder space: Inclusion into continuous functions} and Corollary \ref{Preliminaries: Besov and Triebel-Lizorkin space: Corollary 6.5.11 of BS} to see that 
 \begin{align*}
     \|\Wf_1^{\kappa}(V_1-V_2)\|_{L^{\infty}}&\lesssim \|\Wf_1^{\kappa} (V_1-V_2)\|_{\crk}\\
     \lesssim & \|V_1-V_2\|_{\crk}.
 \end{align*}
 We conclude that $\exists$ $C_{19}\approx 1$ with 
 \begin{align*}
     \|\mathcal{T}(V_1)-\mathcal{T}(V_2)\|_{\crk} \leq C_{19} \epsilon_1 \|V_1-V_2\|_{\crk}.
 \end{align*}
 In particular, if $\epsilon_1\in (0,\min \{\frac{1}{2}C_{19}, \frac{1}{D^2}\})$, we have 
\begin{align*}
    \|\mathcal{T}(V_1)-\mathcal{T}(V_2)\|_{\crk}\leq \frac{1}{2}\|V_1-V_2\|_{\crk}
\end{align*}
and therefore 
\begin{align*}
    \mathcal{T}: \mathcal{M}_{D,\epsilon_1, \rv}\to \mathcal{M}_{D, \epsilon_1, \rv}
\end{align*}
is a strict contraction with unique fixed point $V_{\infty}$. Set $V_0:= 0$ and for $j\geq 0$ set $V_{j+1}:= \mathcal{T}(V_j)$. We have established $V_{j}\in \mathcal{M}_{D,\epsilon_1, \rv}, \forall j$ and the contraction mapping principle shows that $V_{j}\overset{j\to \infty}{\longrightarrow} V_{\infty}$ in $\cc^{\rv+\kappa \ef_1}(\widebar{B^{n}(7/8)}, (W,\dv))$. Since $V_{j}\in \mathcal{M}_{D,\epsilon_1, \rv}$ we have
\begin{align*}
    \|V_{j}\|_{\crk}\lesssim \epsilon_1, 
\end{align*}
and therefore we have
\begin{align}\label{Reduction II: Wf_1^{kappa}(V_j+H) estimate}
    \|\Wf_1^{\kappa} (V_j+H)\|_{\cc^{rv}(W,\dv)}, \|\Wf_1^{\kappa} (V_j+H)\|_{L^{\infty}} &\lesssim \|V_{j}\|_{\crk}+\epsilon_1\nonumber\\
    &\lesssim \epsilon_1,
\end{align}
$\forall j\in \NN$. Similarly,
\begin{align}\label{Reduction II: Wf_1^{kappa}(V_j) estimate}
    \|\Wf_1^{\kappa}V_j\|_{\crk}, \|\Wf_1^{\kappa}V_{j}\|_{L^{\infty}}\lesssim \|V_{j}\|_{\crk}\lesssim \epsilon_1,
\end{align}
$\forall j\in \NN$. 
\begin{lemma}
    For $\psi' \in C_{0}^{\infty}(B^{n}(13/16))$ with $\psi' \prec \psi_3$, we have 
    \begin{align*}
        \psi' V_{j}\in L^{p}_{\sv+\kappa\ef_1}(\widebar{B^{n}(7/8)}, (W,\dv)), \ \forall j\in \NN.
    \end{align*}
    \begin{proof}
        We prove the lemma by induction on $j$. Since $V_0=0$, the base case, $j=0$ is trivial. We assume the result for $j$ and prove it for $j+1$. 

        Take $\psi'', \psi'''
        \in C_{0}^{\infty}(B^{n}(13/16))$ with $\psi' \prec \psi'' \prec \psi''' \prec \psi_3$. By the inductive hypothesis, we have $\psi''' V_{j}\in L^{p}_{\sv+\kappa \ef_1} (\widebar{B^{n}(7/8)}, (W,\dv))$ and we wish to show that $\psi' V_{j}\in L^{p}_{\sv+\kappa \ef_1} (\widebar{B^{n}(7/8)}, (W,\dv))$.

        We have
        \begin{align*}
            \psi' V_{j+1}&= \psi' \mathcal{T}(V_j)\\
            =& \psi' T\mathcal{P}^* \psi_1 (R_{1}(x,\Wf_1^{\kappa}(V_{j}+H)(x))+R_2 (x,\Wf_1^{\kappa} (V_j+H)(x))- g(x))\\
            =& \psi' T\mathcal{P}^* (1-\psi'')\psi_1 (R_{1}(x,\Wf_1^{\kappa}(V_{j}+H)(x))+R_2 (x,\Wf_1^{\kappa} (V_j+H)(x))- g(x))\\
            &+  \psi' T\mathcal{P}^* \psi'\psi_1 (R_{1}(x,\Wf_1^{\kappa}(V_{j}+H)(x))+R_2 (x,\Wf_1^{\kappa} (V_j+H)(x))- g(x))\\
            =:& (I)+ (II).
        \end{align*}
        Since $\psi'\prec \psi''$ and $T\in \mathcal{A}_{loc}^{-2\kappa} ((W,\dv); \CC^{D_2}, \CC^{D_2})$ is a pseudo-local (see Theorem \ref{Preliminaries: Function spaces: Singular integrals: Theorem 5.8.18}), we have 
        \begin{align*}
            Mult [\psi'] T \mathcal{P}^* Mult [(1-\psi'')\psi_1] \in C_{0}^{\infty}(B^{n}(13/16)\times B^{n}(13/16)). 
        \end{align*}
        It then follows that 
        \begin{align*}
            (I)=& \psi' T\mathcal{P}^* (1-\psi'')\psi_1 (R_1(x, \Wf_1^{\kappa}(V_j+H)(x))+R_2(x, \Wf_1^{\kappa}(V_j+H))(x)-g(x))\\&\in C_{0}^{\infty}(B^{n}(13/16)) \subseteq L^{p}_{\sv+\kappa \ef_1}(\overline{B^{n}(13/16)}, (W,\dv)).
        \end{align*}
   So, to complete the proof, it suffices to show that $(II)\in L^{p}_{\sv+\kappa\ef_1}(\widebar{B^{n}(7/8)}, (W\dv))$. 

  Now, 
  \begin{align*}
      (II)= \psi' T \mathcal{P}^* \psi'' (R_1(x, \Wf_1^{\kappa}\psi'' (V_j+H)(x))+R_2(x, \Wf_1^{\kappa}(V_j+H)(x))-g(x)).
  \end{align*}
  By inductive hypothesis, $\psi''' V_{j}\in L^{p}_{\sv+\kappa\ef_1}(\widebar{B^{n}(7/8)}, (W,\dv))$ and since $\psi''' H= \psi''' \psi_3 H= \psi'''\psi_3e_{\infty}\in C_{0}^{\infty}(B^{n}(13/16))$ we also have $\psi''' H\in L^{p}_{\sv+\kappa \ef_1}(\widebar{B^{n}(7/8)}, (W,\dv))$ Corollary \ref{Preliminaries: Besov and Triebel-Lizorkin space: Corollary 6.5.11 of BS} shows that $\Wf_1^{\kappa} \psi''' (V_j+H)\in L^{p}_{\sv}(\widebar{B^{n}(7/8)}, (W,\dv))$.  Then, Corollary \ref{tame estimate for composition maximally subellitpic case proposition corollary} (i) implies that $R_1(x, \Wf_1^{\kappa}\psi''' (V_j+H)(x))+R_2(x, \Wf_1^{\kappa}(V_j+H)(x))-g(x)\in L^{p}_{\sv}(\widebar{B^{n}(7/8)}, (W,\dv))$. Since $Mult[\psi'] T\mathcal{P}^* Mult[\psi'']\in \tilde{\mathcal{A}}^{-\kappa\ef_1}(B^{n}(7/8), (W,\dv))$, Theorem \ref{Preliminaries: Function spaces: boundedness of singular integrals, Theorem 6.3.10 of BS} implies that $(II)\in L^{p}_{\sv+\kappa\ef_1}(\widebar{B^{n}(7/8)}, (W,\dv))$ as desired. 
    \end{proof}
\end{lemma}
Now, we use the function $\psi_{x,\delta}$ and constants $\xi_{2},\xi_{3}$ and $\xi_{4}$ from Theorem \ref{Scaling: theorem 9.2.8 [BS]}. Fix $x_0\in B^{n}(3/4)$ and $\delta\in (0,1]$ with $B_{(W^1, \db^1)}(x_0, \frac{\xi_2^3\xi_4}{\xi_3^3}\delta)\subseteq B^n(3/4)$.

Now, Lemma \ref{Scaling:Scaled Estimates: scaled maximally subelliptic estimate} gives 
\begin{align}\label{Reduction II: scaled maximally subelliptic estimate for V_{j+1}}
   \|\phi_{x,\xi_4\delta}&V_{j+1}\|_{L^{p}_{\sv+\kappa\ef_1}(\delta^{\lambda \dv}W, \dv)}\nonumber\\
   &\lesssim \|\phi_{x, \frac{\xi_2\xi_4}{\xi_3}}\delta^{2\kappa} \mathcal{P}*\mathcal{P}V_{j+1}\|_{L^{p}_{\sv-\kappa\ef_1}(\delta^{\lambda\dv}W, \dv)}+ \|\phi_{x, \frac{\xi_2\xi_4}{\xi_3}\delta}V_{j+1}\|_{L^{p}_{\sv-\vec{N}}}\quad \forall \ \vec{N}\in [0,\infty)^{\nu}
\end{align}
We estimate the two terms on the right of \eqref{Reduction II: scaled maximally subelliptic estimate for V_{j+1}}. Using Lemma \ref{Preliminaries: embedding sobolev space in Zygmund Holder space} if we choose $\vec{N}$ large enough (enough to choose $\vec{N}\geq \sv-\rv+ (1,...,1)-\kappa\ef_1$) we get 
\begin{align}
    \|V_{j+1}\|_{L^{p}_{\sv-\vec{N}}(\delta^{\lambda\dv}W, \dv)}\lesssim \|V_{j+1}\|_{\cc^{\rv+\kappa\ef_1}(\delta^{\lambda\dv}W, \dv)}. 
\end{align}
Thus using the above inequality, Corollary \ref{Scaling: Scaled Estimates: bounding scaled vector field norm by usual norm} we get 
\begin{align*}
    \|\phi_{x_0, \frac{\xi_2\xi_3}{\xi_4}}V_{j+1}\|_{L^{p}_{\sv-\vec{N}}}(\delta^{\lambda\dv}W, \dv)&\lesssim \|V_{j+1}\|_{\cc^{\rv+\kappa \ef_1}(\delta^{\lambda\dv}W, \dv)}\\
    &\lesssim \|V_{j+1}\|_{\crk}\lesssim \epsilon_1.
\end{align*}
To bound the first term on the right hand side of \eqref{Reduction II: scaled maximally subelliptic estimate for V_{j+1}}, we use the fact that $supp(\phi_{x_0, \frac{\xi_2}{\xi_3}{\xi_4}\delta})\subseteq B_{(W^1, \db^1)}(x_0, (\frac{\xi_2^3\xi_4}{\xi_3^3})\delta)\subseteq B^n(3/4)$ and therefore $\phi_{x_0, \frac{\xi_2\xi_4}{\xi_3}\delta}\prec \psi_2$ to see that 
\begin{align}
\phi_{x_0, \frac{\xi_2\xi_4}{\xi_3}\delta} \pp^*\pp Mult [\psi_2] \pp^* Mult[\psi_1]&= \phi_{x_0, \frac{\xi_2\xi_4}{\xi_3}{\xi_4}\delta} \pp^* \pp T Mult[\psi_1]\\
&= \phi_{x_0, \frac{\xi_2\xi_4}{\xi_3}\delta} \pp^* Mult[\psi_1]+ \phi_{x_0, \frac{\xi_2\xi_4}{\xi_3}\delta} E_{\infty} Mult[\psi_1],
\end{align}
where $E_{\infty} \in C_{0}^{\infty}(B^{n}(7/8)\times B^{n}(7/8))$ with $\|E_{\infty}\|_{C^{L}}\lesssim 1, \forall L\in \NN$. Thus, we have 
\begin{align}\label{Reduction II: (III)+(IV)}
 & \|  \phi_{x_0, \frac{\xi_2\xi_4}{\xi_3}\delta}\delta^{2\kappa} \pp^*\pp V_{j+1} \|_{L^{p}_{\sv-\kappa\ef_1}(\delta^{\lambda\dv} W, \dv)}\nonumber\\
 & =\|  \phi_{x_0, \frac{\xi_2\xi_4}{\xi_3}\delta}\delta^{2\kappa} \pp^*\pp \mathcal{T}(V_j) \|_{L^{p}_{\sv-\kappa\ef_1}(\delta^{\lambda\dv} W, \dv)}\nonumber\\
  &=\| \phi_{x_0, \frac{\xi_2\xi_4}{\xi_3}\delta}\delta^{2\kappa} \pp^*\pp \psi_2 T\pp^* \psi_1 (
  \!\begin{aligned}[t]
&  R_{1}(x, \Wf_1^{\kappa} (V_j+H)(x)) +\nonumber\\
& R_{2}(x, \Wf_1^{\kappa}(V_j+H)(x))- g(x)) \|_{L^{p}_{\sv-\kappa\ef_1}(\delta^{\lambda\dv} W, \dv)}
  \end{aligned}
\\  & \leq \| \phi_{x_0, \frac{\xi_2\xi_4}{\xi_3}\delta}\delta^{2\kappa} \pp^* \psi_1 (
  \!\begin{aligned}[t]
&  R_{1}(x, \Wf_1^{\kappa} (V_j+H)(x)) +\nonumber\\
& R_{2}(x, \Wf_1^{\kappa}(V_j+H)(x))- g(x)) \|_{L^{p}_{\sv-\kappa\ef_1}(\delta^{\lambda\dv} W, \dv)}
  \end{aligned}
  \nonumber\\&+ \| \phi_{x_0, \frac{\xi_2\xi_4}{\xi_3}\delta}\delta^{2\kappa} E_{\infty} \psi_1 (
  \!\begin{aligned}[t]
&  R_{1}(x, \Wf_1^{\kappa} (V_j+H)(x)) +\nonumber\\
& R_{2}(x, \Wf_1^{\kappa}(V_j+H)(x))- g(x)) \|_{L^{p}_{\sv-\kappa\ef_1}(\delta^{\lambda\dv} W, \dv)}
  \end{aligned}
  \\&=: (III)+ (IV)
\end{align}
To bound $(IV)$ we use Corollary \ref{scaling: Scaled Estimates: Corollary of Proposition 9.2.10 of [BS]} to see that, for some $L\approx 1$
\begin{align}\label{Reduction II: estimate for (IV) using trinagle inequality}
    (IV)&\lesssim \delta^{2\kappa} \|E_{\infty}\|_{C^L}
 \|R_{1}(x, \Wf_1^{\kappa} (V_j+H)(x)) +
\|R_{2}(x, \Wf_1^{\kappa}(V_j+H)(x))- g(x)) \|_{L^{\infty}}
 \nonumber \\&\lesssim \delta^{2\kappa} (\|R_1(x, \Wf_1^{\kappa}(V_j+H)(x))\|_{L^{\infty}}+ \|R_{2}(x, \Wf_1^{\kappa}(V_j+H)(x))\|_{L^{\infty}}+\|g\|_{L^{\infty}}). 
\end{align}
Using, Remark \ref{Tame estimate for sobolev space: remark about embedding to space of continuous functions}
\begin{align}\label{Reduction II: R_1 L^{infty} estimate}
   \|R_1(x, \Wf_1^{\kappa}(V_j+H)(x))\|_{L^{\infty}} \leq \|R_1\|_{L^{\infty}} \leq \|R_1\|_{\cc^{\sv, M}((W, \dv)\boxtimes \nabla_{\RR^N})}\leq \epsilon_1.
\end{align}
Since $R_2(x, 0)\equiv 0$, using Remark \ref{Tame estimate for sobolev space: remark about embedding to space of continuous functions} and Lemma \ref{Preliminaries: Functions spaces: Zygmund-Holder space: Prop 7.5.11} we have
\begin{align}\label{Reduction II: R_2 L^{infty} estimate}
    \|R_{2}(x, \Wf_1^{\kappa}(V_j+H)(x))\|_{L^{\infty}} &\lesssim \sum_{|\beta|\leq 1} \|\partial_{\zeta}^{\beta} R_2(x, \zeta)\|_{L^{\infty}}\|\Wf_1^{\kappa} (V_j+H)\|_{L^{\infty}}\nonumber\\
    &\lesssim \sum_{|\beta|\leq 1} \|\partial_{\zeta}^{\beta} R_2\|_{\cc^{\sv, M}((W,\dv)\boxtimes \nabla_{\RR^N})}\|\Wf_1^{\kappa} (V_j+H)\|_{L^{\infty}}\nonumber\\
    &\lesssim \|R_2\|_{\cc^{\sv, M+1}((W, \dv)\boxtimes \nabla_{\RR^N})}\|\Wf_1^{\kappa}(V_j+H) \|_{L^{\infty}}\nonumber\\
    &\lesssim \|\Wf_1^{\kappa} (V_j+H)\|_{L^{\infty}}\nonumber\\
    &\lesssim \epsilon_1,
\end{align}
where the final estimate used \eqref{Reduction II: Wf_1^{kappa}(V_j+H) estimate}
We also have, 
\begin{align}\label{Reduction II: g L^{infty} estimate}
    \|g\|_{L^{\infty}}\leq \|g\|_{\cc^{\rv}(W, \dv)}\lesssim \epsilon_1.  
\end{align}
Using \eqref{Reduction II: R_1 L^{infty} estimate}, \eqref{Reduction II: R_2 L^{infty} estimate} and \eqref{Reduction II: g L^{infty} estimate} we can estimate the right hand side of \eqref{Reduction II: estimate for (IV) using trinagle inequality} by 
\begin{align}\label{Reduction II: Final Estimate for IV}
    (IV)\lesssim \delta^{2\kappa} \epsilon_1 \lesssim \epsilon_1.
\end{align}
For $(III)$, we use the fact that $\phi_{x_0, \frac{\xi_2\xi_4}{\xi_3}\delta}\prec \phi_{x_0, \frac{\xi_2^2\xi_4}{\xi_3^2}\delta}$ (see Theorem \ref{Scaling: theorem 9.2.8 [BS]}) to get 
\begin{align*}
   (III)=& \|\phi_{x_0, \frac{\xi_2\xi_4}{\xi_3}\delta} \delta^{2\kappa}\pp \psi_1 (
  \!\begin{aligned}[t]
&  R_{1}(x,\phi_{x_0, \frac{\xi_2^2\xi_4}{\xi_3^2}\delta} \Wf_1^{\kappa} (V_j+H)(x)) +\\
& R_{2}(x, \phi_{x_0, \frac{\xi_2^2\xi_4}{\xi_3^2}\delta}\Wf_1^{\kappa}(V_j+H)(x))- g(x)) \|_{L^{p}_{\sv-\kappa\ef_1}(\delta^{\lambda\dv} W, \dv)}
  \end{aligned}
  \\\lesssim &\delta^{\kappa} \|\phi_{x_0, \frac{\xi_2\xi_4}{\xi_3}\delta}  R_{1}(x,\phi_{x_0, \frac{\xi_2^2\xi_4}{\xi_3^2}\delta} \Wf_1^{\kappa} (V_j+H)(x))\|_{L^{p}_{\sv}(\delta^{\lambda\dv} W, \dv)} \\
  & +\delta^{\kappa} \|\phi_{x_0, \frac{\xi_2\xi_4}{\xi_3}\delta}  R_{2}(x,\phi_{x_0, \frac{\xi_2^2\xi_4}{\xi_3^2}\delta} \Wf_1^{\kappa} (V_j+H)(x))\|_{L^{p}_{\sv}(\delta^{\lambda\dv} W, \dv)} \\
  & +\delta^{\kappa} \|\phi_{x_0, \frac{\xi_2\xi_4}{\xi_3}\delta} g\|_{L^{p}_{\sv}(\delta^{\lambda\dv} W, \dv)} \\
  =:& (V) + (VI) +(VII).
\end{align*}
For $(VII)$, by the assumption on $g$ we have
\begin{align}
    (VII)\lesssim \delta^{\kappa} \|g\|_{L^{p}_{\sv}(W, \dv)} \lesssim \delta^{\kappa} \epsilon_1\lesssim \epsilon_1. 
\end{align}
Since $supp \left(\phi_{x_0, \frac{\xi_2^3\xi_4}{\xi_3^3}\delta}\right)\subseteq B_{(W^1, \db^1)}\left(x_0,\frac{\xi_2^3\xi_4}{\xi_3^3}\delta \right)\subseteq B^n(3/4)$ we can use Theorem \ref{Scaling: theorem 9.2.8 [BS]} to get 
\begin{align}\label{Reduction II: equation lpsk estimate for H multiplied with a different bump function}
    \|\phi_{x_0, \frac{\xi_2^3\xi_4}{\xi_3^3}\delta} H\|_{L^{p}_{\sv+\kappa\ef_1}(\delta^{\lambda\dv}W, \dv)}=  \|\phi_{x_0, \frac{\xi_2^3\xi_4}{\xi_3^3}\delta} \psi_3H\|_{L^{p}_{\sv+\kappa\ef_1}(\delta^{\lambda\dv}W, \dv)}\lesssim \|\psi_3H\|_{L^{p}_{\sv+\kappa\ef_1}(\delta^{\lambda\dv}W, \dv)}\lesssim \epsilon_1,
\end{align}
where we used \eqref{Reduction II: psi_3H estimate}. Using Proposition \ref{Scaling: Scaled Estimates: Proposition 9.2.11 of [BS]} and \eqref{Reduction II: equation lpsk estimate for H multiplied with a different bump function} we have
\begin{align}\label{Reduction II: lps estimate for Wf_1^{kappa}(V_j+H) multiplied with a bump function}
    \delta^{\kappa} \|\phi_{x_0, \frac{\xi_2^2\xi_4}{\xi_3^2}\delta}\Wf_1^{\kappa} (V_j+H)\|_{L^{p}_{\sv}(\delta^{\lambda\dv}W, \dv)}&\lesssim \|\phi_{x_0, \frac{\xi_2^3\xi_4}{\xi_3^3}\delta}(V_j+H)\|_{L^{p}_{\sv}(\delta^{\lambda\dv}W, \dv)}\nonumber\\
    &\leq \|\phi_{x_0, \frac{\xi_2^3\xi_4}{\xi_3^3}\delta}V_j\|_{L^{p}_{\sv}(\delta^{\lambda\dv}W, \dv)}+ \|\phi_{x_0, \frac{\xi_2^3\xi_4}{\xi_3^3}\delta}H\|_{L^{p}_{\sv}(\delta^{\lambda\dv}W, \dv)}\nonumber\\
    &\lesssim \|\phi_{x_0, \frac{\xi_2^3\xi_4}{\xi_3^3}\delta}V_j\|_{L^{p}_{\sv}(\delta^{\lambda\dv}W, \dv)}+ \epsilon_1. 
\end{align}
Finally, since $0\leq \phi_{x_0, \frac{\xi_2^3\xi_4}{\xi_3^3}\delta} \leq 1$ (see Theorem \ref{Scaling: theorem 9.2.8 [BS]}) 
\begin{align}\label{Reduction II: L^{infty} estimate for Wf_1^{kappa}(V_j+H) multiplied with a bump function}
    \|\phi_{x_0, \frac{\xi_2^2\xi_4}{\xi_3^2}\delta} \Wf_1^{\kappa} (V_j+H)\|_{L^{\infty}} \leq \|\Wf_1^{\kappa}(V_j+H)\|_{L^{\infty}}\lesssim \epsilon_1. 
\end{align}
For the following estimates of $(V)$ and $(VI)$, we freely use \eqref{Reduction II: equation lpsk estimate for H multiplied with a different bump function}, \eqref{Reduction II: lps estimate for Wf_1^{kappa}(V_j+H) multiplied with a bump function} and \eqref{Reduction II: L^{infty} estimate for Wf_1^{kappa}(V_j+H) multiplied with a bump function}.

For $(V)$ by Proposition \ref{Scaling: Scaled Estimates: Proposition similar to 9.2.14 in [BS]} 
\begin{align}\label{Reduction II: Estimate for (V)}
    (V)=&\delta^{\kappa} \|\phi_{x_0, \frac{\xi_2\xi_4}{\xi_3}\delta}  R_{1}(x,\phi_{x_0, \frac{\xi_2^2\xi_4}{\xi_3^2}\delta} \Wf_1^{\kappa} (V_j+H)(x))\|_{L^{p}_{\sv}(\delta^{\lambda\dv} W, \dv)}\nonumber\\
    \lesssim &\|R_1\|_{\cc^{\sv,M}((W, \dv)\boxtimes \nabla_{\RR^N})} (\delta^{\kappa}+\delta^{\kappa} \|\phi_{x_0, \frac{\xi_2^2\xi_4}{\xi_3^2}\delta}\Wf_1^{\kappa} (V_j+H)(x))\|_{L^{p}_{\sv}(\delta^{\lambda\dv}W, \dv)})\times\nonumber\\
    & (1+ \|\phi_{x_0, \frac{\xi_2^2\xi_4}{\xi_3^2}\delta}\Wf_1^{\kappa}(V_j+H)(x))\|_{L^{\infty}})^{M+\nu-1}\nonumber\\
    \lesssim&\epsilon_1 (1+ \|\phi_{x_0, \frac{\xi_2^3\xi_4}{\xi_3^3}\delta} V_j\|_{L^{p}_{\sv+\kappa\ef_1}(\delta^{\lambda\dv}W, \dv)}+\epsilon_1) (1+\epsilon_1)^{M+\nu-1}\nonumber\\
    \lesssim & \epsilon_1 +\epsilon_1\|\phi_{x_0, \frac{\xi_2^3\xi_4}{\xi_3^3}\delta}V_{j}\|_{L^{p}_{\sv+\kappa\ef_1}(\delta^{\lambda\dv}W, \dv)}.
\end{align}
For $(VI)$ we apply Proposition \ref{Scaling: Scaled Estimates: Proposition similar to 9.2.14 in [BS]} $(ii)$ with $L=2$ to get 
\begin{align}\label{Reduction II: Estimate for (VI)}
    (VI)=& \delta^{\kappa} \|\phi_{x_0, \frac{\xi_2\xi_4}{\xi_3}\delta}  R_{2}(x,\phi_{x_0, \frac{\xi_2^2\xi_4}{\xi_3^2}\delta} \Wf_1^{\kappa} (V_j+H)(x))\|_{L^{p}_{\sv}(\delta^{\lambda\dv} W, \dv)}\nonumber\\
    \lesssim &\delta^{\kappa} \|R_2\|_{\cc^{\sv,M}((W, \dv)\boxtimes \nabla_{\RR^N})}  (1+ \|\phi_{x_0, \frac{\xi_2^2\xi_4}{\xi_3^2}\delta}\Wf_1^{\kappa}(V_j+H)(x))\|_{L^{\infty}})^{M+\nu-1}\nonumber\\&\|\phi_{x_0, \frac{\xi_2^2\xi_4}{\xi_3^2}\delta}\Wf_1^{\kappa}(V_j+H)(x))\|_{L^{\infty}} \Big(
      \!\begin{aligned}[t]
     &\|\phi_{x_0, \frac{\xi_2^2\xi_4}{\xi_3^2}\delta}\Wf_1^{\kappa}(V_j+H)(x))\|_{L^{\infty}}+ \nonumber \\
     &(1+ \|\phi_{x_0, \frac{\xi_2^2\xi_4}{\xi_3^2}\delta}\Wf_1^{\kappa}(V_j+H)(x))\|_{L^{\infty}})\times\nonumber\\
     &\|\phi_{x_0, \frac{\xi_2^2\xi_4}{\xi_3^2}\delta}\Wf_1^{\kappa}(V_j+H)(x))\|_{L^p_{\sv}(\delta^{\lambda\dv}W, \dv)}\Big)
       \end{aligned}
    \nonumber \\\lesssim & (1+\epsilon_1) \epsilon_1
 (\epsilon_1 + (1+\epsilon_1) ( \|\phi_{x_0, \frac{\xi_2^3\xi_4}{\xi_3^3}\delta} V_j\|_{L^p_{\sv+\kappa\ef_1}(\delta^{\lambda\dv}W, \dv)}+\epsilon_1))   \nonumber\\
 \lesssim& \epsilon_1 +\epsilon_1 \|\phi_{x_0, \frac{\xi_2^3\xi_4}{\xi_3^3}\delta} V_j\|_{L^p_{\sv+\kappa\ef_1}(\delta^{\lambda\dv}W, \dv)}.
\end{align}
Using \eqref{Reduction II: Estimate for (V)} and \eqref{Reduction II: Estimate for (VI)} we get 
\begin{align}\label{Reduction II: Final Estimate for III}
    (III)\lesssim \epsilon_1 +\epsilon_1  \|\phi_{x_0, \frac{\xi_2^3\xi_4}{\xi_3^3}\delta} V_j\|_{L^p_{\sv+\kappa\ef_1}(\delta^{\lambda\dv}W, \dv)}.
\end{align}
Using \eqref{Reduction II: Final Estimate for III} and \eqref{Reduction II: Final Estimate for IV} we can bound the right hand side of \eqref{Reduction II: (III)+(IV)} by 
\begin{align}\label{Reduction II: estimate for III + IV}
    \|\phi_{x_0, \frac{\xi_2\xi_4}{\xi_3}\delta}\delta^{2\kappa}\pp^*\pp V_{j+1}\|_{L^{p}_{\sv-\kappa\ef_1}(\delta^{\lambda\dv}W, \dv)} \lesssim \epsilon_1 +\epsilon_1 \|\phi_{x_0, \frac{\xi_2\xi_4}{\xi_3}\delta} V_j\|_{L^p_{\sv+\kappa\ef_1}(\delta^{\lambda\dv}W, \dv)}.
\end{align}
Set 
\begin{align}\label{Reduction II: Definition for Q_j}
    Q_{j}:= \underset{B_{(W^1, \db^1)}(x,\delta)\subseteq B^{n}(3/4)}{\underset{\delta\in (0,1]}{\sup}}\|\phi_{x,\delta} V_j\|_{L^p_{\sv+\kappa\ef_1}(\delta^{\lambda\dv}W, \dv)} \in [0,\infty],
\end{align}
where the supremum is taken over all such $x$ and $\delta$. Since $V_0=0$, we have $Q_0=0$. For $x_0\in B^{n}(3/4)$ and $\delta\in (0,1]$ such that $B_{(W^1, \db^1)}(x_0, \frac{\xi_2^3\xi_4}{\xi_3^3}\delta)\subseteq B^{n}(3/4)$, \eqref{Reduction II: Definition for Q_j} implies that 
\begin{align}\label{Reduction II: Boundign lps norm of V_{j+1} using epsilon_1 and Q_j}
    \|\phi_{x_0, \xi_0\delta}V_{j+1}\|_{L^{p}_{\sv+\kappa\ef_1}(\delta^{\lambda\dv}W, \dv)}\lesssim \epsilon_1+\epsilon_1Q_j.
\end{align}
Take $x\in B^{n}(3/4), \delta\in (0,1]$ such that $B_{(W^1, \db^1)}(x,\delta)\subseteq B^{n}(3/4)$. Take $N\approx 1$ and $x_1,....,x_{N_1}$ as in Theorem \ref{Scaling: theorem 9.2.8 [BS]} (viii). So \eqref{Reduction II: Boundign lps norm of V_{j+1} using epsilon_1 and Q_j} implies that 
\begin{align*}\label{Reduction II:phi_{x_j,xi_6 detla}V_j bounded by Q_j and epsilon_1}
     \|\phi_{x_j, \xi_0\delta}V_{j+1}\|_{L^{p}_{\sv+\kappa\ef_1}(\delta^{\lambda\dv}W, \dv)}\lesssim \epsilon_1+\epsilon_1Q_j, \quad 1\leq j\leq N_1.
\end{align*}
Thus Theorem \ref{Scaling: theorem 9.2.8 [BS]} (viii) and 
\begin{align*}
     \|\phi_{x, \xi_0\delta}V_{j+1}\|_{L^{p}_{\sv+\kappa\ef_1}(\delta^{\lambda\dv}W, \dv)}\lesssim \sum_{j=1}^{N_1}  \|\phi_{x_j, \xi_0\delta}V_{j+1}\|_{L^{p}_{\sv+\kappa\ef_1}(\delta^{\lambda\dv}W, \dv)} \lesssim \epsilon_1+\epsilon_1Q_j.
\end{align*}
Taking supremum over all such $x$ and $\delta$ shows that 
\begin{align*}
    Q_{j+1}\lesssim \epsilon_1+\epsilon_1Q_j.
\end{align*}
i.e, $\exists C_{20}$ with 
\begin{align*}
    Q_{j+1}\leq C_{20}(\epsilon_1, \epsilon_1Q_j).
\end{align*}
We take $\epsilon_1:= \min\{\frac{1}{2}C_{20}, \frac{1}{2}C_{19}, \frac{1}{D^2}\}$.
Then we have, 
\begin{align}
    Q_{j+1}\leq \frac{1}{2}+\frac{Q_j}{2}.
\end{align}
Since $Q_0=0$, induction argument shows that $Q_j\leq 1, \forall j\in \NN$.
 \end{proof}
\section{Finishing proof of the main theorem}
\numberwithin{equation}{section}
 \subsection{Setting it up}\label{Setting it up}
\numberwithin{equation}{section}
 
Take $J\in \NN$ large, to be chosen later. Let
\ba 
u_{J}:= \underset{j_2,....,j_{\nu}\in \NN}{\underset{0\leq j_1<J}{\sum}} D_{j}u, \ v_{J}:= 2^{\kappa J} \underset{j_2,....,j_{\nu}\in \NN}{\underset{j_1\geq J}{\sum}} D_{j}u,\\
g_{J}:= \underset{j_2,....,j_{\nu}\in \NN}{\underset{0\leq j_1<J}{\sum}} D_{j}g, \ h_{J}:= 2^{\kappa J} \underset{j_2,....,j_{\nu}\in \NN}{\underset{j_1\geq J}{\sum}} D_{j}g,
\ea
Note that 
\ba 
u_{J} + 2^{-\kappa J} v_{J} = \sum_{j\in \NN^{\nu}} D_{j}u =u
\ea
and 
\begin{align}\label{Setting it up: Cutting of higher frequency of u}
    g_{J} + 2^{-\kappa J} h_{J} = \sum_{j\in \NN^{\nu}} D_{j}g =g=g_{J} + 2^{-\kappa J} h_{J}
\end{align}
Set
\begin{align}
G_{J}(\zeta)&:= F(0, \{(W^1)^{\alpha} u_{J}(0)\}_{\deg_{\db^1}(\alpha)\leq \kappa}+ \{2^{-(\kappa- \deg_{\db^1}(\alpha)J)}\zeta_{\alpha}\}_{\deg_{\db}(\alpha)\leq \kappa}),\label{definition of G_J}\\
H_{J}(x,\zeta)&:= F(x, \{(W^1)^{\alpha} u_{J}(0)\}_{\deg_{\db^1}(\alpha)\leq \kappa}+ \{2^{-(\kappa- \deg_{\db^1}(\alpha)J)}\zeta_{\alpha}\}_{\deg_{\db}(\alpha)\leq \kappa})\label{definition of H_J}\\
&- F(0, \{(W^1)^{\alpha} u_{J}(0)\}_{\deg_{\db^1}(\alpha)\leq \kappa}+ \{2^{-(\kappa- \deg_{\db^1}(\alpha)J)}\zeta_{\alpha}\}_{\deg_{\db}(\alpha)\leq \kappa}).\nonumber
\end{align}
Hence,
\ba
G_{J}&(\{(2^{-J\db^1}W^1)^{\alpha}v_{J}(x)\}_{\deg_{\db^1}(\alpha)\leq \kappa}) +H_{J}(x,\{(2^{-J\db^1}W^1)^{\alpha}v_{J}(x)\}_{\deg_{\db^1}(\alpha)\leq \kappa} )\\&= F(x,\{(W^1)^{\alpha}u(x)\}_{\deg_{\db^1}(\alpha)\leq \kappa})=g(x)=g_{J} + 2^{-\kappa J} h_{J}. 
\ea
In the coming preliminary results, let $\vec{r},\vec{s},\vec{t}\in (0,\infty)^{\nu}$ and $\kappa \in \NN_{+}$. 
\begin{lemma}\label{Setting it up: Lemma similar to Lemma 9.2.20 in [BS]}
    Fix $\psi\in C_{0}^{\infty}(B^{n}(7/8))$ and $u\in \cc^{\rv+\kappa \ef_1}(\overline{B^{n}(7/8)}, (W,\dv))$. For $J\in [0,\infty)$ and $\deg_{\db^1}(\alpha)\leq \kappa$, set 
    \ba 
    \widebar{u}_{J,\alpha}(y)&:= \underset{j_2,...,j_{\nu}\in \NN}{\underset{0\leq j_1<J}{\sum}} ((W^1)^{\alpha}D_{j}u) (\Phi_{0,2^{-J}}(y))= ((W^1)^{\alpha}u_{J}(y))(\Phi_{0,2^{-J}}(y)),\\
    \tilde{u}_{J,\alpha}(y)&:= \underset{j_2,...,j_{\nu}\in \NN}{\underset{0\leq j_1<J}{\sum}} [((W^1)^{\alpha}D_{j}u) (\Phi_{0,2^{-J}}(y))-((W^1)^{\alpha}D_{j}u) (0)]\\
    &= \widebar{u}_{J,\alpha}(y)-\widebar{u}_{J,\alpha}(0),\\
    \widebar{v}_{J}(y)&:= \underset{j_2,...,j_{\nu}\in \NN}{\underset{0\leq j_1<J}{\sum}} 2^{J\kappa} D_{j}u(\Phi_{0,2^{-J}}(y))= v_{J}(\Phi_{0,2^{-J}}(y))\\
     \widebar{h}_{J}(y)&:= \underset{j_2,...,j_{\nu}\in \NN}{\underset{0\leq j_1<J}{\sum}} 2^{J\kappa} D_{j}g(\Phi_{0,2^{-J}}(y))= h_{J}(\Phi_{0,2^{-J}}(y)). 
    \ea
    Then 
    \begin{enumerate}[label=(\roman*)]
        \item $ \|\psi \widebar{u}_{J,\alpha}\|_{\cc^{\vec{t}}(W^{0,2^{-J}},\dv)}\leq  C_1 \|u\|_{\cc^{\rv+\kappa \ef_1}(W,\dv)}$
        \item $\|\psi \widebar{u}_{J,\alpha}\|_{L^{p}_{\sv}(W^{0,2^{-J}},\dv) }\leq C_{2} \|u\|_{\cc^{\rv+\kappa \ef_1}(W,\dv)}$
        \item $\|\psi \tilde{u}_{J,\alpha}\|_{\cc^{\vec{t}}(W^{0,2^{-J}},\dv)}\leq C_3 2^{-j(\rv_1\wedge 1)}\|u\|_{\cc^{\rv+\kappa \ef_1}(W,\dv)}$
        \item $\|\psi \tilde{u}_{J,\alpha}\|_{L^p_{\sv}(W^{0,2^{-J}},\dv)}\leq C_4 2^{-j(\rv_1\wedge 1)}\|u\|_{\cc^{\rv+\kappa \ef_1}(W,\dv)}$
        \item $     \|\psi \widebar{v}_J\|_{\cc^{\rv+\kappa\ef_1}(W^{0,2^{-J}},\dv)}\leq C_5 2^{-J\rv_1} \|u\|_{\cc^{\rv+\kappa\ef_1}(W,\dv)}$
        \item $  
       \|\psi \widebar{g}_{J,
    \alpha}\|_{\cc^{\vec{t}}(W^{0,2^{-J}},\dv)}\leq C_6 
       \|g\|_{\cc^{\rv}(W,\dv)}$
       \item $\|\psi \widebar{h}_{J}\|_{\cc^{\rv}(W^{0,2^{-J}},\dv)} \leq C_{7}2^{-J\rv_1}\|g\|_{\cc^{\rv}(W,\dv)}$
       \item $\|\psi \widebar{h}_J\|_{L^{p}_{\sv}(W^{0,2^{-J}},\dv)} \leq C_8 2^{-J \sv_1}\|g\|_{L^p_{\sv}(W,\dv)}$
    \end{enumerate}
    Here, $C_1=C_1(\psi),C_3=C_3(\psi),C_6=C_6(\psi)\geq 0$ are $(\rv, \vec{t}, \kappa)-$multi-parameter unit-admissible constants$C_2=C_2(\psi),C_4=C_4(\psi),C_8=C_8(\psi)\geq 0$ are $(p,\rv, \sv, \kappa)-$multi-parameter unit-admissible constants and $C_5=C_5(\psi),C_7=C_7(\psi)\geq 0$ are $(p,\rv, \kappa)$-multi-parameter unit-admissible constants. 
    \begin{remark}
        \cite{BS} does provide a proof of (i),(iii) and (v), but we give the proof here for the sake of completeness. 
    \end{remark}
      \begin{proof}
          It is easy to see that (i)$\implies$ (ii) and (iii)$\implies$ (iv) using Lemma \ref{Preliminaries: embedding sobolev space in Zygmund Holder space}. The proof of (vii) follows from the similar steps as in (v), and (vi) follows similar steps as in (i). So, we will only give a proof of (i), (iii), (v) and (viii), and the rest will be rest to the reader to verify. \\ 

          For any ordered multi-index $\alpha, \beta$, it follows from Theorem \ref{Preliminaries: Function spaces: properties of elementary operators} part (c) that
          \\  $\left(\{(2^{-jd^1}X^1)^{\beta}(2^{-j\db^1}W^1)^{\alpha} D_j,2^{-j}\right): j\in \NN^{\nu}\}$ is a bounded set of generalized $(W,\dv)$ elementary operators, and therefore the definition of Zygmund-Holder space $\cc^{\rv+\kappa \ef_1}$ shows that 
          \begin{align}\label{1.1}
              \|(2^{-j d^1} X^1)^{\beta}(2^{-j\db^1}W^1)^{\alpha} D_j u\|_{L^{\infty}} \lesssim_{\alpha,\beta} 2^{-j\cdot (\rv+\kappa \ef_1)} \|u\|_{\cc^{\rv+\kappa \ef_1}(W,\dv)}.
          \end{align}
          Letting $\deg_{\db}(\alpha)\leq \kappa$ and letting $\beta$ be any ordered-index, and letting $j\in \NN^{\nu}$ such that $j_1<J$, we have using \eqref{1.1} that 
          \begin{align}\label{1.2}
              \|(2^{-jd^1}X^1)^{\beta} (W^1)^{\alpha} D_j u \|_{L^{\infty}} &= 2^{(j_1-J)\deg_{d^1}(\beta)} 2^{j_1 \deg_{\db^1}(\alpha)} \|(2^{-jd^1}X^1)^{\beta}(2^{-j_1\db^1}W^1)^{\alpha} D_j u\|_{L^{\infty}}\nonumber\\
              &\lesssim 2^{(j_1-J)\deg_{d^1}(\beta)} 2^{j_1 \deg_{\db^1}(\alpha)} 2^{-j\cdot(\rv+\kappa \ef_1)}\|u\|_{\cc^{\rv+\kappa \ef_1}(W,\dv)}\nonumber\\
              &\lesssim 2^{(j_1-J)\deg_{d^1}(\beta)} 2^{-j\cdot \rv}\|u\|_{\cc^{\rv+\kappa \ef_1}(W,\dv)}
          \end{align}
          Thus,
          \begin{align}\label{1.3}
               \|(2^{-jd^1}X^1)^{\beta} (W^1)^{\alpha} u_J \|_{L^{\infty}} &=  \|(2^{-jd^1}X^1)^{\beta} (W^1)^{\alpha} \underset{\underset{j_2,..., j_{\nu}\in \NN}{0\leq j_1<J}}{\sum}D_j u \|_{L^{\infty}}  \nonumber \\
               &\lesssim \underset{\underset{j_2,..., j_{\nu}\in \NN}{0\leq j_1<J}}{\sum} 2^{(j_1-J)\deg_{d^1}(\beta)} 2^{-j\cdot \rv}\|u\|_{\cc^{\rv+\kappa \ef_1}(W,\dv)} \nonumber \\
               &\lesssim \begin{cases}
                   \|u\|_{\cc^{\rv+\kappa \ef_1}(W,\dv)} \quad \quad \quad |\beta|=0\\
                   2^{-J(\rv_1\wedge 1)} \|u\|_{\cc^{\rv+\kappa \ef_1}(W,\dv)} \quad |\beta|>0 \\
               \end{cases} 
          \end{align}
          In particular, using \eqref{1.3} combined with Remark \ref{lemma 9.2.19 and 9.2.18} shows that (i) holds. \\

          Let $x\in B_{W^1,\db^1}(0,2^{-J})$. Thus, there exists continuous curve $\gamma:[0,1]\rightarrow B_{(X^1,d^1)(0,2^{-J})}$ such that $\gamma(0)=0, \gamma(1)=x$, and $\gamma'(t)= \sum_{l=1}^{q_1} a_{l}(t) 2^{-Jd_{l}^1} X_{l}^1(\gamma(t))$, with $\|\sum_{l}|a_{l}|^2\|_{L^{\infty}([0,1])}\leq 1$. 
          
          Hence, we have 
          \begin{align*}
              |(W^1)^{\alpha} D_{j}u(x)- (W^1)^{\alpha}D_{j}u(0)| &=  |(W^1)^{\alpha} D_{j}u(\gamma(1))- (W^1)^{\alpha}D_{j}u(\gamma(0))|\\
              &\leq \int_{0}^1 |\frac{d}{dt} (W^1)^{\alpha} D_{j}u(\gamma(t)|\ dt\\
              &\lesssim \sum_{l=1}^{q_1} \int_{0}^{1} |2^{-Jd^1}X_{l}^1 (W^1)^{\alpha} D_{j}u(\gamma(t))| \ dt \\ 
              &\lesssim \sum_{l=1}^{q_1} \|2^{-Jd^1}X_{l}^1 (W^1)^{\alpha} D_{j}u\|_{L^{\infty}}\\ 
              &\lesssim 2^{j_1-J} 2^{-j\cdot \rv} \|u\|_{\cc^{\rv+\kappa \ef_1} (W,\dv)},
          \end{align*}
          where the final inequality uses \eqref{1.2}. Thus, 
          \begin{align}\label{1.4}
             & \|\underset{\underset{j_2,...,j_{\nu}\in \NN}{j_1<J}}{\sum} [(W^1)^{\alpha}D_{j}u(0)- (w^1)^{\alpha}D_{j}u(0)]\|_{L^{\infty}(B_{X^{1},d^1}(0,2^{-J}))}\nonumber\\
              &\lesssim \|\underset{\underset{j_2,...,j_{\nu}\in \NN}{j_1<J}}{\sum} 2^{j_1-J}2^{-j\cdot \rv}\|u\|_{\cc^{\rv+\kappa\ef_1}(W,\db^1)}\nonumber \\
              &\lesssim 2^{-J(\rv_1 \wedge 1)} \|u\|_{\cc^{\rv+\kappa \ef_1}(W,\db^1)}
          \end{align}
          Combining \eqref{1.3} and \eqref{1.4} we will get the required estimate (iv), which uses Remark \ref{lemma 9.2.19 and 9.2.18} once again. \\

          We now proof (v) We have 
          \ba
          \widebar{v}_J= \sum_{j\in \NN^{\nu}} 2^{J\kappa} D_{j+J\ef_1} u(\Phi_{0,2^{-J}} (y)) .
          \ea
          Thus by Remark \ref{lemma 9.2.19 and 9.2.18} it is sufficient to show that, $\forall L \in \NN$
          \begin{align*}
              \sum_{|\beta|\leq L} \|(2^{-(j+J\lambda)\dv}W)^{\beta} 2^{J\kappa}D_{j+J\ef_1} u\|_{L^{\infty}} \lesssim 2^{-j\cdot(\rv+\kappa \ef_1)} 2^{-j\rv_1} \|u\|_{\cc^{\rv+\kappa \ef_1}(W,\dv)}.
          \end{align*}
          Using Remark \ref{Function spaces: remark about explicit norm: Equivalence of norms} we get
        \ba
         \sum_{|\beta|\leq L} \|(2^{-(j+J\lambda)\dv} W)^{\beta}2^{J\kappa}D_{j+J\ef_1} u\|_{L^{\infty}}&=  \sum_{|\beta|\leq L} 2^{-J(0,\lambda_2,...,\lambda_{\nu})}2^{J\kappa}\|(2^{-(j+J\ef_1)\dv} W)^{\beta}D_{j+J\ef_1} \|u\|_{L^{\infty}}\\
         &\lesssim 2^{-(j+J\ef_1)\cdot (\rv+\kappa\ef_1)}2^{J\kappa} \|u\|_{\cc^{\rv+\kappa\ef_1}(W,\dv)}\\
         &\lesssim 2^{-j\cdot (\rv+\kappa\ef_1)}2^{J\rv_1} \|u\|_{\cc^{\rv+\kappa\ef_1}(W,\dv)},
    \ea 
which completes the proof of (v).  

Next, we will prove (viii).  Then, using Lemma \ref{Scaling: Scaled Estimates: Lemma 9.2.12 in [BS]}
\begin{align}\label{shrinking sobolev norm first step}
    \|\psi \widebar{h}_J\|_{L^p_{\sv+\kappa\ef_1}(W^{0,2^{-J}},\dv)}&\lesssim  \left\|\left( \sum_{k\in \NN^{\nu}}\left|2^{k\cdot (\sv+\kappa\ef_1)}D_{k}( h_J)\right|^2\right)^{1/2}\right\|_{L^{p}(B^{n}(7/8))}\nonumber\\
      &\lesssim \left\|\left(\sum_{k\in \NN^{\nu}}\left|2^{k\cdot (\sv+\kappa\ef_1)}D_{k}\left(\underset{j_{2,...,j_{\nu}\in \NN}}{\underset{j_1\geq J}{\sum}}2^{\kappa J}D_{j}g\right)\right|^2\right)^{1/2}\right\|_{L^{p}(B^{n}(7/8))}\nonumber\\
           &\lesssim \left\|\left( \sum_{k\in \NN^{\nu}}\left|2^{k\cdot (\sv+\kappa\ef_1)}D_{k}\left(\sum_{j\in \NN^{\nu}}2^{\kappa J}D_{j+J\ef_1}g\right)\right|^2\right)^{1/2}\right\|_{L^{p}(B^{n}(7/8))}.
\end{align}
Fix $M\in \NN$ such that $(|\sv|_{\infty}+\kappa +1)\nu\leq M$. By Theorem \ref{Preliminaries: Function spaces: properties of elementary operators} part(d) we can write $D_{k}$ as 
\begin{align}\label{shrinking sobolev norm second step}
    D_{k}=\sum_{|\alpha_{\mu}|\leq M} 2^{(|\alpha_{\mu} -M|k_{\mu})} \left(2^{-k_{\mu}\db^{\mu}} W^{\mu}\right)^{\alpha_{\mu}} E_{k,\mu,\alpha_{\mu}},
\end{align}
where $\left\{(E_{k,\mu,\alpha_{\mu}},2^{-k}):\mu\in \{1,...,\nu\}, |\alpha_{\mu}|\leq M\right\}$ is a bounded set of $(W,\dv)$ generalized elementary operators. Using \eqref{shrinking sobolev norm first step} and \eqref{shrinking sobolev norm second step} we get 
\begin{align*}
    \Bigg\| \Bigg(\sum_{k\in \NN^{\nu}} \Big|\sum_{|k\vee j -j|_{\infty}=0} 2^{\kappa J}2^{k\cdot(\sv+\kappa \ef_1)} &D_{k}(D_{j+J\ef_1}g) \\ \sum_{|k\vee j-j|_{\infty}>0}\sum_{|\alpha_{\mu}|\leq M} &2^{\kappa J} 2^{k\cdot(\sv+\kappa \ef_1)} 2^{(|\alpha_{\mu}-M|k_{\mu} +(j_{\mu}-k_{\mu})\deg_{\db^{\mu}}(\alpha_{\mu})) }\times \\&
    \left(2^{-j_{\mu}\db^{\mu}} W^{\mu}\right)^{\alpha_{\mu}} E_{k,\mu,\alpha_{\mu}} \left(D_{j+J\ef_1} g\right)\Big|^{2}\Bigg)^{1/2}\Bigg\|_{L^{p}(B^{n}(7/8
    ))}
\end{align*}
Now, we use $\deg_{\db^{\mu}}(\alpha_{\mu})\geq |\alpha_{\mu}|$ and $|\alpha_{\mu}|\leq M$ to get the following estimate for the above quantity 
\begin{align*}
    \lesssim  \Bigg\| \Bigg(\sum_{k\in \NN^{\nu}} \Big|\sum_{|k\vee j -j|_{\infty}=0} 2^{\kappa J}2^{k\cdot(\sv+\kappa \ef_1)} &D_{k}(D_{j+J\ef_1}g) \\ \sum_{|k\vee j-j|_{\infty}>0}\sum_{|\alpha_{\mu}|\leq M} &2^{\kappa J} 2^{k\cdot(\sv+\kappa \ef_1)} 2^{-M(k_{\mu}-j_{\mu}) }\times \\&
    \left(2^{-j_{\mu}\db^{\mu}} W^{\mu}\right)^{\alpha_{\mu}} E_{k,\mu,\alpha_{\mu}} \left(D_{j+J\ef_1} g\right)\Big|^{2}\Bigg)^{1/2}\Bigg\|_{L^{p}(B^{n}(7/8
    ))}
\end{align*}
We now make a change of variable $j\mapsto j+k$ and use the convention that $D_{j}=D_0$ for if $\exists \theta\in \{1,..,\nu\}$ such that $j_{\theta}< 0$ to get the RHS in the above quantity to be 
\begin{align}\label{shrinking sobolev norm step three}
    \Bigg\| \Bigg(\sum_{k\in \NN^{\nu}} \Big|\sum_{|k\vee (j+k) -(j+k)|_{\infty}=0} 2^{\kappa J}2^{k\cdot(\sv+\kappa \ef_1)} &D_{k}(D_{j+k+J\ef_1}g) \nonumber\\ \sum_{|k\vee (j+k)-(j+k)|_{\infty}>0}\sum_{|\alpha_{\mu}|\leq M} &2^{\kappa J} 2^{k\cdot(\sv+\kappa \ef_1)} 2^{-M(k_{\mu}-(j+k)_{\mu}) } \times\nonumber\\&
    \left(2^{-(j+k)_{\mu}\db^{\mu}} W^{\mu}\right)^{\alpha_{\mu}} E_{k,\mu,\alpha_{\mu}} \left(D_{j+k+J\ef_1} g\right)\Big|^{2}\Bigg)^{1/2}\Bigg\|_{L^{p}(B^{n}(7/8
    ))}
\end{align}
For the ease of notation, denote $\sv+\kappa \ef_1$ by $\vec{t}$. Note that by the assumption on $M$, $$(|\vec{t}|_{\infty}+1)\nu\leq (|\sv|_{\infty}+\kappa +1)\nu\leq M.$$
Using the above inequality we get 
\begin{align}\label{shrinking sobolev norm step four}
    2^{-j\cdot \vec{t}} 2^{-M|k\vee (j+k) -(j+k)|_{\infty}} &\leq 2^{-j\cdot \vec{t}} 2^{-(|\vec{t}|_{\infty}+1)|k\vee(j+k)-(j+k)|_1}\nonumber\\&\leq \prod_{\{\theta: j_{\theta} <0\}} 2^{-(|\vec{t}|_{\infty}+1)(-j_{\theta})}2^{-j_{\theta}} \prod_{\{\theta: j_{\theta}\geq 0\}}2^{-j_{\theta}t_{\theta}} \nonumber \\
    &\leq \prod_{\theta\in \{1,..,\nu\}}2^{-|j_{\theta}| (1\wedge t_{\theta})}.
\end{align}
Using \eqref{shrinking sobolev norm step four} and triangle inequality in \eqref{shrinking sobolev norm step three} we get 
\begin{align*}
    \lesssim 2^{-Js_1} \Bigg[  \sum_{j\geq 0} 2^{-j\cdot (\sv+\kappa\ef_1)} \Bigg\| \Bigg( \Big|\sum_{k\in \NN^{\nu}}&2^{(k+j+J\ef_1)\cdot(\sv+\kappa \ef_1)} D_{k}(D_{j+k+J\ef_1}g)\Big|^{2}\Bigg)^{1/2}\Bigg\|_{L^{p}(B^{n}(7/8
    ))} +\nonumber\\  \sum_{j\in \ZZ^{\nu}}\sum_{|\alpha_{\mu}|\leq M}\prod_{\theta\in \{1,..,\nu\}}2^{-|j_{\theta}| (1\wedge t_{\theta})} \Bigg\| \Bigg( \sum_{k\in \NN^{\nu}}&2^{(k+j+J\ef_1)\cdot(\sv+\kappa \ef_1)}  \times\nonumber\\&
    \left(2^{-(j+k)_{\mu}\db^{\mu}} W^{\mu}\right)^{\alpha_{\mu}} E_{k,\mu,\alpha_{\mu}} \left(D_{j+k+J\ef_1} g\right)\Big|^{2}\Bigg)^{1/2}\Bigg\|_{L^{p}(B^{n}(7/8
    ))}\Bigg]
\end{align*}
Now, Lemma 6.4.9 of \cite{BS} gives the following bound for the above quantity 
\begin{align*}
       \lesssim 2^{-Js_1} \Bigg[  \sum_{j\geq 0} 2^{-j\cdot (\sv+\kappa\ef_1)} \Bigg\| \Bigg( \Big|\sum_{k\in \NN^{\nu}}&2^{(k+j+J\ef_1)\cdot(\sv+\kappa \ef_1)} D_{j+k+J\ef_1}g\Big|^{2}\Bigg)^{1/2}\Bigg\|_{L^{p}(B^{n}(7/8
    ))} +\nonumber\\  \sum_{j\in \ZZ^{\nu}}\sum_{|\alpha_{\mu}|\leq M}\prod_{\theta\in \{1,..,\nu\}}2^{-|j_{\theta}| (1\wedge t_{\theta})} \Bigg\| \Bigg( \sum_{k\in \NN^{\nu}}&2^{(k+j+J\ef_1)\cdot(\sv+\kappa \ef_1)} D_{j+k+J\ef_1} g\Big|^{2}\Bigg)^{1/2}\Bigg\|_{L^{p}(B^{n}(7/8
    ))}\Bigg].
\end{align*}
    Now, using the fact that 
    \begin{align*}
        \sum_{j\geq 0} 2^{-j\cdot \sv}, \ \sum_{j\in \ZZ^{\nu}}\sum_{\alpha_{\mu}\leq M} \prod_{\theta\in \{1,...,\nu\}}2^{-|j_{\theta}|(1\wedge t_{\theta})} <\infty
    \end{align*}
    we get the required bound 
    \begin{align*}
        \lesssim 2^{-Js_1}\|g\|_{L^{p}_{\sv+\kappa\ef_1}(W,\dv)}. 
    \end{align*}
    \end{proof}
\end{lemma}

The next Proposition shows that appropriate Zygmund-H\"older norm of $H_J$ defined in \eqref{definition of H_J} is small if $J$ is large. We will only state the theorem and refer the reader to \cite{BS}.
\begin{prop}(\cite[Proposition 9.2.21]{BS})\label{Setting it up: Proposition 9.2.21 of [BS]}
    Let $\sv, \rv \in (0,\infty)^{\nu}$ and $M\in \NN$ such that $0<\lfloor|\sv|_1\rfloor+2+\nu< M$, $u\in \cc^{\rv+\kappa \ef_1} (\widebar{B^{n}(7/8)}, (W,\dv);\RR^{D_2})$ and $F\in \cc^{\sv,M} (\widebar{B^{n}(7/8)}\times\RR^N, (W,\dv)\boxtimes \nabla_{\RR^N})$. For $J\in \NN, \psi \in C_{0}^{\infty}(B^{n}(7/8)), c_{\alpha}\in [0,1]$ for $\deg_{\db^1}(\alpha) \leq \kappa$, set 
    \begin{align}\label{definition of H_{j}^c}
        \hat{H}_J^{c} (y,\zeta):= \psi(y) \Bigg[ &F\Big(\Phi_{0,2^{-J}} (y), \underset{j_2,...,j_{\nu}\in \NN}{\underset{j_1<J}{\sum}}\Wf_{1}^{\kappa} D_{j} u (\Phi_{0,2^{-J}} (y)) +\{c_{\alpha} \zeta_{\alpha}\}_{\deg_{\db^1}(\alpha)\leq \kappa} \Big)- \nonumber \\ &F\Big(0, \underset{j_2,...,j_{\nu}\in \NN}{\underset{j_1<J}{\sum}}\Wf_{1}^{\kappa} D_{j} u (0) +\{c_{\alpha} \zeta_{\alpha}\}_{\deg_{\db^1}(\alpha)\leq \kappa} \Big)\Bigg].
    \end{align}
    Then $\hat{H}_{j}^c \in \cc^{\sv, M- \lfloor |\sv|_{1}\rfloor -2- \nu} (\widebar{B^{n}(7/8)} \times \RR^{N}, (W^{0,2^{-J}}, \dv) \boxtimes \nabla_{\RR^N})$ and 
    \begin{align*}
         \left\|\hat{H}_{J}^{c}\right\|_{\cc^{\sv, M- \lfloor |\sv|_{1}\rfloor -2- \nu} (\widebar{B^{n}(7/8)} \times \RR^{N}, (W^{0,2^{-J}}, \dv) \boxtimes \nabla_{\RR^N})} \leq C &\|F\|_{\cc^{\sv, M}((W,\dv) \boxtimes \nabla_{\RR^N})} (1+\|u\|_{\cc^{\rv}+\kappa \ef_1} (W,\dv)) \\&\times (2^{-J(\rv_1 \wedge 1)} +2^{-J (\min\lambda_{\mu}(\sv_{\mu}\wedge 1))}).
    \end{align*}
    Here $C=C(\psi,D_2)\geq 0$ is an $(p,\rv,\sv,M)-$ multi-parameter unit admissible constant. 
\begin{remark}
 Crucial to the proof of above proposition is Theorem \ref{Appendix B: Theorem 7.5.2 of BS}. 
\end{remark}
\end{prop}
\subsection{Putting it together}
In this section we will complete the proof of Theorem \ref{Introduction: Main Theorem} using the set up from. Since $\delta, K$ in the theorem are allowed to depend on $C_0$, WLOG we assume that 
\begin{align*}
    \|F\|_{L^{p}_{\sv,M}((W,\dv)\boxtimes \nabla_{\RR^N})}, \|u\|_{\cc^{\rv+\kappa\ef_1}(W,\dv)}, \|g\|_{L^{p}_{\sv}(W,\dv)}\lesssim 1. 
\end{align*}
Fix, $\psi\in C_{0}^{\infty}(B^{n}(7/8))$ with $\psi\equiv 1$ on a neighbourhood of $B^{n}(13/16)$. We consider only $J\in \NN_+$ large such that $B_{(X^1,d^1)}(0,2^{-J})\subseteq B^{n}(13/16)$; by the Picard-Lindel\"of theorem, how large $J$ must be for $B_{(X^1,d^1)}(0,2^{-J})\subseteq B^{n}(13/14)$ can be chosen to depend only on an upper bound for $\sum_{l=1}^{q}\|X\|_{C^1}$ and therefore may be bounded by a $0-$multi-parameter unit-admissible constant. 

Set
\begin{align*}
    \hat{u}_{J}(y)&:= \psi(y) u_{J}(\phi_{0,2^{-J}}(y))\\
     \hat{v}_{J}(y)&:= \psi(y) v_{J}(\phi_{0,2^{-J}}(y))\\
      \hat{g}_{J}(y)&:= \psi(y) g_{J}(\phi_{0,2^{-J}}(y))\\
       \hat{h}_{J}(y)&:= \psi(y) h_{J}(\phi_{0,2^{-J}}(y))\\
      \hat{H}_{J}(y)&:= \psi(y) H_{J}(\phi_{0,2^{-J}}(y)).\\  
\end{align*}
Since $\psi\equiv 1$ on $B^{n}(13/16)$ and $\Phi_{0,2^{-J}}(B^{n}(1)) \subseteq B_{(X^{1},d^1)}(0,2^{-J}) \subseteq B^{n}(13/16)$ (see Theorem \ref{Preliminaries: Theorem 3.15.5 of [BS]} we have from the setup in the Section \ref{Setting it up} that 
\begin{align}\label{Finishing the proof of the main theorem: scaled PDE equation}
    G_{J}(\{(W^{1,0,2^{-J}})^{\alpha} \hat{v}_{J}\})+ \hat{H}_{J} (Y, \{(W^{1,0,2^{-J}})^{\alpha} \hat{v}_J(y)\})= \hat{h}_J(y)\ \forall y\in B^{n}(13/16). 
\end{align}
We will show that for $J$ sufficiently large, \eqref{Finishing the proof of the main theorem: scaled PDE equation} satisfies the setup in Proposition \ref{Reduction I: main prposition of second reduction}.

Using Lemma \ref{Preliminaries: Zygmund H\"older space: Corollary 7.5.12 of [BS]} (ii), we have
\begin{align*}
    \|G_J\|_{\cc^{M+1}(\RR^N)}= \|F(0,\cdot)\|_{\cc^{M+1}(\RR^N)}\lesssim \|F\|_{\cc^{\sv, M+1}((W, \dv)\boxtimes \nabla_{\RR^N})}\leq \|F\|_{\cc^{\sv, 2M}((W, \dv)\boxtimes \nabla_{\RR^N})}\lesssim 1.
\end{align*}
That is, $\exists C_{21}\approx 1$ with 
\begin{align}\label{Finishing the proof of main theorem: G_J estimate} 
    \|G_J\|_{\cc^{M+1}}\leq C_{21}.
\end{align}
Set 
\begin{align}
    \widehat{\pp}_{J} w:= dG_{J} (\{(W^{1,0, 2^{-J})^{\alpha} \widehat{v}_{J}(0)}\}_{\deg_{\db^1}(\alpha)\leq \kappa}) \{(W^{1,0,2^{-J}})^{\alpha}w\}_{\deg_{\db^1(\alpha)\leq \kappa}}.
\end{align}
We will show that $\widehat{\pp}_{J}$ is maximally subelliptic of degree $\kappa$ with respect to $(W^{1,0,2^{-J}}, \db^1)$ in a quantitative sense. 

Let $\pp$ be given by \eqref{Introduction: Main theorem: Linearization}. So, $\pp$ is of the form 
\begin{align*}
    \pp = \sum_{\deg_{\db^1}(\alpha)\leq \kappa} a_{\alpha} (W^1)^{\alpha}, 
\end{align*}
where $a_{\alpha} \in \MM^{D_1\times D_2}(\RR)$ are constant matrices; $a_{\alpha}$ are defined by
\begin{align}\label{Finishing the proof of main theorem: coefficients in the linearized equation}
    dG (\{(W^1)^{\alpha}u(0)\})\{\zeta_{\alpha}\}_{\deg_{\db^1}(\alpha)\leq \kappa} = \sum_{\deg_{\db^1}(\alpha)\leq \kappa} a_{\alpha} \zeta_{\alpha}.
\end{align}
Using the fact that $(W^1)^{\alpha}u_{J}(0)+ 2^{-\kappa J} (W^1)^{\alpha}v_{J}(0)=(W^1)^{\alpha} v_j(0)= (W^1)^{\alpha}u(0)$, we have
\begin{align}\label{Finishing the proof of main theorem: linearized equation}
    dG_J (\{(2^{-J\db^1}W^1)^{\alpha} v_J(0)\})\{\zeta_{\alpha}\}_{\deg_{\db^1}(\alpha)\leq \kappa} &= dG_J (\{(W^1)^{\alpha} v_J(0)\})\{2^{-(\kappa-\deg_{\db^1}(\alpha))}\zeta\}_{\deg_{\db^1}(\alpha)\leq \kappa}\nonumber\\
    &= \sum_{\deg_{\db^1}(\alpha)\leq \kappa} 2^{-(\kappa-\deg_{\db^1}(\alpha))} a_{\alpha}\zeta_{\alpha}.
\end{align}
Since $\phi_{0,2^{-J}}(0)=0$ and $\phi_{0,2^{-J}}^* 2^{-J\db^1_l}W_l^1=W_l^{1, 0, 2^{-J}}$ (see Theorem \ref{Preliminaries: Theorem 3.15.5 of [BS]}), we have
\begin{align}\label{Finishing the proof of main theorem: pullback vector field action}
    (W^{1, 0, 2^{-J}})^{\alpha} \widehat{v}_J(0) =(2^{-J\db^1}W^1)^{\alpha} v_{J}(0).
\end{align}
Now, using \eqref{Finishing the proof of main theorem: coefficients in the linearized equation} and \eqref{Finishing the proof of main theorem: pullback vector field action} in \eqref{Finishing the proof of main theorem: linearized equation} we get 
\begin{align*}
    \widehat{\pp}_Jw:&=  dG_J (\{(W^{1,0,2^{-J}})^{\alpha} \widehat{v}_J(0)\})\{(W^{1,0,2^{-J}})^{\alpha}w\}_{\deg_{\db^1}(\alpha)\leq \kappa}\\
    &=dG_J (\{(2^{-J\db^1}W^1)^{\alpha} v_J(0)\})\{(W^{1,0,2^{-J}})^{\alpha}w\}_{\deg_{\db^1}(\alpha)\leq \kappa}\\
    &=\sum_{\deg_{\db^1}(\alpha)\leq \kappa} 2^{-(\kappa-\deg_{\db^1}(\alpha))J}a_{\alpha} (W^{1,0,2^{-J}})^{\alpha}w.
\end{align*}
Thus we have
\begin{align}\label{Finishing the proof of maint theorem: the pullback PDE}
    \Phi_{0,2^{-J}}^{*} 2^{-\kappa J} \pp(\Phi_{0,2^{-J}})_*&= \Phi_{0,2^{-J}}^*  \sum_{\deg_{\db^1}(\alpha)\leq \kappa} 2^{-(\kappa-\deg_{\db^1}(\alpha))J}a_{\alpha} (2^{-J\db^1}W^1)^{\alpha}(\phi_{0,2^{-J}})_*\nonumber\\
    &=\sum_{\deg_{\db^1}(\alpha)\leq \kappa} 2^{-(\kappa-\deg_{\db^1}(\alpha))J}a_{\alpha} (W^{1,0,2^{-J}})^{\alpha} =\widehat{\pp}_{J}.
\end{align}
Multiplying \eqref{Introduction: Main theorem: subelliptic Schauder type estimate} by $2^{-\kappa J}$ gives, $\forall f\in C_{0}^{\infty}(B^{n}(1);\CC^{D_2})$, 
\begin{align}\label{Finishing the proof of main theorem: subelliptic schauder estimate multiplied with 2^{-kappa J}}
    \sum_{j=1}^{r_1}\|(2^{-J\db^1_j}W^{1}_j)^{n_j}f\|_{L^{2}(B^{n}(1), h\sigma_{Leb};\CC^{D_2})} \leq A (\|2^{-\kappa J}\pp f\|_{L^{2}(B^{n}(1), h\sigma_{Leb};\CC^{D_1})}+2^{-\kappa J}\|f\|_{L^{2}(B^{n}(1), h\sigma_{Leb};\CC^{D_2})}).
\end{align}
For $g\in C_{0}^{\infty}(B^n(1);\CC^{D_2})$ we apply \eqref{Finishing the proof of main theorem: subelliptic schauder estimate multiplied with 2^{-kappa J}} to $g\circ \Phi_{0,2^{-J}}^{-1}\in C_{0}^{\infty}(B^n(1);\CC^{D_2})$ and use \eqref{Finishing the proof of maint theorem: the pullback PDE} to  get 
\begin{align}\label{Finshing the proof of the main theorem: subelliptic schauder estimate apllied to pullback of g}
    \sum_{j=1}^{r_1} &\|(W^{1,0,2^{-J}})^{n_j}g\|_{L^2(B^n(1), \Lambda(0,2^{-J})h_{0,2^{-J}}\sigma_{Leb};\CC^{D_2})}\nonumber\\
   & \leq  A (\|\pp_J g\|_{L^2(B^n(1), \Lambda(0,2^{-J})h_{0,2^{-J}}\sigma_{Leb};\CC^{D_1})}+2^{-\kappa J}\|g\|_{L^2(B^n(1), \Lambda(0,2^{-J})h_{0,2^{-J}}\sigma_{Leb};\CC^{D_2})})\nonumber\\
   &\leq A (\|\pp_J g\|_{L^2(B^n(1), \Lambda(0,2^{-J})h_{0,2^{-J}}\sigma_{Leb};\CC^{D_1})}+\|g\|_{L^2(B^n(1), \Lambda(0,2^{-J})h_{0,2^{-J}}\sigma_{Leb};\CC^{D_2})}),
\end{align}
where $\Lambda(0,2^{-J})$ and $h_{0,2^{-J}}$ are as in Theorem \ref{Preliminaries: Theorem 3.15.5 of [BS]}.

Multiplying both sides of \eqref{Finshing the proof of the main theorem: subelliptic schauder estimate apllied to pullback of g} by $\Lambda(0,2^{-J})^{-1/2}$, we obtain 
\begin{align}\label{Finshing the proof of main theorem: hat{pp}_J is maximally subelliptic}
    \sum_{j=1}^{r_1}&\|(W^{1,0,2^{-J}})^{n_j}g\|_{L^{2}(B^{n}(1), h\sigma_{Leb};\CC^{D_2})}  \nonumber\\
    &\leq A (\|\pp_J g\|_{L^{2}(B^{n}(1), h\sigma_{Leb};\CC^{D_1})} +\|g\|_{L^{2}(B^{n}(1), h\sigma_{Leb};\CC^{D_2})}).
\end{align}
That is, $\widehat{\pp}_J$ is maximally subelliptic of degree $\kappa$ with respect to $(W^{1,0,2^{-J}},\db^1)$ with the measure $h_{0,2^{-J}}\sigma_{Leb}$ with the same constant $A$ as in \eqref{Introduction: Main theorem: subelliptic Schauder type estimate}.

Theorem \ref{Preliminaries: Theorem 3.15.5 of [BS]} shows that multi-parameter unit-admissible constants with respect to $(W^{1,0,2^{-J}}, \dv)$ and $h_{0,2^{-J}}$ can be take to be multi-parameter unit-admissible constants with respect to $(W, \dv)$ and $h$. Thus, we may apply Proposition \ref{Reduction I: main prposition of second reduction} with $(W,\dv)$and $h$ replaced by $(W^{0,2^{-J}}, \dv)$ and $h_{0,2^{-J}}$ to find a $(\kappa, \sv, \rv, M, \sigma)$-multi-parameter unit-admissible constant $\epsilon_2=\epsilon_2(C_{21},A,D_1, D_2)\in (0,1]$ as in Proposition \ref{Reduction I: main prposition of second reduction}.

We have already shown that $G_J$ satisfies the assumptions of $G$ in Proposition \ref{Reduction I: main prposition of second reduction}. Thus, we will be able to apply Proposition \ref{Reduction I: main prposition of second reduction} once we show that if $J$ is sufficiently large, then
\begin{align}
    \|\widehat{H}_J\|_{\cc^{\sv, M}((W^{0,2^{-J}}, \dv)\boxtimes \nabla_{\RR^N})}&\leq \epsilon_2 \label{Finishing the proof of the main theorem: required hat{H}_J estimate}\\
     \|\widehat{v}_J\|_{\cc^{\rv+\kappa\ef_1}((W^{0,2^{-J}}, \dv)\boxtimes \nabla_{\RR^N})}&\leq \epsilon_2\label{Finishing the proof of the main theorem: required hat{v}_J estimate}\\
    \|\widehat{h}_J\|_{L^{p}_{\sv}((W^{0,2^{-J}}, \dv)\boxtimes \nabla_{\RR^N})}&\leq \epsilon_2 \label{Finishing the proof of the main theorem: required hat{h}_J lps estimate}\\
    \|\widehat{h}_J\|_{\cc^{\rv}((W^{0,2^{-J}}, \dv)\boxtimes \nabla_{\RR^N})}&\leq\epsilon_2. \label{Finishing the proof of the main theorem: required hat{h}_J cc^{rv} estimate}
\end{align}
Proposition \ref{Setting it up: Proposition 9.2.21 of [BS]} with $c_{\alpha}= 2^{-(\kappa-\deg_{\db^1}(\alpha))}$ shows that
\begin{align}\label{Finishing the proof of the main theorem: hat{H}_J estimate}
\|\hat{H}_{J}&\|_{\cc^{\sv, M}((W^{0,2^{-J}}, \dv)\boxtimes \nabla_{\RR^N})} \nonumber\\
    = & \|\hat{H}_{J}\|_{\cc^{\sv, M+\lfloor|\sv|_1\rfloor+2+\nu-(\lfloor|\sv|\rfloor_1+2+\nu)}((W^{0,2^{-J}}, \dv)\boxtimes \nabla_{\RR^N})}\\
    \lesssim &\|F\|_{\cc^{\sv, 2M}((W^{0,2^{-J}}, \dv)\boxtimes \nabla_{\RR^N})} (1+\|u\|_{\cc^{\rv+\kappa\ef_1}(W, \dv)})^{\lfloor|\sv|_1\rfloor+\nu+2}\nonumber\\
    &\times (2^{-J(\rv_1\wedge 1)}+2^{-J(\min_{\mu}\lambda_{\mu}(s_\mu{\wedge 1}))})\nonumber\\
    \lesssim & (2^{-J(\rv_1\wedge 1)}+2^{-J(\min_{\mu}\lambda_{\mu}(s_\mu{\wedge 1}))}).
    \end{align}
    In the above computation we also used the fact that $\lfloor|\sv|_1 \rfloor+2 +\nu \leq \nu(|\sv|_{\infty}+1)+2< M$. 

    Now, Lemma \ref{Setting it up: Lemma similar to Lemma 9.2.20 in [BS]} (v) shows that 
    \begin{align}
        \|\hat{v}_J\|_{\cc^{\rv+\kappa\ef_1}(W^{0,2^{-J}}, \dv)}&\lesssim 2^{-J\rv_1} \|u\|_{\cc^{\rv+\kappa\ef_1}(W, \dv)}\lesssim 2^{-J\rv_1} \label{Finishing the proof of the main theorem: hat{v}_J estimate}\\
         \|\psi \widebar{h}_{J}\|_{\cc^{\rv}(W^{0,2^{-J}},\dv)} &\lesssim 2^{-J\rv_1}\|g\|_{\cc^{\rv}(W,\dv)}\lesssim 2^{-J\rv_1}\label{Finishing the proof of the main theorem: hat{h}_J cc^{rv} estimate}\\
       \|\psi \widebar{h}_J\|_{L^{p}_{\sv}(W^{0,2^{-J}},\dv)} &\lesssim 2^{-J \sv_1}\|g\|_{L^p_{\sv}(W,\dv)}\lesssim 2^{-J\sv_1}\label{Finishing the proof of the main theorem: hat{h}_J lps estimate}
    \end{align}
    So, \eqref{Finishing the proof of the main theorem: hat{v}_J estimate}, \eqref{Finishing the proof of the main theorem: hat{h}_J cc^{rv} estimate}, \eqref{Finishing the proof of the main theorem: hat{h}_J lps estimate} and \eqref{Finishing the proof of the main theorem: hat{v}_J estimate} applied to \eqref{Finishing the proof of the main theorem: hat{H}_J estimate} shows that \eqref{Finishing the proof of the main theorem: required hat{H}_J estimate}, \eqref{Finishing the proof of the main theorem: required hat{H}_J estimate}, \eqref{Finishing the proof of the main theorem: required hat{h}_J lps estimate}, \eqref{Finishing the proof of the main theorem: required hat{h}_J cc^{rv} estimate} holds if we take $J=J(A, C_0, D_1, D_2)\in \NN$ to be sufficiently large $(\sv, \rv, \kappa, M)-$multi-parameter unit admissible constant. 

    For such a $J$, Proposition \ref{Reduction I: main prposition of second reduction} applies to \eqref{Finishing the proof of the main theorem: scaled PDE equation} to show that for $\psi\in C_0^{\infty}(B^n(2/3))$, $\psi_0\hat{v}_J \in L^{p}_{\sv+\kappa\ef_1}(\widebar{B^{n}(7/8)}, (W^{0,2^{-J}}, \dv))$ and 
    \begin{align}\label{Finishing the proof of the main theorem: hat{v}_J lphk estimate}
        \|\psi_0 \hat{v}_{J}\|_{L^{p}_{\sv+\kappa\ef_1}(\widebar{B^{n}(7/8)}, (W^{0,2^{-J}}, \dv))}\lesssim 1. 
    \end{align}
    Since $\psi\equiv 1$ on a neighbourhood of $\widebar{B^{n}(13/16)}$, we have $\psi_0\psi=\psi_0$ and therefore \eqref{Finishing the proof of the main theorem: hat{v}_J lphk estimate} is equivalent to 
    \begin{align}\label{Finishing the proof of the main theorem: required lpsk estimate for v_{J}}
        \|\psi_0 \Phi_{0,2^{-J}}^* v_J\|_{L^{p}_{\sv+\kappa\ef_1}(W^{0,2^{-J}}, \dv))} \lesssim 1. 
    \end{align}
    Then, if we use Lemma \ref{Setting it up: Lemma similar to Lemma 9.2.20 in [BS]} (iv) with $|\alpha|=0$, $\sv$ replaced by $\sv+\kappa\ef_1$ and $\psi$ replaced by $\psi_0$ shows that
    \begin{align}\label{Finishing the proof of the main theorem: required lpsk estimate for u_{J}}
        \|\psi_0 \Phi_{0,2^{-J}}^* u_J\|_{L^{p}_{\sv+\kappa\ef_1}(W^{0,2^{-J}}, \dv)}\lesssim 1. 
    \end{align}
    Using \eqref{Setting it up: Cutting of higher frequency of u},\eqref{Finishing the proof of the main theorem: required lpsk estimate for v_{J}} and \eqref{Finishing the proof of the main theorem: required lpsk estimate for u_{J}} we get 
    \begin{align}\label{Finishing the proof of the main theorem: lpsk norm of pullback of u estimate}
        \|\psi_0 \Phi_{0,2^{-J}}^* u\|_{L^{p}_{\sv+\kappa\ef_1}(W^{0,2^{-J}}, \dv)}\lesssim 1.
    \end{align}
    By Lemma \ref{Scaling: Vector Fields and Norms: Similar to Lemma 9.2.7 of [BS]}, \eqref{Finishing the proof of the main theorem: lpsk norm of pullback of u estimate} is equivalent to 
    \begin{align}\label{Finishing the proof of the main theorem: Almost the required estimate but with scaled vector field norm}
        \|(\psi_0\circ \Phi_{0,2^{-J}}^{-1}) u\|_{L^p_{\sv+\kappa\ef_1}(2^{-J\lambda}W, \dv)}\lesssim 1.
    \end{align}
    Since $J\approx 1$, Lemma \ref{Scaling: Vector Fields and Norms: Similar to Lemma 9.2.6 of [BS]} shows that \eqref{Finishing the proof of the main theorem: Almost the required estimate but with scaled vector field norm} is equivalent to 
    \begin{align*}
         \|(\psi_0\circ \Phi_{0,2^{-J}}^{-1}) u\|_{L^p_{\sv+\kappa\ef_1}(W, \dv)}\lesssim 1.
    \end{align*}
    Now, taking $\delta=2^{-J}$ finishes the proof Theorem \ref{Introduction: Main Theorem}.
\subsection{Corollaries of multi-parameter theorem}\label{No loss of derivative}
 The following is a stronger version of Theorem \ref{Introduction: Main single parameter theorem with infinitely smooth nonlinearity}. 
\begin{theorem}\label{Introduction: Main single parameter theorem}
     Let $1<p<\infty$, $s>r>0$, $M\geq s+ 3$. Suppose that 
     \begin{itemize}
         \item $F\in \cc^{s, 2M}((W, \db)\boxtimes \nabla_{\RR^N})$ near $x_0$,
         \item $g\in L^p_{s}(W, \db)$ near $x_0$, 
         \item $u\in \cc^{r+\kappa}(W, \db)$ near $x_0$. 
     \end{itemize}
     Then $u\in L^p_{s+\kappa}(W, \db)$ near $x_0$. 
 \end{theorem}
 \begin{proof}
     Apply Theorem \ref{Introduction: Main theorem: qualitative version} for $\nu=1$. 
 \end{proof}
 
  Let us say that we are in the single parameter setting, i.e $\nu=1$. Suppose we are also given two different set of H\"ormander vector fields $(W, \db)=\{(W_1, \db_1, ..., (W_r, \db_r))\}$ and $(Z, \tilde{d})=\{(Z_1, \tilde{d}_1,..., (Z_{\nu}, \tilde{d}_{q})\}$ with formal degrees on $\mathcal{M}$. Then, it is an interesting question to ask about the regularity of a solution to \eqref{Introduction: Main theorem: The PDE} with respect to Sobolev space adapted to $(Z, \tilde{d})$ with apprporiate apriori assumptions. To do this, we use the theory developed in \cite[Section 6]{BS} that compares the Sobolev spaces adapted to $(W, \db)$ and Sobolev spaces adapted to $(Z, \tilde{d})$. 

  First, we need the following definition of a "cost factor" similar to Definition \ref{Introduction: definition of standard cost factor}.
  \begin{definition}\label{Introduction: No loss of derivative: cost factor}(\cite[Definition 6.6.9]{BS})
      For $x\in \mathcal{M}$ and $1\leq j\leq \nu$, define
      \begin{align}
          \lambda(x, (W, \db), (Z_j, \tilde{d}_j)):= \inf \{\lambda'>0: &\exists \ \text{a neighbourhood U of $x$ such that} \ Z_j|_{U}\\\nonumber & \text{is in the}\ C^{\infty}(U) \ \text{module generated by}\ \\\nonumber & \{X_1: (X_1, d_1)\in Gen((W, \db)), d_1\leq \lambda' \tilde{d}_j\}\}.     
      \end{align}
      Then, set
      \begin{align}
          \lambda(x, (W, \db), (Z, \tilde{d})):= \max_{1\leq j \leq \nu} \lambda(x, (W, \db), (Z_j, \tilde{d}_j)). 
      \end{align}
  \end{definition}
  \begin{remark}\label{Finishing the proof of the main theorem: Corollaries of the multi-parameter theorem: remark about the finitesness of the cost factor}
       It can be safely assumed that $\lambda(x, (W, \db), (Z, \tilde{d}))<\infty$ and that the infimum in Definition \ref{Introduction: No loss of derivative: cost factor} is achieved. We refer the reader to \cite[Section 6.6.2]{BS} to read about this and further exposure to the cost factor.
  \end{remark}

 We will first state prove some preliminary lemmas and state some results from \cite{BS} without a proof. 

 The following lemma tells us how to compare Zygmund-H\"older spaces after adding parameters.  
\begin{lemma}\label{No Loss of derivative: Proposition 7.6.1 of BS}(\cite[Proposition 7.6.1]{BS})
    Let $(Y, \hat{d}), (Z, \tilde{d})\subset C^{\infty}(\mathcal{M};T\mathcal{M})\times \NN_+$ be two lists of H\"ormander vector fields with formal degress on $\mathcal{M}$. Suppose $(Y, \hat{d})$ and $ (Z, \tilde{d})$ locally weakly approximately commute on $\mathcal{M}$ and set $(W, \dv);=(Y, \hat{d})\boxtimes (Z, \tilde{d})\subset C^{\infty}(\mathcal{M};T\mathcal{M})\times \NN^2\backslash \{0\}$. Then, $\forall \epsilon>0$, $\exists \epsilon'\in (0,1)$, $\forall s >\epsilon$, $\forall t>0$, 
    \begin{align}
        \cc^{s}(\mathcal{K}, (Z, \tilde{d}))&\subseteq \cc^{(\epsilon', s-\epsilon)}(\mathcal{K}, (W, \dv)),\\
        \cc^{s, t}(\mathcal{K}\times \RR^N, (Z, \tilde{d})\boxtimes \nabla_{\RR^N})&\subseteq \cc^{(\epsilon', s-\epsilon),t} (\mathcal{K}\times \RR^N, (W, \dv)\boxtimes \nabla_{\RR^N}),\label{No loss of derivative: Proposition 7.6.1: equation 1}
    \end{align}
    and $\forall \epsilon>0$, $\forall s, t>0$, 
    \begin{align}
        \cc^{\epsilon, s}(\mathcal{K} (W, \dv))&\subseteq \cc^{s}(\mathcal{K}, (Z ,\tilde{d})), \\
        \cc^{(\epsilon, s), t}(\mathcal{K}\times \RR^N, (W, \dv)\boxtimes \nabla_{\RR^N})&\subseteq \cc^{s, t}(\mathcal{K}\times \RR^N (Z, \tilde{d})\boxtimes \nabla_{\RR^N}),  \label{No loss of derivative: Proposition 7.6.1: equation 2}
    \end{align}
    where the inclusions are continuous. 
\end{lemma}
For our purpose, we also need a lemma similar to to Lemma \ref{No Loss of derivative: Proposition 7.6.1 of BS} for Sobolev spaces. It is possible to prove statements similar to \eqref{No loss of derivative: Proposition 7.6.1: equation 1} and \eqref{No loss of derivative: Proposition 7.6.1: equation 2}. However, we do not pursue it here as it is not necessary for our purpose. 
\begin{lemma}\label{No loss of derivative: Proposition 6.8.2}(\cite[Proposition 6.8.2]{BS})   
For $1<p<\infty$, 
    \begin{align}
        L^p_{0,s}(\mathcal{K}(,W, \dv)) = L^{p}_s(\mathcal{K},(Z, \tilde{d})),
    \end{align}
    with equality in topologies. 
\end{lemma}
\begin{remark}
    The above lemma is one the reasons why expect "no loss of derivative".
\end{remark}
Now, we are ready to describe the sharp regularity theorem with respect to another set of vector fields with formal degrees. The following theorem also subsumes Theorem \ref{Intrroduction: No loss of derivative: Change of Vector field theorem: smooth F and Euclidean vector fields}.
\begin{theorem}\label{Intrroduction: No loss of derivative: Change of Vector field theorem}
      Let $s > r>0$ and $M\geq2s+5$. Let $\lambda:=\lambda(x_0, (W, \db), (Z, \tilde{d}))$ be as in Definition \ref{Introduction: No loss of derivative: cost factor}. Suppose that 
\begin{itemize}
    \item $(W, \db)$ and $(Z, \tilde{d})$ locally weakly approximately commute on $\mathcal{M}$,
    \item $F\in \cc^{s, 2M}((Z, \tilde{d})\boxtimes \nabla_{\RR^N})$ near $x_0$,
    \item $g\in L^p_{s}(Z, \tilde{d})$ near $x_0$,
    \item $W^{\alpha} u\in \cc^{r}(Z, \tilde{d})$ near $x_0$, $\forall \deg_{\db}(\alpha)\leq \kappa$.
\end{itemize}
Then, we have
\begin{enumerate}[label=(\roman*)]
    \item $u\in L^p_{s+\kappa/\lambda}(Z, \tilde{d})$ near $x_0$, and
    \item $W^{\alpha} u\in L^p_{s}(Z, \tilde{d})$ near $x_0$, $\forall \deg_{\db^1}(\alpha)\leq \kappa$. 
\end{enumerate}
  \end{theorem}
\begin{proof}
    We rename $(W, \db)$ to be $(W^1, \db^1)$ and set $(W, \db^2):=(Z, \tilde{d})$ and let $(W, \dv):=(W^1, \db^1)\boxtimes (W^2, \db^2)$. By assumptions of Theorem \ref{Intrroduction: No loss of derivative: Change of Vector field theorem}, we may pick $\psi_1, \psi_2 \in C_0^{\infty}(\mathcal{M})$, with $\psi_1\prec \psi_2$, with $\psi_1\equiv 1$ on a neighbourhood of $x_0$, and with small enough support that if $\mathcal{K}:= supp(\psi_2)$, we have 
    \begin{align}
        \psi_{1}(x)F(x, \zeta)&\in \cc^{s, 2M}(\mathcal{K}\times \RR^N, (W^2, \db^2)\boxtimes \nabla_{\RR^N}), \label{No loss of derivative: Main Theorem: equation 1}\\
       \psi_1 g &\in L^p_s(\mathcal{K}, (W^2, \db^2)),  \label{No loss of derivative: Main Theorem: equation 2}\\
       \psi_2(W^1)^{\alpha} u &\in \cc^{r}(\mathcal{K}, (W^2, \db^2)), \ \forall \ \deg_{\db^1}(\alpha)  \leq \kappa\label{No loss of derivative: Equation 9.9 of BS}. 
     \end{align}
     Note that for $\deg_{\db^1}(\alpha)\leq \kappa$, we have 
     \begin{align*}
         (W^1)^{\alpha}\psi_1 u = \psi_2 (W^1)^{\alpha} \psi_1 u=  \sum_{\deg_{\db^1}(\beta)\leq \kappa} g_{\alpha}^{\beta} \psi_2(W^1)^{\beta}u, 
     \end{align*}
     where $g_{\alpha}^{\beta}\in C_{0}^{\infty}(\mathcal{M})$. Therefore, Corollary \ref{Preliminaries: Besov and Triebel-Lizorkin space: Corollary 6.5.10 of BS} combined with \eqref{No loss of derivative: Equation 9.9 of BS} implies that 
     \begin{align}
         (W^1)^{\alpha}\psi_1 u \in \cc^{r} (\mathcal{K}, (W^2, \db^2)), \ \forall \ \deg_{\db^1} (\alpha)\leq \kappa.\label{No loss of derivative: Main theorem: equation 4}
     \end{align}
     Let $\epsilon$ be a number in $(0,r)$. Then, Lemma \ref{No Loss of derivative: Proposition 7.6.1 of BS} applied to \eqref{No loss of derivative: Main Theorem: equation 1} and \eqref{No loss of derivative: Main theorem: equation 4} shows that there exists $\epsilon'>0$ with 
     \begin{align}
         \psi_1(x) F(x, \zeta) &\in \cc^{(\epsilon',s-\epsilon), 2M} (\mathcal{K}\times \RR^N, (W, \dv)\boxtimes \nabla_{\RR^N}), \label{No loss of derivative: Maint theorem: equation 5} \\
         (W^1)^{\alpha} \psi_1 u &\in \cc^{(\epsilon', r-\epsilon)} (\mathcal{K}, (W, \dv)), \quad \forall \ \deg_{\db^1}(\alpha) \leq \kappa. \label{No loss of derivative: Maint theorem: equation 6}
     \end{align}
     Now, we get that $\psi_1 u \in \cc^{\epsilon'+\kappa, r-\epsilon}(\mathcal{K}, (W, \dv))$ as consequnce of \eqref{No loss of derivative: Maint theorem: equation 6} and\cite[Proposition 6.5.12]{BS}.  

     We would like to apply Theorem \ref{Introduction: Main theorem: qualitative version} now. Since $|s|>|\max\{|\epsilon'|, |s-\epsilon|\}|$ we see that  $M\geq 2s+5 > 2(\max{|s-\epsilon|, |\epsilon'|}+1)+3$. Also, since $\psi_1\equiv 1$ on a neighbourhood of $x_0$, the above shows that Theorem \ref{Introduction: Main theorem: qualitative version} applies to give that there exists $\tilde{\phi}\in C_0^{\infty}(\mathcal{M})$, with $\tilde{\phi}\equiv 1$ on a neighbourhood of $x_0$, such that 
     \begin{align}
         \tilde{\phi}u\in L^p_{(\kappa,s)}(supp(\tilde{\phi}), (W, \dv)). \label{No loss of derivative: Main theorem: equation 7}
     \end{align}
     Using these observations, we will first prove (ii). Take $\phi\in C_0^{\infty}(\mathcal{M})$ with $\phi\prec \tilde{\phi}$ and $\phi\equiv 1$ on a neighbourhood of $x_0$. Using, \eqref{No loss of derivative: Main theorem: equation 7}, Corollary \ref{Preliminaries: Besov and Triebel-Lizorkin space: Corollary 6.5.10 of BS} we get that for all $\deg_{\db^1}\leq \kappa$, 
     \begin{align*}
         \phi (W^1)^{\alpha}u \in L^p_{(\kappa,s)}(supp(\phi), (W, \dv)) \subseteq L^p_{s}(supp(\phi), (W^2, \db^2)).
     \end{align*}
     Next, we will prove (i). Let $\lambda:= \lambda(x_0, (W^1, \db^1), (W^2, \db^2))>0$ be as in Definition \ref{Introduction: No loss of derivative: cost factor}. Using Remark \ref{Finishing the proof of the main theorem: Corollaries of the multi-parameter theorem: remark about the finitesness of the cost factor} and the fact that $(W^2, \db^2)$ is a finite set we see that there exists a neighbourhood $\Omega'\Subset \mathcal{M}$ of $x$ such that for $1\leq j\leq \nu, W^2_j|_{\Omega'}$ is in the $C^{\infty}(\Omega')$ module generated by $\{X_1: (X_1, d_1)\in Gen ((W^1, \db^1)), d_1\leq \lambda \db^2_j\}$.  Take $\phi\in C_0^{\infty}(\Omega')$ with $\phi\prec \phi'$ and $\phi\equiv 1$ on a neighbourhood of $x_0$ and set $\mathcal{K}_1:=supp(\phi) \Subset \Omega'$. Now, we combine Corollary \ref{Preliminaries: Besov and Triebel-Lizorkin space: Corollary 6.5.10 of BS} and \eqref{No loss of derivative: Main theorem: equation 7} to get 
     \begin{align*}
         \phi u = \phi \tilde{\phi}\in L^{p}_{(\kappa, s)}(\mathcal{K}_1, (W, \dv)). 
     \end{align*}
    Now, we use the result(\cite[Theorem 6.7.2]{BS}) that lets us trade derivatives. We do not state the result in full generality.It tells us that we have the continuous inclusion
    \begin{align*}
         L^{p}_{(\kappa, s)}(\mathcal{K}_1, (W, \dv)) \subseteq L^p_{(\kappa, s)+ (-\kappa, \kappa/ \lambda)} (\mathcal{K}_1, (W, \dv)).
    \end{align*}
   i.e, we trade $\kappa$ derivatives in $(W^1, \db^1)$ for $\kappa/\lambda$ derivatives in $(W^2, \db^2)$. This result is sharp and cannot be improved beyond the cost factor $\lambda$. Then,
    \begin{align*}
        \phi u \in L^{p}_{(\kappa, s)}(\mathcal{K}_1, (W, \dv)) \subseteq L^p_{(\kappa, s)+ (-\kappa, \kappa/ \lambda)} (\mathcal{K}_1, (W, \dv)) = L^p_{(0, s+\kappa/\lambda)} (\mathcal{K}_1, (W, \dv)). 
    \end{align*}
    Now, we can apply Lemma \ref{No loss of derivative: Proposition 6.8.2} to get 
    \begin{align*}
        \phi u \in  L^p_{(0, s+\kappa/\lambda)} (\mathcal{K}_1, (W, \dv)) = L^p_{s+\kappa/\lambda}(\mathcal{K}_1, (W^2, \db^2)), 
    \end{align*}
    which finishes the proof. 
\end{proof}
As we mentioned in Remark \ref{Introduction: remark about Besov and Triebel Lizorkin space theorem}, it is possible to prove sharp interior regularity theorem for other adapted Besov and Triebel-Lizorkin spaces too. We state a version of such a theorem so that an interested reader can prove it. Let $\mathfrak{X}^{s}$ denote any Besov and Triebel-Lizorkin spaces corresponding to the single-parameter H\"ormander vector fields with formal degrees $(W, \db)$. Let $u$ be a solution to the following PDE
\begin{align}\label{Corollaries of mult-parameter theorem: general partial differential equation}
    F(x, \{W^{\alpha}u(x)\}_{\deg_{\db}(\alpha)\leq \kappa})=g(x),
\end{align}
such that for a fixed $x_0\in \mathcal{M}$ and $u: \mathcal{M}\to \RR^{D_2}$, the linearized operator $\mathcal{P}_{u,x_0}$ as defined in \eqref{Introduction: linearized operator} is a maximally subelliptic operator.
Then, we have the following theorem.
     \begin{theorem}\label{Preliminaroes: Function spaces: Theorem about Besov and Triebel Lizorkin space}
        Let $s>r>0$ and suppose that 
     \begin{itemize}
         \item $F\in C^{\infty}(\mathcal{M}\times \RR^N)$ near $x_0$,
         \item $g\in \mathfrak{X}^s(W, \db)$ near $x_0$, 
         \item $u\in \cc^{r+\kappa}(W, \db)$ near $x_0$. 
     \end{itemize}
     Then $u\in \mathfrak{X}^{s+\kappa}(W, \db)$ near $x_0$.  
     \end{theorem}
     One can also prove the following theorem regarding the regularity of solutions in the standard Besov and Triebel-Lizorkin spaces. 
      \begin{theorem}\label{Preliminaries: loss of derivative: Change of Vector field theorem: smooth F and Euclidean vector fields for Triebel Lizorkin space}
      Let $s > r>0$. Let $\lambda_{std}:= \lambda_{std}(x_0, (W, \db))$. Suppose that 
\begin{itemize}
    \item $F\in \cc^{\infty}(\mathcal{M}\times \RR^N)$ near $x_0$,
    \item $g\in \mathfrak{X}^s_{std}$ near $x_0$,
    \item $W^{\alpha} u\in \cc^{r}_{std}$ near $x_0$, $\forall \deg_{\db}(\alpha)\leq \kappa$.
\end{itemize}
Then, $\forall \epsilon>0$ we have
\begin{enumerate}[label=(\roman*)]
    \item $u\in \mathfrak{X}^{s+\kappa/\lambda_{std}-\epsilon}_{std}$ near $x_0$, and
    \item $W^{\alpha} u\in \mathfrak{X}^{s-\epsilon}_{std}$ near $x_0$, $\forall \deg_{\db}(\alpha)\leq \kappa$. 
\end{enumerate}
  \end{theorem}
  \begin{remark}
      One could also state Theorem \ref{Preliminaries: loss of derivative: Change of Vector field theorem: smooth F and Euclidean vector fields for Triebel Lizorkin space} for function spaces adapted to other vector fields as in Theorem \ref{Intrroduction: No loss of derivative: Change of Vector field theorem}.
  \end{remark}
  \begin{remark}
      Observe that Theorem \ref{Preliminaries: loss of derivative: Change of Vector field theorem: smooth F and Euclidean vector fields for Triebel Lizorkin space} comes with an $\epsilon$-loss in regularity. This loss does not seem fixable using our technique; also see \cite[Proposition 9.1.6, Corollary 9.1.8]{BS}. 
  \end{remark}
\appendix
\section{}\label{Appendix A}
    This section is dedicated to linear maximally subelliptic PDE. The results in this section can also be stated in single parameter setting, however we state them in the multi-parameter setting for the ease making references in Section \ref{Reduction II}. 
    
    \textbf{Setup}. Fix $D_1, D_2\in \NN_+$. Let $(W, \dv)\in C^{\infty}(\mathcal{M}; T\mathcal{M})\times \NN_+$ be as in Section \ref{Main multi-parameter theorem}. Fix $\kappa\in \NN_+$ such that $\db_j^1$ divides $\kappa$ for $1\leq j\leq r$ and set $n_j:=\kappa/\db_j^1\in \NN_+$.  We are interested in partial differential operators of the form 
    \begin{align}\label{Appendix: Linear maximally subelliptic operator}
        \pp:= \sum_{\deg_{\db^1}(\alpha)\leq \kappa} a_{\alpha}(x) (W^1)^{\alpha}, \quad a_{\alpha}\in C^{\infty}(\mathcal{M}; \MM^{D_1\times D_2}(\CC)),
    \end{align}
     where $\MM^{D_1\times D_2}$ is the space of $D_1\times D_2$ matrices with entries in $\CC$. We also assume that $\pp$ is maximally subelliptic of degree $\kappa$ with respect to $(W^1, \db^1)$, i.e., for every relatively compact, open set $\Omega \Subset M$ there exists $C_{\Omega}\geq 0$ satisfying
 \begin{align}
     \sum_{j=1}^{k} \|(W_{j}^1)^{n_j}f\|_{L^2(\mathcal{M},Vol; \CC^{D_2})} \leq C_{\Omega} (\|\pp f\|_{L^2(\mathcal{M},Vol; \CC^{D_1})}+\|f\|_{L^2(\mathcal{M},Vol; \CC^{D_2})}), 
 \end{align}
for every $f\in C_{0}^{\infty}(\Omega; \CC^{D_2})$, where $n_j=\kappa/\db_j^1\in \NN_+$.
 
\begin{theorem}\label{Preliminaries: Theorem 8.1.1}(\cite{BS}, Theorem 8.1.1)
  Let $\mathcal{P}$ of the form \eqref{Appendix: Linear maximally subelliptic operator}. Then the following are equivalent
    \begin{enumerate}[label=(\roman*)]
        \item $\mathcal{P}$ is maximally subelliptic of degree $\kappa$ with respect to $(W^1,\db^1)$ on $\mathcal{M}$.
        \item $\mathcal{P}_0:= \sum_{\deg_{\db^1}(\alpha)=\kappa} a_{\alpha}(x)(W^1)^{\alpha}$ is maximally subelliptic with respect to $(W^1,\db^1)$ on $\mathcal{M}$.
        \item $\forall x_0\in \mathcal{M}$, there exists an open neighbourhood $U\subseteq \mathcal{M}$ of $x_0$ such that $\mathcal{P}$ is maximally subelliptic of degree $\kappa$ with respect to $(W^1,\db^1)$ on $U$.
        \item $\forall x_0\in \mathcal{M}$, there exists an open neighborhood $U\subseteq \mathcal{M}$ of $x_0$ with frozen coefficient operator $\sum_{\deg_{\db^1}(\alpha)\leq \kappa} a_{\alpha}(x_0) (W^1)^{\alpha}$ is maximallay subelliptic fo degree $\kappa$ with respect to $(W^1,\db^1)$ on $U$.
        \item For any scale of spaces $\mathcal{X}^{s}$ of the form 
       \begin{align*}
        \mathcal{X}^{s}\in \{\mathcal{B}^{s}_{p,q}\in [1,\infty]\} \bigcup \{\mathcal{F}^{s}_{p,q}: p\in (1,\infty), (1,\infty]\}\},
       \end{align*}
        we have the following. Let $\phi_1,\phi_2,\phi_3 \in C_{0}^{\infty}(\mathcal{M})$ with $\phi_1\prec \phi_2\prec\phi_3$. Then, $\forall s\in \RR$, 
        \begin{align*}
           & \phi_{3}\mathcal{P}u\in \bigcup_{\mathcal{K}}\mathcal{X}^{s}(\mathcal{K},(W^1,\db^1);\CC^{D_1})\implies \phi_{1}u\in \bigcup_{\mathcal{K}}\mathcal{X}^{s+\kappa} (\mathcal{K},(W^1,\db^1);\CC^{D_2}),\\
            &\forall u\in C_{0}^{\infty}(\mathcal{M})',
        \end{align*}
       where the union is taken over all compact set $\mathcal{K}\subset \mathcal{M}$. 
        Moreover, for every $N\geq 0$ there exists $C=C(s, N, \mathcal{X}^s, (W^1,\db^1), \phi_1,\phi_2,\phi_3)\geq0$ such that \begin{align*}
           & \|\psi_1 u\|_{\mathcal{X}^{s+\kappa}((W^1,\db^1);\CC^{D_2})} \leq C (\|\psi_3 \mathcal{P}u\|_{\mathcal{X}^{s} ((W^1,\db^1);\CC^{D_1})}+ \|\psi_2 u\|_{\mathcal{X}^{s-N} ((W^1,\db^1);\CC^{D_2})}),\\
           &\forall u\in C_{0}^{\infty}(\mathcal{M};\CC^{D_2})' \ \text{with}\ \phi_3\mathcal{P}u \in \bigcup_{\mathcal{K}}\mathcal{X}^{s}(\mathcal{K},(W^1,\db^1);\CC^{D_1}).
        \end{align*} 
        \item $\pp^*\pp$ is maximally subelliptic of degree $2\kappa$ with respect to $(W^1,\db^1)$ on $\mathcal{M}$. 
        \item There exists $T\in \mathcal{A}_{loc}^{-\kappa} ((W^1,\db^1); \CC^{D_2},\CC^{D_1})$  such that 
        \begin{align*}
            T\mathcal{P}\equiv I\ \mod C^{\infty}(\mathcal{M}\times \mathcal{M} ; \MM^{D_2\times D_1}(\CC)). 
        \end{align*}
        Moreover, $T\in \mathcal{A}_{loc}^{-\kappa\ef_1} ((W,\dv); \CC^{D_2},\CC^{D_1})$ (see Theorem \cite[5.11.18]{BS}), where $\ef_1=(1,0,..,0)\in\NN^{\nu}$. 
        \item There exists $S\in \mathcal{A}_{loc}^{-2\kappa} ((W,\db);\CC^{D_2},\CC^{D_1})$ such that
        \begin{align*}
            S\pp ^*\pp, \pp^*\pp \equiv I \ \mod C^{\infty}(\mathcal{M}\times \mathcal{M}; \MM^{D_2\times D_2}(\CC)).
        \end{align*}
        Moreover, $S\in \mathcal{A}_{loc}^{-2\kappa\ef_1} ((W,\dv); \CC^{D_2},\CC^{D_1})$ as in (vii). 
     \end{enumerate}
\end{theorem}
\begin{remark}
     \cite{BS} has further characterization of Theorem \ref{Preliminaries: Theorem 8.1.1} using the kernel for a self-adjoint extension of $\pp^*\pp$. But, we do not mention it here as we do not require it for the purpose of our result. 
\end{remark}

\end{document}